\documentclass[a4paper,10pt]{amsart}

\usepackage{a4}
\usepackage{graphics,graphicx}
\usepackage{amssymb,amsmath}
\usepackage[all]{xy}
\usepackage{longtable}
\usepackage{multirow}
\usepackage{showlabels}
\CompileMatrices
\usepackage{hyperref}

\usepackage{tikz}
\usetikzlibrary{calc}
\usetikzlibrary{decorations.markings}
\usetikzlibrary{arrows,shapes,matrix}
\usetikzlibrary{positioning,shapes,shadows,arrows}
\usepackage{tikz-3dplot}

\usepackage{dpfloat}

\newtheorem{theorem}{Theorem}[section]
\newtheorem{lemma}[theorem]{Lemma}

\theoremstyle{definition}

\newtheorem{example}[theorem]{Example}
\newtheorem{remark}[theorem]{Remark}

\newtheorem{step}[theorem]{Step}

\theoremstyle{remark}

\def\tt{\ensuremath{\mathfrak{t}}}
\def\TT{\mathbb{T}}
\def\KK{\mathbb{K}}
\def\ZZ{\mathbb{Z}}

\def\PP{\mathbb{P}}
\def\QQ{\mathbb{Q}}

\def\OO{\mathcal O}

\def\<{\langle}
\def\>{\rangle}

\def\cox{\mathcal{R}}

\newcommand{\cone}[1]{\mathrm{cone}(#1)}

\makeatletter
\newcommand{\thickhline}{%
    \noalign {\ifnum 0=`}\fi \hrule height 1pt
    \futurelet \reserved@a \@xhline
}
\makeatother

\renewcommand{\phi}{\varphi}

\def\div{{\rm div}}

\def\reg{{\rm reg}}
\def\rq#1{\widehat{#1}}

\def\b#1{\overline{#1}}

\def\KK{{\mathbb K}}
\def\TT{{\mathbb T}}
\def\ZZ{{\mathbb Z}}

\def\QQ{{\mathbb Q}}
\def\PP{{\mathbb P}}

\def\Cox{\cox}

\def\WDiv{\operatorname{WDiv}}

\def\Cl{\operatorname{Cl}}

\def\Pic{\operatorname{Pic}}

\def\cone{{\rm cone}}

\def\rank{\operatorname{rank}}

\def\ZZZ{\ZZ_{\geq 0}}

\def\KT#1{\KK[T_1,\ldots,T_{#1}]}

\def\tt#1{\texttt{#1}}

\usepackage{lmodern}
\makeatletter
\ifcase \@ptsize \relax% 10pt
  \newcommand{\miniscule}{\@setfontsize\miniscule{3}{5}}% \tiny: 5/6
\fi
\makeatother

\makeindex

\author[U.~Derenthal, J.~Hausen, A.~Heim, S.~Keicher and A.~Laface]{Ulrich~Derenthal, J\"urgen~Hausen, Armand~Heim, \\ Simon~Keicher and Antonio~Laface}

 \address{Institut f\"ur Algebra, Zahlentheorie und Diskrete
  Mathematik, Leibniz Universit\"at Hannover,
  Welfengarten 1, 30167 Hannover, Germany}
\email{derenthal@math.uni-hannover.de}

 \address{Mathematisches Institut, Universit\"at T\"ubingen,
Auf der Morgenstelle 10, 72076 T\"ubingen, Germany}
\email{juergen.hausen@uni-tuebingen.de}

\address{Mathematisches Institut, Universit\"at T\"ubingen,
Auf der Morgenstelle 10, 72076 T\"ubingen, Germany}
\email{armand-martin.heim@student.uni-tuebingen.de}

\address{Mathematisches Institut, Universit\"at T\"ubingen,
Auf der Morgenstelle 10, 72076 T\"ubingen, Germany}
\email{simon.keicher@uni-tuebingen.de}

\address{Departamento de Matem\'atica, Universidad de Concepci\'on,
Casilla 160-C, Concepci\'on, Chile}
\email{alaface@udec.cl}

\title[Cox rings of cubic surfaces and Fano threefolds]{Cox rings of cubic surfaces \\ and Fano threefolds} 

\subjclass[2010]{
14L24, 14L30, 14C20, 14Q15\\
The first, second and fourth author were supported 
by the DFG Priority Program SPP 1489.
The fifth author was partially supported 
by Proyecto FONDECYT Regular N. 1110096.
The fourth and fifth author worked jointly on part of 
this manuscript in Concepci\'on 
as funded by the DAAD (project ID: 57055392)
and the Conicyt (project PCCI13005).
}

\begin{document}

\begin{abstract}
We determine the Cox rings of the minimal resolutions 
of cubic surfaces with at most rational double points, 
of blow-ups of the projective plane at non-general 
configurations of six points and of three dimensional 
smooth Fano varieties of Picard numbers one and two. 
\end{abstract}

\maketitle

\section{Introduction}

The homogeneous coordinate ring of a toric 
variety as introduced by D.~Cox~\cite{cox} 
quickly became an important tool in toric geometry.
Starting with the well known work~\cite{HuKe}
by Hu and Keel, this ring was also considered in 
a more general context: 
to any normal complete variety $X$ with finitely 
generated divisor class group $\Cl(X)$ one 
associates its {\em Cox ring\/}
$$ 
\mathcal{R}(X)
\ = \ 
\bigoplus_{\Cl(X)} \Gamma(X,\mathcal{O}_X(D)),
$$
where we refer to~\cite{ArDeHaLa} for the precise 
formulation of this definition and basic background.
The Cox ring is a rich invariant of $X$.
In case of finite generation, it gives even
rise to an explicit encoding of $X$, see~\cite{BeHa, Ha2}
Among other things, this opens a computational approach
to the geometry of $X$ once its Cox ring $\mathcal{R}(X)$
is known; see~\cite{HaKe}, for example.

The computation of Cox rings is an active field, see for
example~\cite{CaTe, Ca, LaVe, StuXu}. The aim of this
paper is to enhance the pool of known Cox rings for some classical
classes of varieties; we hope that these are particularly interesting
both in themselves and for arithmetic applications. We work over an
algebraically closed field $\KK$ of characteristic $0$.

One arithmetic motivation is that determining generators and relations
of the Cox rings of varieties gives an explicit description of their
universal torsors. These can be used to parameterize rational points
on varieties, leading for example to proofs of Manin's conjecture
\cite{FMT, BM} on the asymptotic behavior of the number of rational
points of bounded height on Fano varieties. In particular, this
approach was very successful for toric varieties \cite{Salberger}
(where the Cox rings are polynomial rings \cite{cox}), smooth quintic
del Pezzo surfaces \cite{dlB} (see \cite{Sk0} for the universal torsor
and \cite{BaPo} for the Cox ring) and many singular quartic del Pezzo
surfaces (based on the computation of Cox rings in \cite{DeThesis,Der,HaKeLa}).

Cubic surfaces are probably the rational surfaces that have received
the most attention. Smooth and singular cubic surfaces were classified
by Schl\"afli and Cayley in the 1860s. See \cite{Manin,Do} for a
modern account of their geometry. The Cox rings of smooth cubic
surfaces were determined in \cite{BaPo, StiTeVe}. For cubic surfaces
with rational double points as singularities, Cox rings were
determined so far only in the cases where there is at most one
relation in the Cox ring \cite{HaTs,Der}; this includes the toric
cubic surface of singularity type $3A_2$ and seven other types; for
several of them, Manin's conjecture is known, see the table in
Section~\ref{sec:surfaces}. In Section~\ref{sec:surfaces}, we complete
the list of the Cox rings for minimal desingularizations of such singular
cubic surfaces, see
Theorem~\ref{thm:cubic}. This should be a useful step in the further
investigation of Manin's conjecture for singular cubic surfaces via
the universal torsor method. Continuing~\cite{HaKeLa}, we also provide
the Cox rings of the blow-ups of $\PP_2$ in six distinct points in
non-general position.

While Cox rings of surfaces have been widely studied, higher
dimensional results are scarce. In dimension $3$, to our knowledge, we
only have explicit information in toric cases, for varieties
with a torus action of complexity $1$ \cite{HaHeSu} and for some
blow-ups of $\PP_3$ and $(\PP_1)^3$ \cite{CaTe, Ottem, StuXu, StuVe, HaKeLa,
  Baeker}. The three-dimensional analogs of del Pezzo surfaces are
Fano threefolds. Smooth Fano threefolds have been classified by
Iskovskikh~\cite{Is}, Mori and Mukai~\cite{MoMu}. We compute Cox rings
in the cases of Picard number one, see Theorem~\ref{thm:fano1}, and
two, see Theorem~\ref{thm:fano2}.  
This should provide the foundation of a systematic investigation of
Manin's conjecture in dimension $3$ via the universal torsor method;
see \cite{dlBSegre} for a successful application to Segre's
\emph{singular} cubic threefold.

In our computations we make use of the methods based on toric ambient
modifications developed in~\cite{Ha2,BaHaKe,HaKeLa}. In the case of
our singular cubic surfaces, this is relatively straightforward since
their minimal desingularizations are blow-ups of $\PP_2$ in six points
in \emph{almost general position}; a slight complication arises from
the fact that iterated blow-ups of points on exceptional divisors are
allowed. For the smooth Fano threefolds, the situation is much more
involved. In \cite{AG5}, each type is typically described as a double
cover of a complete intersection with prescribed branch curve or
branch divisor, or as the blow-up of a complete intersection in a
subvariety. For each type, we apply a suitable combination of existing
and new theoretical results (e.g., Lemma~\ref{lem:doublecover} on the
Cox rings of double covers of smooth Fano varieties) with
algorithmic methods. For the case of complete intersection Cox rings,
we determine the Cox rings for the whole family by applying our
algorithms formally, without the help of a computer. In other cases,
we compute the Cox rings for a typical
representative of the family.

We provide an implementation of algorithms from~\cite{HaKeLa} in the
computer algebra system \texttt{Singular}~\cite{singular}.  The
package is presented in Section~\ref{sec:example}. Its features
comprise verifying finite generation, verifying a guess of generators,
producing a guess of generators and computing relations between
generators for Cox rings.

We are grateful to the anonymous referees for valuable suggestions and remarks.
Moreover, we would like to thank Hendrik S\"u{\ss} for helpful comments.

\tableofcontents

\section{Computational preparations}
\label{sec:example}

We recall the necessary background on Cox rings,
the technique of toric ambient modifications
introduced in~\cite{Ha2} and the computational 
approach presented in~\cite{HaKeLa}.
Moreover, we perform example computations 
with our package \tt{compcox.lib}
indicating the use of its central functions;
for the full description, we refer to the 
manual available under~\cite{DeHaHeKeLa}.
The aim of~\tt{compcox.lib} is efficient
Cox ring computation. For a general package
for working with the geometry of Mori dream 
space we refer to~\cite{HaKe}.

First consider any normal complete variety $X$
with finitely generated divisor class group 
$\Cl(X)$. The Cox ring of $X$ is
$$ 
\mathcal{R}(X)
\ = \ 
\bigoplus_{\Cl(X)} \Gamma(X,\mathcal{O}_X(D)),
$$
where this definition is straightforward in the 
case of a torsion free divisor class group
and requires some care if torsion occurs;
see~\cite[Sec.~1.4.2]{ArDeHaLa}.
An important feature of the Cox ring is
its divisibility theory: $\mathcal{R}(X)$
is $\Cl(X)$-factorial in the sense that 
we have unique factorization in the 
multiplicative monoid of non-zero homogeneous 
elements, see~\cite[Sec.~1.5.3]{ArDeHaLa}.

If $\mathcal{R}(X)$ is finitely generated
and $X$ projective, then $X$ is called a
\emph{Mori dream space}.
In this setting, the choice of suitable $\Cl(X)$-homogeneous 
generators for the Cox ring gives rise to 
certain embeddings $X \subseteq Z$ into projective
toric varieties $Z$.
The idea is to work with such embeddings 
and to study the effect of a modification of the
ambient toric variety on the Cox ring of $X$.
More precisely, the toric variety $Z$ arises from
a fan $\Sigma$ in a lattice $N$. Recall
from~\cite{cox},
and also~\cite[Sections~2.1.2.--2.1.4]{ArDeHaLa},
that there are exact sequences
$$
\xymatrix{
0
\ar[r]
&
L
\ar[r]
&
{\ZZ^r}
\ar[r]^{P}
&
N,
&
\\
0
\ar@{<-}[r]
&
K
\ar@{<-}[r]_{Q}
&
{\ZZ^r}
\ar@{<-}[r]_{P^*}
&
M
\ar@{<-}[r]
&
0,
}
$$
where $r$ is the number of rays of $\Sigma$,
the linear map $P \colon \ZZ^r \to N$ sends the $i$-th
canonical basis vector to the primitive generator
of the $i$-th ray and $P^*$ denotes the dual map
of $P$.
The abelian group $K$ equals the divisor class
group $\Cl(Z)$.
The Cox ring of $Z$ is the polynomial ring 
$\KK[T_1,\ldots,T_r]$ with the $K$-grading 
assigning to the variable $T_i$ the image $Q(e_i)$
as its $K$-degree, where $e_i \in \ZZ^r$ 
denotes the $i$-th canonical basis vector.
Moreover, we obtain a fan
$$
\rq{\Sigma}
\ := \ 
\{
\rq{\sigma} \preceq \QQ^{r}_{\ge 0}; \;
P(\rq{\sigma}) \subseteq \tau \text{ for some } \tau \in \Sigma
\},
$$
where ``$\preceq$'' denotes the face relation and
we regard $P$ also as a map of the
corresponding rational vector spaces.
The fan $\rq{\Sigma}$ defines an open toric subvariety
$\rq{Z} \subseteq \KK^{r}$ and $P$ defines 
a toric morphism $p \colon \rq{Z} \to Z$.
Assume that $X \subseteq Z$ meets the big torus orbit 
$\TT \subseteq Z$, denote by $\b{X} \subseteq \KK^{r}$ 
the closure of $p^{-1}(X \cap \TT)$, 
by $I \subseteq \KK[T_1,\ldots,T_r]$
the vanishing ideal of $\b{X}$ and consider
the $K$-graded factor ring
$$
R \ := \ \KK[T_1,\ldots,T_r]/I.
$$
We call $X \subseteq Z$ a
\emph{compatibly embedded Mori dream space (CEMDS)}
if the $K$-graded ring $R$ is the Cox ring of $X$
and the variables $T_i$ define pairwise non-associated
$K$-prime elements in $R$.
We encode $X \subseteq Z$ by the triple 
$(P,\Sigma,G)$, where~$G \subseteq \KK[T_1,\ldots,T_r]$ 
is a system of $K$-homogeneous generators of the 
ideal~$I$.
We then also speak of~$(P,\Sigma,G)$ as a CEMDS.
If $G$ provides generators just for the vanishing 
ideal of $p^{-1}(X\cap \TT)$ in
$\KK[T_1^{\pm 1},\ldots,T_r^{\pm 1}]$, then we speak 
of a {\em weak CEMDS\/}.
Note that $X \subseteq Z$ and the Cox ring $R$
of $X$ can be directly reconstructed from $(P,\Sigma,G)$.
Let us see how to declare these data in our package.

\begin{example}[\tt{createCEMDS}]
\label{ex:createCEMDS0}
The $E_6$-singular cubic surface 
$X = V(z_1z_2^2+z_2z_0^2+z_3^2) \subseteq \PP_3$ 
has divisor class group $\Cl(X)=\ZZ$ and its Cox ring is 
given in terms of generators and relations by
\[
R
\ =\ 
\KT{4}/\<T_1^3T_2 + T_3^3 + T_4^2\>
,\qquad
Q\,=\,
  \left[
  \mbox{\tiny $
  \begin{array}{rrrr}
  1 & 3 & 2 & 3
  \end{array}
$}
\right]
\]
where the $i$-th column of $Q$ is the $\Cl(X)$-degree of 
the variable $T_i$; see e.g.~\cite[Ex.~V.4.3.5]{ArDeHaLa}.
The corresponding matrix $P$ and the fan $\Sigma$ 
of an ambient toric variety $Z$ according to the 
above setting are given by
\begin{center}
 \begin{minipage}{5cm}
  \[
P\,=\,
\left[
\mbox{\tiny 
$\begin{array}{rrrr}
 -3 & -1 & 3 & 0\\
 -3 & -1 & 0 & 2\\
 -2 & -1 & 1 & 1
 \end{array}
$
}
\right],
  \]
 \end{minipage}
 \ 
 \begin{minipage}{6.5cm}
 \raisebox{4.7ex}{$\Sigma\, :=\, \mbox{\footnotesize ${\rm fan}(\sigma^+,\sigma^-,\tau)\ $}$}
 \tiny
 \tdplotsetmaincoords{75}{-310}
\begin{tikzpicture}[tdplot_main_coords, scale=.9]

\draw[->] (0,0,0) -- (1,0,0);
\draw[->] (0,0,0) -- (0,1,0);

% sig-
\fill[color=black!61] (-1,-1,-1) -- (1,0,1/3) -- (0,0,0) -- cycle;
\fill[color=black!50] (-1,-1,-1) -- (0,1,1/2) -- (0,0,0) -- cycle;
\fill[color=blue!30] (1,0,1/3) -- (0,1,1/2) -- (0,0,0) -- cycle;

% tau
\fill[color=black!73] (-1,-1,-2/3) -- (-1,-1,-1) -- (0,0,0) -- cycle;

% sig+
\fill[color=black!35] (-1,-1,-2/3) -- (1,0,1/3) -- (0,0,0) -- cycle;
\fill[color=black!50] (-1,-1,-2/3) -- (0,1,1/2) -- (0,0,0) -- cycle;
\fill[color=black!20] (1,0,1/3) -- (0,1,1/2) -- (0,0,0) -- cycle;

\draw[->] (0,0,0) -- (0,0,.7);

\draw (-1,-1,-2/3) node[anchor=south]{$v_1$};
\draw (-1,-1,-1) node[anchor=north]{$v_2$};
\draw (1,0,1/3) node[anchor=north west]{$v_3$};
\draw (0,1,1/2) node[anchor=west]{$v_4$};
  \end{tikzpicture}
 \end{minipage}
 \end{center}
where $\sigma^+:=\cone(v_1,v_3,v_4)$, $\sigma^-:=\cone(v_2,v_3,v_4)$ 
and $\tau:=\cone(v_1,v_2)$ in terms of the columns $v_1,v_2,v_3,v_4$ 
of~$P$.
Then $X \subseteq Z$ is a CEMDS encoded by
$(P,\Sigma,G)$, where $G=(T_1^3T_2+T_3^3+T_4^2)$.
These data are entered as follows in a \tt{Singular}
session. 
First we activate the \tt{compcox.lib} package:
\\[1ex]
\begingroup
\noindent
\footnotesize
\tt{> LIB \char`\" compcox.lib\char`\";}\\[1ex]
\endgroup
\noindent
Next the integral $3 \times 4$ matrix $P$ is defined.
\\[1ex]
\begingroup
\noindent
\footnotesize
\tt{> intmat P[3][4] = }\\
\tt{>\ \ -3, -1, 3, 0,}\\
\tt{>\ \ -3, -1, 0, 2,}\\
\tt{>\ \ -2, -1, 1, 1;}\\[1ex]
\endgroup
\noindent
To define the fan $\Sigma$, we first
define its maximal cones.
For example, $\sigma^+$:
\\[1ex]
\begingroup
\noindent
\footnotesize
\tt{> intmat mplus[3][3] = }\\
\tt{> 	-3,-3,-2,}\\
\tt{> 	3,0,1,}\\
\tt{> 	0,2,1;}\\
\tt{> cone sigplus = coneViaPoints(mplus);}\\[1ex]
\endgroup
\noindent
Proceeding similarly for the other cones 
$\tau$, $\sigma^-$, we can enter $\Sigma$ as\\[1ex] 
\begingroup
\noindent
\footnotesize
\tt{> fan Sigma = fanViaCones(sigplus, sigminus, tau);}\\[1ex]
\endgroup
\noindent
Finally, we enter the Cox ring of~$X$
and create the CEMDS~$X$ encoded by $(P,\Sigma, G)$:
\\[1ex]
\begingroup
\noindent
\footnotesize
\tt{> ring S = 0,T(1..4),dp;}\\
\tt{> ideal G = T(1)\textasciicircum 3*T(2) + T(3)\textasciicircum 3 + T(4)\textasciicircum 2;}\\
\tt{> CEMDS X = createCEMDS(P, Sigma, G);}\\[1ex]
\endgroup
We can print information on $X$ with the command\\[1ex]
\begingroup
\footnotesize
\tt{X;}\\
\tt{The CEMDS's ring:}\\
\tt{//   characteristic : 0}\\
\tt{//   number of vars : 4}\\
\tt{//        block   1 : ordering dp}\\
\tt{//                  : names    T(1) T(2) T(3) T(4)}\\
\tt{//        block   2 : ordering C}\\
\tt{The column matrix P of the CEMDS's fan's rays:}\\
\tt{    -3    -1     3     0}\\
\tt{    -3    -1     0     2}\\
\tt{    -2    -1     1     1}\\
\tt{The CEMDS's fan via its maximal cones, each one denoted by a column matrix of its rays:}\\
\tt{Dimension 2:}\\
\tt{1st maximal cone:}\\
\tt{-3, -1,}\\
\tt{-3, -1,}\\
\tt{-2, -1}\\
\tt{Dimension 3:}\\
\tt{1st maximal cone:}\\
\tt{-3, 3, 0,}\\
\tt{-3, 0, 2,}\\
\tt{-2, 1, 1}\\
\tt{2nd maximal cone:}\\
\tt{3, 0, -1,}\\
\tt{0, 2, -1,}\\
\tt{1, 1, -1}\\
\tt{The equations' ideal G embedding the MDS into its ambient toric variety:}\\
\tt{T(1)\textasciicircum 3*T(2)+T(3)\textasciicircum 3+T(4)\textasciicircum 2}\\[1ex]
\endgroup
We can extract the matrix $P$, the fan $\Sigma$, the polynomial ring $S$ and 
the ideal $G$ of the CEMDS $X$ as follows:\\[1ex]
\begingroup
\noindent
\footnotesize
\tt{> intmat PTmp =  X.rvcvzP; // the matrix P}\\
\tt{> fan SigmaTmp =  X.scnSigma; // the fan Sigma}\\
\tt{> def STmp =  X.R; setring RTmp; // the polynomial ring}\\
\tt{> ideal GTmp =  X.spG; // the ideal G of the Cox ring}
\endgroup
\noindent
\end{example}

The \tt{compcox.lib} package then provides tools 
for working with modifications of compatibly 
embedded Mori dream spaces.
Its core procedures implement 
Algorithms~3.1, 3.2, 3.5, 3.6 and 5.4
of~\cite{HaKeLa}; here is an overview:
\begin{itemize}
\item 
\tt{stretchCEMDS:} 
changes the presentation of a given 
(weak) CEMDS by adding generators of 
the Cox ring.
\item 
\tt{compressCEMDS:} 
changes the presentation of a given 
(weak) CEMDS by removing redundant 
generators of the Cox 
ring.
\item 
\tt{contractCEMDS: }
computes the new CEMDS arising from 
a given (weak) one by contracting divisors.
\item 
\tt{modifyCEMDS:} 
computes the new CEMDS arising from 
a given (weak) CEMDS by a toric ambient 
modification.
\item 
\tt{blowupCEMDS:} 
computes the new CEMDS arising from 
a given one by blowing-up a subvariety 
contained in the smooth locus.
\item 
\tt{blowupCEMDSpoints:} 
computes the new CEMDS arising from 
a given one by blowing-up a list of smooth 
points.
\end{itemize}
The last three procedures treat similar problems,
where \tt{modifyCEMDS} applies most generally and 
\tt{blowupCEMDSpoints} is the most special one.
Our reason for implementing three related 
procedures is that we can reduce  considerably
in the more special settings the computational 
complexity of the necessary verification steps.

We now present the use of these procedures and 
provide some background on each.
We do this by going through the 
Cox ring computation steps for the cubic surface $X$
with singularity type $A_4$.
In $\PP_3$, this surface is given as
\[
X
\ =\ 
V(x_0x_1x_3 - x_1^2x_2 + x_2^2x_3 - x_2x_3^2)
\ \subseteq\ \PP_3.
\]
The minimal resolution $X'$ of $X$ 
can be obtained as the blow-up of $\PP_2$ 
in the three toric fixed points 
$[1,0,0]$, $[0,1,0]$, $[0,0,1]$
plus the blow-up of the intersection point of the second 
exceptional divisor with the strict transform of $V(\PP_2;\,x_2)$.
This is followed by the blow-up of a general point in the last 
exceptional divisor and, finally, a blow-up of the preimage of 
$[0,1,1]\in \PP_2$.

The four toric blow-ups of $\PP_2$ are easily 
performed in terms of fans.
The result is the complete toric surface $X_0$ 
with fan $\Sigma_0$
whose rays are generated by the columns of the matrix~$P_0$:
 \begin{center}
 \begin{minipage}{4cm}
\[
P_0
\ :=\ 
  \left[
  \mbox{\tiny $
  \begin{array}{rrrrrrr}
  -1 & 1 & 0 & 1 & -1 &  0 & -1 \\
  -1 & 0 & 1 & 1 &  0 & -1 &  1
  \end{array}
  $}
  \right],
  \]
 \end{minipage}
 \qquad\qquad\qquad\qquad
 \begin{minipage}{3cm}
 \raisebox{1.5ex}{$\Sigma_0 \ := \ \ $}
 \begin{tikzpicture}[scale=.4]
 \fill[color=black!30] (1,0) -- (1,1) -- (0,1) -- (-1,1) -- (-1,-1) -- (0,-1) -- (1,0) -- cycle;
 
 \draw[thick] (0,0) -- (1,0); 
 \draw[thick] (0,0) -- (0,1);
 \draw[thick] (0,0) -- (-1,-1);
 \draw[thick] (0,0) -- (1,1);
  \draw[thick] (0,0) -- (0,-1);
 \draw[thick] (0,0) -- (-1,0);
  \draw[thick] (0,0) -- (-1,1); 
 \end{tikzpicture}
 \end{minipage}
 \end{center}

\begin{step}[\tt{createCEMDS}]
\label{step:createCEMDS}
Following the lines of Example~\ref{ex:createCEMDS0},
we enter the CEMDS $X_0 = (P_0,\Sigma_0,G_0)$
in \tt{Singular} where $G_0 = (0)$.
\end{step}

We now perform the fifth blow-up and thereby 
demonstrate the use of the function \tt{modifyCEMDS}.
A preparatory step is to
present $X_0$ as a CEMDS $X_1 \subseteq Z_1$ 
such that the point to be blown up is suitably cut 
out by toric divisors.

\begin{step}[\tt{stretchCEMDS} and \tt{compressCEMDS}]
\label{step:stretchCEMDS}
Consider the general point $x_0 \in X_0$ 
in the last exceptional divisor with 
Cox coordinates 
$z_0 = (1,1,1,1,1,1,0) \in p_0^{-1}(x_0) \subseteq  \KK^7$.
The~\tt{compcox.lib} package computes equations 
for the closure of the fiber $p_0^{-1}(x_0)$:\\[1ex]
\begingroup
\footnotesize
\tt{> vector z0 = [1,1,1,1,1,1,0];}\\
\tt{> ideal FL = varproductOrbitClosureIdeal(P0, z0);}\\
\tt{FL;}\\
\tt{T(2)*T(3)*T(4)\textasciicircum 2-T(1)\textasciicircum 2*T(5)*T(6),}\\
\tt{T(7)}\\[1ex]
\endgroup
These polynomials define divisors on $X$ having 
$x_0$ as their intersection point.
A toric embedding such that $x_0$ is cut out by 
toric divisors is obtained via \tt{stretchCEMDS};
we add a new variable representing the first 
equation of \tt{FL}:
 \\[1ex]
\begingroup
\footnotesize
\tt{> CEMDS X1 = stretchCEMDS(X0, list(FL[1]), list(), 1); // 1 means compute the fan}\\
\tt{> def S1 = X1.R;}\\
\tt{> setring S1;}\\
\tt{> X1.spG; // print the ideal of the Cox ring}\\
\tt{-T(2)*T(3)*T(4)\textasciicircum 2+T(1)\textasciicircum 2*T(5)*T(6)+T(8)}
\\[1ex]
\endgroup
The result is a CEMDS $X_1 \subseteq Z_1$ with $X_0 \cong X_1$ 
having the desired property.
Note that $R_1$ is isomorphic to $R_0$ as a graded ring.
To get rid of redundant generators (which occur frequently 
during computations), one can use \tt{compressCEMDS}:
\\[1ex]
\begingroup
\footnotesize
\tt{> CEMDS X0again = compressCEMDS(X1, 0, list(), 1); // this is X0}
\endgroup
\end{step}

Before presenting \tt{modifyCEMDS}, 
let us indicate the theoretical background.
For $i = 1,2$, let $Z_i$ be projective toric varieties,
$X_i \subseteq Z_i$ be closed subvarieties both
intersecting the big torus orbits $\TT_i \subseteq Z_i$
and let $\pi \colon Z_2 \to Z_1$ be a birational
toric morphism such that $\pi(X_2) = X_1$
holds. Then, following the notation introduced
before, we have a commutative diagram
$$
\xymatrix{
{\KK^{r_2}}
\ar@{}[r]|\supseteq
&
{\b{X}_2}
&
{\b{X}_1}
\ar@{}[r]|\subseteq
&
{\KK^{r_1}}
\\
{\rq{Z}_2}
\ar@{}[r]|\supseteq
\ar@{}[u]|{\rotatebox[origin=c]{90}{$\scriptstyle \subseteq$}}
\ar[d]_{p_2}
&
{\rq{X}_2}
\ar@{}[u]|{\rotatebox[origin=c]{90}{$\scriptstyle \subseteq$}}
\ar[d]_{p_2}
&
{\rq{X}_1}
\ar@{}[u]|{\rotatebox[origin=c]{90}{$\scriptstyle \subseteq$}}
\ar[d]^{p_1}
\ar@{}[r]|\subseteq
&
{\rq{Z}_1}
\ar@{}[u]|{\rotatebox[origin=c]{90}{$\scriptstyle \subseteq$}}
\ar[d]^{p_1}
\\
Z_2
\ar@{}[r]|\supseteq
\ar@/_2pc/[rrr]_{\pi}
&
X_2
\ar[r]
&
X_1
\ar@{}[r]|\subseteq
&
Z_1
\\
&&&
}
$$
Let $\rq{X}_i \subseteq \rq{Z}_i$ be the
closure of $p_i^{-1}(X \cap \TT_i)$ and
$I_i \subseteq \KK[T_1,\ldots,T_{r_i}]$
the vanishing ideal of the closure
$\b{X}_i \subseteq \KK^{r_i}$
of $\rq{X}_i \subseteq \rq{Z}_i$.
Set $R_i := \KK[T_1,\ldots,T_{r_i}] / I_i$.
Note that $R_i$ is graded by
$K_i = \Cl(Z_i)$.
We have the following statements,
see~\cite[Thm.~2.6]{HaKeLa}.
\begin{itemize}
\item
If $X_1 \subseteq Z_1$ is a CEMDS, the ring $R_2$
is normal and 
$T_{1}, \ldots, T_{r_2}$ define pairwise
non-associated $K_2$-primes in~$R_2$,
then $X_2 \subseteq Z_2$ is a CEMDS. 
In particular, $K_2$ is the divisor class group
of $X_2$ and $R_2$ is the Cox ring of~$X_2$.
\item
If $X_2 \subseteq Z_2$ is a CEMDS,
then $X_1 \subseteq Z_1$ is a CEMDS.
In particular, $K_1$ is the divisor class group
of $X_1$ and $R_1$ is the Cox ring of~$X_1$.
\end{itemize}

\begin{step}[\tt{modifyCEMDS}]
\label{step:modifyCEMDS}
We continue Step~\ref{step:stretchCEMDS}.
First, read out the needed data of the weak CEMDS
$X_1 \subseteq Z_1$ encoded by $(P_1,\Sigma_1,G_1)$:
\\[1ex]
\begingroup
\noindent
\footnotesize
\tt{> intmat P1 =  X1.rvcvzP;}\\
\tt{> fan Sigma1 =  X1.scnSigma;}\\[1ex]
\endgroup
\noindent
Our task is to blow up the point $x_1\in X_1$
with Cox coordinates $z_1=(z_0,0)\in \KK^{8}$
by means of a toric ambient modification.
For this, we perform the stellar subdivision 
$\Sigma_2 \to \Sigma_1$
at $v:=p_7 + p_{8}$ 
with the columns $p_i$ of~$P_1$:\\[1ex]
\begingroup
\noindent
\footnotesize
\tt{> intvec v = intmatTakeCol(P1, 7) + intmatTakeCol(P1, 8);}\\
\tt{> fan Sigma2 = stellarSubdivision(Sigma1, v);}\\
\tt{> intmat P2 = intmatAppendCol(P1, v);}\\[1ex]
\endgroup
\noindent
The next command computes the proper transform $X_2 \subseteq Z_2$ 
under the toric ambient modification $Z_2 \to Z_1$ given by 
the stellar subdivision $\Sigma_2 \to \Sigma_1$.
\\[1ex]
\begingroup
\noindent
\footnotesize
\tt{> list L = modifyCEMDS(X1, list(P2, Sigma2), 1);}\\
\tt{> CEMDS X2 = L[1];}\\
\tt{> L[2];}\\
\tt{1}\\[1ex]
\endgroup
\noindent
The parameter ``1'' in the call to \tt{modifyCEMDS} advises
the algorithm to verify whether the result is a CEMDS
using the criterion presented before.
The result of this verification step is stored in \tt{L[2]}:
the entry $1$ tells us that $X_2$, encoded by $(P_2,\Sigma_2,G_2)$,
is a CEMDS. 
Let us print $G_2$.\\[1ex]
\begingroup
\noindent
\footnotesize
\tt{> ring S2 = X2.R;}\\
\tt{> setring S2;}\\
\tt{> X2.spG;}\\
\tt{T(2)*T(3)*T(4)\textasciicircum 2-T(1)\textasciicircum 2*T(5)*T(6)-T(8)*T(9)}\\[1ex]
\endgroup
\noindent
Note that it still remains to show that 
$X_2$  is indeed the blow-up of $X_1$ in the point~$x_1$.
In this case, this follows from smoothness of~$X_2$.
\end{step}

For the sixth blow-up step, we use the function 
\tt{blowupCEMDS}. It applies to the more special case of a blow-up
$\pi\colon X_2\to X_1$ of a CEMDS $X_1$ in a subvariety $C\subseteq X_1$ that
is contained in the smooth locus $X_1^\reg$.

Let us recall from~\cite[Algorithm~5.4]{HaKeLa} the background 
of the algorithm.
By the hypothesis $C\subseteq X_1^\reg$,
the map $\pi$ induces a well defined 
pull back map at the level of divisor class groups 
$\pi^*\colon {\rm Cl}(X_1)\to {\rm Cl}(X_2)$. If we denote
by $R_1=\mathcal R(X_1)$ the Cox ring of $X_1$, by
$I$ the ideal $I(p^{-1}(C))$ and by $J = I(\overline X_1
\setminus\widehat X_1)$ the irrelevant ideal,
then we can form the {\em saturated Rees algebra}:
\[
 R_1[I]^{\rm sat}
 \, :=\, 
 \bigoplus_{d\in\mathbb Z}(I^{d}\colon J^\infty)t^{-d},
\]
where we set $I^{d}\colon J^\infty = R$ if $d\leq 0$.
Note that $R_1[I]^{\rm sat}$ is canonically graded by 
${\rm Cl}(X_2) = {\rm Cl}(X_1)\oplus \ZZ$.
Let $E$ be the exceptional divisor of the blow-up $\pi$
and let $f_E\in\mathcal R(X)_{[E]}$ be the canonical
section for $E$. Then the following is an 
isomorphism of ${\rm Cl}(X_2)$-graded algebras~\cite[Prop.~5.2]{HaKeLa}:
\[
 R_1[I]^{\rm sat}\to\mathcal R(X_2),
 \qquad
 g\cdot t^{-d}\mapsto \pi^*g\cdot f_E^{-d}.
\]
Given ${\rm Cl}(X_1)$-prime generators $f_1,\ldots,f_l$ for $I$
and coprime integers $d_1,\ldots,d_l\in \ZZZ$,
the algorithm forms the  ${\rm Cl}(X_2)$-graded algebra $A$
generated by $f_1,\ldots,f_l$.
We have 
\[
 R_1[I]\ :=\ 
  \bigoplus_{d\in\mathbb Z}I^{d}t^{-d}
 \ \subseteq\ A\ \subseteq\ 
 R_1[I]^{\rm sat}.
\]
Then we test if the second inclusion is an equality:
If $T^\nu$ is the product over all $T_i$ 
with $C\not\subseteq V(X_1;\,T_i)$,
the equality $A=R_1[I]^{\rm sat}$ holds if and only
if the following inequality holds~\cite[Algorithm~5.4]{HaKeLa}
$$
\dim(I_2 +\langle T_{r_2}\rangle)\ 
>\ 
\dim(I_2 + \langle T_{r_2},T^\nu\rangle)
$$
where we present $A = \KT{r_2}/I_2$ 
and $T_{r_2}$ corresponds to the variable $t$.
Then $A = \KT{r_2}/I_2$ is the Cox ring of the blow-up $X_2\to X_1$
along~$C$.
 \label{eq:multiplic}
 Also note that, geometrically, $d_i$ is the 
multiplicity of the prime divisor defined by $f_i$
at the generic point of $C$.

We now continue Step~\ref{step:modifyCEMDS}, i.e.,
the computation of the cubic surface with singularity type $A_4$,
by applying the procedure \tt{blowupCEMDS}.

\begin{step}[\tt{blowupCEMDS}]
\label{step:blowupCEMDS}
Consider the CEMDS 
$X_2 \subseteq Z_2$ given by $(P_2,\Sigma_2,G_2)$
as computed in Step~\ref{step:modifyCEMDS}. 
We blow up $X_2$ in a point $x_2$ with Cox coordinates 
$z_2 \in \b{X}_2 \subseteq \KK^9$ that projects
under the previous sequence of blowing-up and embedding
$$
X_2
\,\longrightarrow\,
X_1
\,\longleftarrow\, 
X_0
\,\longrightarrow\, 
\PP_2
$$ 
to the point $[0,1,1]\in \PP_2$.
We choose the point $x_2\in X_2$ with Cox coordinates
$z_2=(0,1,\ldots,1)\in \KK^{9}$.
In the \texttt{Singular} session, we write\\[1ex]
\begingroup
\footnotesize
\tt{> def S2 = X2.R;}\\
\tt{> setring S2;}\\
\tt{> vector z2 = [0,1,1,1,1,1,1,1,1];}\\
\tt{> ideal FL = varproductOrbitClosureIdeal(P2, z2); // we defined P2 in~\ref{step:modifyCEMDS}}\\
\tt{> FL;}\\
\tt{-T(2)*T(3)*T(4)\textasciicircum 2+T(7)*T(9),}\\
\tt{-T(3)\textasciicircum 2*T(4)\textasciicircum 2*T(5)*T(8)\textasciicircum 2*T(9)+T(6)*T(7),}\\
\tt{-T(3)*T(5)*T(8)\textasciicircum 2*T(9)\textasciicircum 2+T(2)*T(6),}\\
\tt{T(1)}\\[1ex]
\endgroup
\noindent
Then \tt{FL} contains generators for the closure 
of the fiber $p_2^{-1}(x_2)$.
Using the first three generators, 
we blow-up $X_2$ in $x_2$ with the commands\\[1ex]
\begingroup
\footnotesize
\tt{> list L =  FL[1..4];}\\
\tt{> intvec d = 1,1,1,1;}\\
\tt{> CEMDS X3 = blowupCEMDS(X2, L, d, 1); }
\begin{center}
\tiny
The resulting embedded space was successfully verified to be a CEMDS.
\end{center}
\endgroup
\noindent
where the entries of \tt{d} are the multiplicities $d_i$
of the elements of \tt{L} and the parameter ``$1$'' 
tells the program to compute the fan of the ambient toric variety.
Then the CEMDS $X_3=(P_3,\Sigma_3,G_3)$ is the desired blow-up.
We can inspect the ideal $\<G_3\>$ of its Cox ring $R_3$ with\\[1ex]
\begingroup
\footnotesize
\tt{> def S3 = X3.R;}\\
\tt{> setring S3;}\\
\tt{> X3.spG;}\\
\tt{-T(4)*T(5)\textasciicircum 2*T(11)\textasciicircum 2*T(12)-T(2)*T(3)\textasciicircum 2*T(10)+T(8)*T(9),,}\\
\tt{-T(2)*T(4)\textasciicircum 2*T(5)*T(6)\textasciicircum 2*T(8)*T(11)\textasciicircum 2*T(12)+T(1)*T(9)-T(7)*T(10),,}\\
\tt{T(4)*T(5)*T(11)\textasciicircum 2*T(12)\textasciicircum 2-T(1)*T(2)*T(3)\textasciicircum 2+T(7)*T(8),,}\\
\tt{-T(2)\textasciicircum 2*T(3)\textasciicircum 2*T(4)*T(6)\textasciicircum 2*T(8)+T(5)*T(7)+T(9)*T(12),,}\\
\tt{-T(2)*T(4)*T(6)\textasciicircum 2*T(8)\textasciicircum 2+T(1)*T(5)+T(10)*T(12)}\\[1ex]
\endgroup
Then $R_3$ is the Cox ring of the resolution of the cubic surface;
it is listed  (up to isomorphism) in Theorem~\ref{thm:cubic}, case~$A_4$.
\end{step}

\begin{remark}[\tt{blowupCEMDSpoints}]
Blowing-up a series of smooth points on a variety
can also be accomplished using the function 
\tt{blowupCEMDSpoints} which automatically performs the steps
indicated in Step~\ref{step:blowupCEMDS}
and works if $d_i=1$ holds for all~$i$.
For instance, we could have replaced Step~\ref{step:stretchCEMDS}, \ref{step:modifyCEMDS} and~\ref{step:blowupCEMDS} by the following commands:\\[1ex]
\begingroup
\footnotesize
\tt{> vector z1 = [1,1,1,1,1,1,0];}\\ 
\tt{> vector z2 = [0,1,1,1,1,1,1];}\\
\tt{> CEMDS X3 = blowupCEMDSpoints(X1, list(z1, z2), 1, 1); }
\endgroup
\end{remark}

Finally, we turn to the Cox ring of the singular 
cubic surface $X$ of Step~\ref{step:createCEMDS}.
Note that $X$ is obtained from the resolution 
$X' = X_3$ by contracting the $(-2)$-curves.
In terms of Cox rings, the contraction of a 
prime divisor merely means to set the corresponding 
variable equal to one and to adapt the grading 
accordingly; this is basically the statement 
of~\cite[Prop.~2.2]{HaKeLa}.

\begin{step}[\tt{contractCEMDS}]
\label{ex:contractCEMDS}
Consider the Cox ring $R_3$ of the resolution $X' = X_3$
computed in Step~\ref{step:blowupCEMDS}.
An inspection of the sequence of blow-ups (and one embedding)  
$$
X_3
\,\longrightarrow\,
X_2
\,\longrightarrow\,
X_1
\,\longleftarrow\, 
X_0
\,\longrightarrow\, \PP_2
$$ 
shows that the $(-2)$-curves of $X_3$ correspond to 
the variables
$T_2,T_4,T_6,T_{11}$ in $R_3$ 
(note that we have to take into account
the substitutions 
$T_1 = T_{13}$ and $T_{10} = T_2T_3T_4^2-T_8T_9$ 
performed by \tt{blowupCEMDS} in
Step~\ref{step:blowupCEMDS}). 
We obtain the Cox ring of their contraction by\\[1ex]
\begingroup
\footnotesize
\tt{> intvec con11 = 1,1,1,1,1,1,1,1,1,1,0,1;}\\
\tt{> CEMDS X4 = contractCEMDS(X3, con11);}\\
\tt{> intvec con6 = 1,1,1,1,1,0,1,1,1,1,1;}\\
\tt{> CEMDS X5 = contractCEMDS(X4, con6);}\\
\tt{> intvec con4 = 1,1,1,0,1,1,1,1,1,1;}\\
\tt{> CEMDS X6 = contractCEMDS(X5, con4);}\\
\tt{> intvec con2 = 1,0,1,1,1,1,1,1,1;}\\
\tt{> CEMDS X7 = contractCEMDS(X6, con2);}\\[1ex]
\endgroup
Then the CEMDS $X_7=(P_7,\Sigma_7,G_7)$ is the singular cubic surface $X$
with singularity type $A_4$.
Its Cox ring is $\KT{8}/\<G_7\>$ where $G_7$ is\\[1ex]
\begingroup
\footnotesize
\tt{> def S7 = X7.R;}\\
\tt{> setring S7;}\\
\tt{> X7.spG;}\\
\tt{-T(2)\textasciicircum 2*T(7)-T(3)\textasciicircum 2*T(8)+T(5)*T(6),}\\
\tt{-T(3)*T(5)*T(8)+T(1)*T(6)-T(4)*T(7),}\\
\tt{-T(1)*T(2)\textasciicircum 2+T(3)*T(8)\textasciicircum 2+T(4)*T(5),}\\
\tt{-T(2)\textasciicircum 2*T(5)+T(3)*T(4)+T(6)*T(8),}\\
\tt{T(1)*T(3)-T(5)\textasciicircum 2+T(7)*T(8)}
\endgroup
\end{step}

\begin{example}[Parameters]\label{ex:parameters}
It is possible to compute with formal parameters in \tt{compcox.lib}:
they can appear both as coefficients in polynomials
and in Cox coordinates of points.
As an example, we compute the Cox ring of the 
blow-up of $\PP_2$ in five general points, i.e.,
the smooth del Pezzo surface $X$ of degree~$4$.
We first enter $\PP_2$ as a CEMDS $X_1=(P_1,\Sigma_1,(0))$;
the matrix $P_1$ is\\[1ex]
\begingroup
\footnotesize
\tt{> intmat P1[2][3] = }\\
\tt{> -1,1,0}\\
\tt{> -1,0,1;}\\[1ex]
\endgroup
The three maximal cones of the fan $\Sigma$ are entered as before:\\[1ex]
\begingroup
\footnotesize
\tt{> intmat m1[2][2] = }\\
\tt{> 1,0,}\\
\tt{> 0,1;}\\
\tt{> cone c1 = coneViaPoints(m1);}\\[1ex]
\endgroup
Enter the other two cones accordingly and define $\Sigma$ by\\[1ex]
\begingroup
\footnotesize
\tt{> fan Sigma1 = fanViaCones(c1, c2, c3);}\\[1ex]
\endgroup
Define the Cox ring $R_1$ of $X_1$ with two formal parameters $a,b$ 
and define $X_1$ by
\\[1ex]
\begingroup
\footnotesize
\tt{> ring S1 = (0,a,b),T(1..3),dp;}\\
\tt{> ideal G1 = 0;}\\
\tt{> CEMDS X1 = createCEMDS(P1, Sigma1, G1);}\\[1ex]
\endgroup
We now blow-up $X_1$ in five general points:
\\[1ex]
\begingroup
\footnotesize
\tt{> vector z1 = [1,0,0];}\\
\tt{> vector z2 = [0,1,0];}\\
\tt{> vector z3 = [0,0,1];}\\
\tt{> vector z4 = [1,1,1];}\\
\tt{> vector z5 = [1,a,b];}\\
\tt{> CEMDS X2 = blowupCEMDSpoints(X1, list(z1,z2,z3,z4,z5), 1);}
\endgroup
\begin{center}
\tiny
The resulting embedded space was successfully verified to be a CEMDS.
\end{center}
\noindent
The Cox ring of $X_2$ is $\KT{16}/\<G_2\>$ where the polynomials $G_2$ are\\[1ex]
\begingroup
\footnotesize
\tt{> def R2 = X2.R;}\\
\tt{> setring R2;}\\
\tt{> X2.spG;}\\
\tt{T(4)*T(13)+(-a)*T(5)*T(14)+(b)*T(6)*T(15),}\\
\tt{T(1)*T(13)-T(2)*T(14)+T(3)*T(15),}\\
\tt{T(10)*T(12)+(-a+1)*T(1)*T(13)+(-a+b)*T(3)*T(15),}\\
\tt{T(6)*T(12)-T(8)*T(13)+(a)*T(7)*T(14),}\\
\tt{T(5)*T(12)-T(9)*T(13)+(b)*T(7)*T(15),}\\
\tt{T(4)*T(12)+(-a)*T(9)*T(14)+(b)*T(8)*T(15),}\\
\tt{T(10)*T(11)+(-a+1)*T(5)*T(14)+(b-1)*T(6)*T(15),}\\
\tt{T(3)*T(11)-T(8)*T(13)+T(7)*T(14),}\\
\tt{T(2)*T(11)-T(9)*T(13)+T(7)*T(15),}\\
\tt{T(1)*T(11)-T(9)*T(14)+T(8)*T(15),}\\
\tt{(a-1)*T(5)*T(8)+(-b+1)*T(6)*T(9)-T(11)*T(16),}\\
\tt{(a*b-b)*T(2)*T(8)+(-a*b+a)*T(3)*T(9)-T(12)*T(16),}\\
\tt{T(4)*T(7)-T(5)*T(8)+T(6)*T(9),}\\
\tt{T(1)*T(7)-T(2)*T(8)+T(3)*T(9),}\\
\tt{(a-b)*T(2)*T(6)+(a)*T(7)*T(10)-T(13)*T(16),}\\
\tt{(b-1)*T(1)*T(6)-T(8)*T(10)+T(14)*T(16),}\\
\tt{T(3)*T(5)-T(2)*T(6)-T(7)*T(10),}\\
\tt{(a-1)*T(1)*T(5)-T(9)*T(10)+T(15)*T(16),}\\
\tt{T(3)*T(4)-T(1)*T(6)-T(8)*T(10),}\\
\tt{T(2)*T(4)-T(1)*T(5)-T(9)*T(10)}
\endgroup
\end{example}

\begin{remark}
\label{rem:nofan}
If one is merely interested in the Cox ring $\mathcal R(X_2)$ of 
a blow up $X_2\to X_1$ and not the CEMDS $X_2$, 
one can use \tt{blowupCEMDS} without 
specifying the ambient fan $\Sigma_1$ and turn of the computation
of $\Sigma_2$, see Example~\ref{ex:X19}.
\end{remark}

\section{Cubic surfaces and other blow-ups of the plane}
\label{sec:surfaces}

Here we compute Cox rings of rational surfaces. The first example
class arises from cubic surfaces $Y \subseteq \PP_3$ with at most
rational double points as singularities. Besides the smooth cubic
surfaces, there are precisely $20$ types of singular cubic surfaces
with rational double points that are distinguished by their
\emph{singularity types} (in the $ADE$-classification), as in 
Table~\ref{tab:cubic}. For each singularity type, there may be an infinite
family of isomorphy classes of cubic surfaces. For details on the
geometry and classification of singular cubic surfaces, see
\cite[Section~8 and 9]{Do}, for example.

We are interested in the Cox ring of a minimal desingularization $X$
of the singular cubic surface $Y$. Since $X$ arises by a sequence of
blow-ups from the projective plane $\PP_2$, we can apply our
algorithms. Note that the Cox ring of the singular surface $Y$ can be
obtained from the Cox rings of its desingularization $X$ by
\cite[Proposition~2.2]{HaKeLa}.

Here, we complete the computation of Cox rings of minimal
desingularizations of such cubic surfaces. The previously known cases
are the Cox rings of smooth cubic surfaces \cite[Theorem~9.1]{LaVe}, the
toric cubic surface of type $3A_2$ \cite{cox} and the seven types of
cubic surfaces whose Cox rings have precisely one relation
\cite{HaTs, Der, Hug}. We provide a precise description
of the Cox rings for the remaining $12$ types where there is more
than one relation in the Cox ring.

All results are summarized in Table~\ref{tab:cubic}. Note that Manin's
conjecture is known only for the toric type and five examples of the seven types of singular cubic surfaces where the Cox ring has precisely one relation. Hence our results should help to approach the other 12 types.

\begin{table}[ht]
  \centering
  \small\begin{tabular}{cccccccc}
      \hline
      no. & singularity type & $\#$lines & $\#$generators & $\#$relations & Manin's conjecture\\
      \hline
      0&$-$ & 27 & 27 & 81 & \\
      i&$A_1$ & 21 & 22 & 48 & \\
      ii&$2A_1$ & 16 & 18 & 27 & \\
      iii&$A_2$ & 15 & 17 & 21 & \\
      iv&$3A_1$ & 12 & 15 & 15 & \\
      v&$A_2A_1$ & 11 & 14 & 10 & \\
      vi&$A_3$ & 10 & 13 & 6 & \\
      vii&$4A_1$ & 9 & 13 & 9  & \text{bounds: \cite{MR2075628}} \\
      viii&$A_22A_1$ & 8 & 12 & 5 & \\
      ix&$A_3A_1$ & 7 & 11 & 2 & \\
      x&$2A_2$ & 7 & 11 & 2 & \\
      xi&$A_4$ & 6 & 12 & 5 & \\
      xii&$D_4$ & 6 &  10 & 1 & \text{\cite{arXiv:1207.2685}} \\
      xiii&$A_32A_1$ & 5 & 10 & 1 & \\
      xiv&$2A_2A_1$ & 5 & 10 & 1 & \text{\cite{arXiv:1105.3495}}\\
      xv&$A_4A_1$ & 4 & 10 & 1 & \\
      xvi&$A_5$ & 3 & 13 & 9 & & \\
      xvii&$D_5$ & 3 & 10 & 1 & \text{\cite{MR2520769}}\\
      xviii&$3A_2$ & 3 & 9 & -- & \text{\cite{MR1620682}, \dots, \cite{arXiv:1204.0383}}\\
      xix&$A_5A_1$ & 2 & 10 & 1 & \text{\cite{arXiv:1205.0373}}\\
      xx&$E_6$ & 1 & 10 & 1 & \text{\cite{MR2332351,arXiv:1311.2809}}\\
      \hline
  \end{tabular}
  \smallskip
  \caption{Cox rings of minimal desingularizations of cubic surfaces}
  \label{tab:cubic}
\end{table}

\begin{theorem}
\label{thm:cubic} 
Let $Y$ be a cubic surface whose singularities are rational double
points. Let $X$ be its minimal desingularization. Assume that the Cox
ring of $X$ has at least two relations.  In the following list sorted
by the singularity type of $Y$, we provide the Cox ring of $X$ by
specifying generators, their degrees and the ideal of relations.
\\[1ex]
(i) The blow-up of $\PP_2$ in $[1,0,0]$, $[0,1,0]$, $[0,0,1]$,
$[1,1,1]$, $[1,\lambda,0]$, and $[1,\mu,\kappa]$ with singularity type
$A_1$ has the $\ZZ^7$-graded Cox ring $\KT{22}/I$ with the following
generators for $I$ and degree matrix where $\lambda,\mu,\kappa\in
\KK^*\setminus \{1\}$ are such that $\mu\ne\lambda$ and
$\mu\ne\kappa$.
\begin{center}
\tiny
		$
\begin{array}{ll}
		(-\mu +\kappa )T_{3}T_{18}T_{21}+\kappa T_{13}T_{20}-T_{12}T_{22}, & 
(\kappa -1)T_{3}T_{17}T_{21}+\kappa T_{11}T_{20}-T_{10}T_{22}, \\
(\lambda -\mu )T_{9}T_{20}+(\mu -1)T_{8}T_{21}+(-\lambda +1)T_{7}T_{22}, & 
(\lambda \kappa -\lambda +\mu -\kappa )T_{3}T_{17}T_{18}+\kappa T_{6}T_{20}-T_{4}T_{22}, \\
(\lambda \kappa -\lambda +\mu -\kappa )T_{9}T_{19}+(-\mu +1)T_{6}T_{21}+(\lambda -1)T_{5}T_{22}, & 
(-\lambda +1)T_{3}T_{17}T_{18}+T_{8}T_{19}-T_{6}T_{20}, \\
(\lambda \kappa -\lambda +\mu -\kappa )T_{7}T_{19}+(\lambda -\mu )T_{5}T_{20}+(-\mu +1)T_{4}T_{21}, & 
-T_{3}T_{18}T_{21}+T_{2}T_{19}-T_{13}T_{20}, \\
-T_{3}T_{17}T_{21}+T_{1}T_{19}-T_{11}T_{20}, & 
(\mu -1)T_{11}T_{18}+(-\kappa +1)T_{9}T_{19}-T_{5}T_{22}, \\
(\lambda -\mu )T_{10}T_{18}+\kappa T_{7}T_{19}-T_{4}T_{21}, & 
(\lambda -1)T_{1}T_{18}+T_{9}T_{20}-T_{8}T_{21}, \\
T_{13}T_{17}-T_{11}T_{18}+T_{9}T_{19}, & 
T_{12}T_{17}-\mu T_{10}T_{18}+\kappa T_{7}T_{19}, \\
T_{2}T_{17}-T_{1}T_{18}+T_{9}T_{20}, & 
\mu T_{3}T_{5}T_{9}+T_{11}T_{15}-T_{6}T_{16}, \\
\end{array}
$
\\[1ex]
\setlength\arraycolsep{4pt}
\mbox{\tiny
%old
% 	$
% 	\left[
% 	\begin{array}{rrrrrrrrrrrrrrrrrrrrrr}
% 	1 & 1 & 1 & 1 & 1 & 1 & 1 & 1 & 1 & 1 & 1 & 1 & 1 & 2 & 2 & 2 & 0 & 0 & 0 & 0 & 0 & 0 \\
% 0 & -1 & -1 & 0 & 0 & 0 & 0 & 0 & 0 & 0 & 0 & -1 & -1 & 0 & -1 & -1 & 1 & 0 & 0 & 0 & 0 & 0 \\
% -1 & 0 & -1 & 0 & 0 & 0 & 0 & 0 & 0 & -1 & -1 & 0 & 0 & -1 & 0 & -1 & 0 & 1 & 0 & 0 & 0 & 0 \\
% 0 & 0 & 1 & 1 & 1 & 1 & 0 & 0 & 0 & 1 & 1 & 1 & 1 & 1 & 1 & 1 & 0 & 0 & 1 & 0 & 0 & 0\\
% 0 & 0 & 0 & 0 & -1 & -1 & 0 & 0 & -1 & 0 & -1 & 0 & -1 & -1 & -1 & -1 & 0 & 0 & 0 & 1 & 0 & 0 \\
% 1 & 1 & 0 & 0 & 1 & 0 & 1 & 0 & 1 & 1 & 1 & 1 & 1 & 1 & 1 & 2 & 0 & 0 & 0 & 0 & 1 & 0 \\
% 0 & 0 & 0 & -1 & -1 & 0 & -1 & 0 & 0 & -1 & 0 & -1 & 0 & -1 & -1 & -1 & 0 & 0 & 0 & 0 & 0 & 1 
% 	\end{array}
% 	\right]
% 	$
%new
	$
	\left[
	\begin{array}{rrrrrrrrrrrrrrrrrrrrrr}
	1 & 1 & 1 & 1 & 1 & 1 & 1 & 1 & 1 & 1 & 1 & 1 & 1 & 2 & 2 & 2 & 0 & 0 & 0 & 0 & 0 & 0 \\
0 & -1 & -1 & 0 & 0 & 0 & 0 & 0 & 0 & 0 & 0 & -1 & -1 & 0 & -1 & -1 & 1 & 0 & 0 & 0 & 0 & 0 \\
-1 & 0 & -1 & 0 & 0 & 0 & 0 & 0 & 0 & -1 & -1 & 0 & 0 & -1 & 0 & -1 & 0 & 1 & 0 & 0 & 0 & 0 \\
-1 & -1 & 0 & 0 & 0 & 0 & -1 & -1 & -1 & 0 & 0 & 0 & 0 & -1 & -1 & -1 & 0 & 0 & 1 & 0 & 0 & 0\\
0 & 0 & 0 & 0 & -1 & -1 & 0 & 0 & -1 & 0 & -1 & 0 & -1 & -1 & -1 & -1 & 0 & 0 & 0 & 1 & 0 & 0 \\
0 & 0 & -1 & -1 & 0 & -1 & 0 & -1 & 0 & 0 & 0 & 0 & 0 & -1 & -1 & 0 & 0 & 0 & 0 & 0 & 1 & 0 \\
0 & 0 & 0 & -1 & -1 & 0 & -1 & 0 & 0 & -1 & 0 & -1 & 0 & -1 & -1 & -1 & 0 & 0 & 0 & 0 & 0 & 1 
	\end{array}
	\right]
	$
}
\setlength\arraycolsep{5pt}
\end{center}
(ii)  The blow-up of $\PP_2$ in $[1,0,0]$, $[0,1,0]$, $[0,0,1]$,
$[1,1,1]$, $[1,\lambda,0]$, and $[1,0,\mu]$
with singularity type 
$2A_1$
has the $\ZZ^7$-graded Cox ring
$\KT{18}/I$ with generators for $I$
and the degree matrix as follows
where  $\lambda,\mu\in \KK^*\setminus\{1\}$
are such that
 $\lambda\mu-\lambda-\mu\ne 0$.
\begin{center}
\tiny 
		$
		\begin{array}{ll}
		T_{3}T_{14}T_{17}-T_{2}T_{15}T_{18}+T_{11}T_{16}, &
(\mu -1)T_{3}T_{13}T_{17}+\mu T_{10}T_{16}-T_{9}T_{18}, \\
(\lambda -1)T_{2}T_{13}T_{18}+\lambda T_{8}T_{16}-T_{7}T_{17},  &
(\lambda \mu -\lambda -\mu )T_{3}T_{13}T_{14}+\mu T_{6}T_{16}-T_{4}T_{18}, \\
(\lambda \mu -\lambda -\mu )T_{2}T_{13}T_{15}+\lambda T_{5}T_{16}-T_{4}T_{17},  &
(\lambda \mu -\lambda -\mu )T_{8}T_{15}+T_{6}T_{17}+(-\lambda +1)T_{5}T_{18}, \\
(-\lambda +1)T_{3}T_{13}T_{14}+T_{7}T_{15}-T_{6}T_{16},  &
-T_{3}T_{13}T_{17}+T_{1}T_{15}-T_{10}T_{16}, \\
T_{10}T_{14}+(\mu -1)T_{8}T_{15}-T_{5}T_{18},  &
(-\mu +1)T_{2}T_{13}T_{15}+T_{9}T_{14}-T_{5}T_{16},\\
-T_{2}T_{13}T_{18}+T_{1}T_{14}-T_{8}T_{16},  &
T_{11}T_{13}-T_{10}T_{14}+T_{8}T_{15}, \\
\lambda T_{3}T_{5}T_{8}-\mu T_{2}T_{6}T_{10}+T_{11}T_{12},  &
-\mu T_{2}T_{10}T_{15}+T_{3}T_{5}T_{17}+T_{9}T_{11}, \\
\lambda T_{3}T_{8}T_{14}-T_{2}T_{6}T_{18}+T_{7}T_{11},  &
\lambda T_{3}T_{5}T_{14}-\mu T_{2}T_{6}T_{15}+T_{4}T_{11},\\
T_{3}T_{8}T_{17}-T_{2}T_{10}T_{18}+T_{1}T_{11}, &
(\lambda \mu -\lambda )T_{3}T_{8}T_{13}+\mu T_{7}T_{10}-T_{12}T_{18}, \\
\lambda T_{3}T_{5}T_{13}+T_{4}T_{10}-T_{12}T_{15},  &
(\lambda \mu -\mu )T_{2}T_{10}T_{13}+\lambda T_{8}T_{9}-T_{12}T_{17}, \\
(-\lambda \mu +\lambda +\mu -1)T_{2}T_{3}T_{13}^2+T_{7}T_{9}-T_{12}T_{16},  &
(\lambda -1)T_{3}T_{5}T_{13}+T_{6}T_{9}-T_{12}T_{15}, \\
\mu T_{2}T_{6}T_{13}+T_{4}T_{8}-T_{12}T_{14}, &
(\lambda \mu -\lambda -\mu )T_{2}T_{6}T_{13}+\lambda T_{5}T_{7}-T_{4}T_{8}+(-\lambda +1)T_{12}T_{14}, \\
(\lambda \mu -\lambda -\mu )T_{3}T_{8}T_{13}+\mu T_{1}T_{6}-T_{12}T_{18},  &
(\lambda \mu -\mu )T_{1}T_{5}+(-\lambda \mu +\lambda +\mu )T_{8}T_{9}-T_{12}T_{17}, \\
(-\lambda \mu +\lambda +\mu )T_{2}T_{3}T_{13}^2+T_{1}T_{4}-T_{12}T_{16} 
\end{array}
$
\\[1ex]
\mbox{\tiny
	$
	\left[
	\begin{array}{rrrrrrrrrrrrrrrrrr}
1 & 1 & 1 & 1 & 1 & 1 & 1 & 1 & 1 & 1 & 1 & 2 & 0 & 0 & 0 & 0 & 0 & 0 \\
0 & -1 & -1 & 0 & 0 & 0 & 0 & 0 & 0 & 0 & -1 & 0 & 1 & 0 & 0 & 0 & 0 & 0  \\
-1 & 0 & -1 & 0 & 0 & 0 & 0 & 0 & -1 & -1 & 0 & -1 & 0 & 1 & 0 & 0 & 0 & 0  \\
-1 & -1 & 0 & 0 & 0 & 0 & -1 & -1 & 0 & 0 & 0 & -1 & 0 & 0 & 1 & 0 & 0 & 0  \\
0 & 0 & 0 & 0 & -1 & -1 & 0 & -1 & 0 & -1 & -1 & -1 & 0 & 0 & 0 & 1 & 0 & 0  \\
0 & 0 & -1 & -1 & 0 & -1 & -1 & 0 & 0 & 0 & 0 & -1 & 0 & 0 & 0 & 0 & 1 & 0  \\
0 & -1 & 0 & -1 & -1 & 0 & 0 & 0 & -1 & 0 & 0 & -1 & 0 & 0 & 0 & 0 & 0 & 1 
\end{array}
	\right]
	$
}
\end{center}
(iii)  The blow-up of $\PP_2$ in $[1,0,0]$, $[0,1,0]$, $[0,0,1]$,
$[1,1,1]$, $[1,\lambda,0]$, and $[1,1,\mu]$
with singularity type 
$A_2$
has the $\ZZ^7$-graded Cox ring
$\KT{17}/I$ with   generators for $I$, where $\mu\ne \lambda\in \KK^*\setminus\{1\}$,
and the degree matrix given by
\begin{center}
\tiny 
%old
% 		$
% 		\begin{array}{ll}
% 		(-\mu +1)T_{3}T_{6}T_{9}T_{13}+T_{8}T_{15}-T_{7}T_{16},&
% -\mu T_{6}T_{9}T_{10}+T_{4}T_{15}-T_{5}T_{16},\\
% -T_{3}T_{9}T_{10}T_{13}+T_{1}T_{15}-T_{2}T_{16},&
% (-\lambda \mu )T_{6}T_{9}T_{10}+T_{13}T_{14}+(\lambda -1)T_{4}T_{15},\\
% (-\mu +1)T_{3}T_{4}T_{6}T_{9}+T_{8}T_{14}+T_{11}T_{16},&
% (-\lambda \mu +\lambda )T_{3}T_{5}T_{6}T_{9}+T_{7}T_{14}+T_{11}T_{15},\\
% -\lambda T_{3}T_{5}T_{9}T_{10}+T_{2}T_{14}-T_{12}T_{15},&
% -T_{3}T_{4}T_{9}T_{10}+T_{1}T_{14}-T_{12}T_{16},\\
% (\lambda -1)T_{3}T_{4}T_{5}+\mu T_{6}T_{12}-T_{14}T_{17},&
% T_{10}T_{11}+(\mu -1)T_{6}T_{12}-T_{14}T_{17},\\
% -\lambda T_{3}T_{5}T_{9}T_{17}+T_{2}T_{11}+T_{7}T_{12},&
% -T_{3}T_{4}T_{9}T_{17}+T_{1}T_{11}+T_{8}T_{12},\\
% (-\mu +1)T_{3}T_{4}T_{13}+\mu T_{8}T_{10}+T_{16}T_{17},&
% (-\mu +1)T_{3}T_{5}T_{13}+\mu T_{7}T_{10}+T_{15}T_{17},\\
% T_{6}T_{9}T_{17}+(\lambda -1)T_{5}T_{8}-T_{11}T_{13},&
% T_{4}T_{7}-\lambda T_{5}T_{8}+T_{11}T_{13},\\
% T_{3}T_{9}T_{13}T_{17}+T_{1}T_{7}-T_{2}T_{8},&
% (\mu -1)T_{2}T_{6}-T_{7}T_{10}-T_{15}T_{17},\\
% (\mu -1)T_{1}T_{6}-T_{8}T_{10}-T_{16}T_{17},&
% -T_{9}T_{10}T_{17}+(\lambda -1)T_{1}T_{5}+T_{12}T_{13},\\
% T_{2}T_{4}-\lambda T_{1}T_{5}-T_{12}T_{13}
% 		\end{array}
% 		$
%new
		$
		\begin{array}{ll}
   (-\mu  + 1)T_{1}T_{2}T_{5}T_{7} - T_{10}T_{14} + T_{11}T_{13},&
    \mu T_{2}T_{5}T_{6} + T_{3}T_{11} - T_{4}T_{14},\\
    T_{1}T_{2}T_{6}T_{7} - T_{9}T_{14} + T_{11}T_{12},&
    \lambda \mu T_{2}T_{5}T_{6} + (\lambda  - 1)T_{3}T_{11} - T_{7}T_{17},\\
    (-\mu  + 1)T_{1}T_{2}T_{3}T_{5} - T_{13}T_{17} + T_{14}T_{16},&
    (-\lambda \mu  + \lambda )T_{1}T_{2}T_{4}T_{5} - T_{10}T_{17} + T_{11}T_{16},\\
    \lambda T_{1}T_{2}T_{4}T_{6} - T_{9}T_{17} + T_{11}T_{15},&
    T_{1}T_{2}T_{3}T_{6} - T_{12}T_{17} + T_{14}T_{15},\\
    (\lambda  - 1)T_{1}T_{3}T_{4} - \mu T_{5}T_{15} + T_{8}T_{17},&
    (-\mu  + 1)T_{5}T_{15} - T_{6}T_{16} + T_{8}T_{17},\\
    -\lambda T_{1}T_{2}T_{4}T_{8} + T_{9}T_{16} - T_{10}T_{15},&
    -T_{1}T_{2}T_{3}T_{8} + T_{12}T_{16} - T_{13}T_{15},\\
    (-\mu  + 1)T_{1}T_{3}T_{7} - \mu T_{6}T_{13} + T_{8}T_{14},&
    (-\mu  + 1)T_{1}T_{4}T_{7} - \mu T_{6}T_{10} + T_{8}T_{11},\\
    T_{2}T_{5}T_{8} + (\lambda  - 1)T_{4}T_{13} - T_{7}T_{16},&
    T_{3}T_{10} - \lambda T_{4}T_{13} + T_{7}T_{16},\\
    T_{1}T_{2}T_{7}T_{8} - T_{9}T_{13} + T_{10}T_{12},&
    (\mu  - 1)T_{5}T_{9} + T_{6}T_{10} - T_{8}T_{11},\\
    (\mu  - 1)T_{5}T_{12} + T_{6}T_{13} - T_{8}T_{14},&
    T_{2}T_{6}T_{8} + (\lambda  - 1)T_{4}T_{12} - T_{7}T_{15},\\
    T_{3}T_{9} - \lambda T_{4}T_{12} + T_{7}T_{15}
		\end{array}
		$
\\[1ex]
\mbox{\tiny
%old
% 	$
% 	\left[
% 		\begin{array}{rrrrrrrrrrrrrrrrr}
% 		0 & 0 & 0 & 0 & 0 & 0 & -1 & -1 & -1 & 1 & -1 & 0 & 0 & 0 & 0 & 0 & 0 \\
% 1 & 1 & 0 & 0 & 0 & -1 & 0 & 0 & 1 & 0 & 0 & 1 & 0 & 0 & 0 & 0 & 0 \\
% 0 & 0 & -2 & 1 & 1 & 0 & 0 & 0 & 1 & 0 & 0 & 0 & 1 & 0 & 0 & 0 & 0  \\
% -1 & -1 & -1 & 1 & 1 & 1 & -1 & -1 & -1 & 1 & 0 & 0 & 0 & 1 & 0 & 0 & 0  \\
% 0 & 1 & 1 & -1 & 0 & 0 & 1 & 0 & 0 & 0 & 0 & 0 & 0 & 0 & 1 & 0 & 0  \\
% 1 & 0 & 1 & 0 & -1 & 0 & 0 & 1 & 0 & 0 & 0 & 0 & 0 & 0 & 0 & 1 & 0  \\
% 0 & 0 & 1 & 0 & 0 & 1 & 0 & 0 & -2 & 1 & 0 & 0 & 0 & 0 & 0 & 0 & 1 
% 		\end{array}
% 	\right]
% 	$
%new
	$
	\left[
		\begin{array}{rrrrrrrrrrrrrrrrr}
 1 & 1 & 0 & 0 & 0 & 0 & 0 & 0 & 1 & 1 & 1 & 1 & 1 & 1 & 1 & 1 & 1 \\
-1 & 0 & 1 & 0 & 0 & 0 & 0 & 0 &-1 &-1 &-1 & 0 & 0 & 0 & 0 & 0 & 0 \\
-1 & 0 & 0 & 1 & 0 & 0 & 0 & 0 & 0 & 0 & 0 &-1 &-1 &-1 & 0 & 0 & 0 \\
 0 &-1 & 0 & 0 & 1 & 0 & 0 & 0 &-1 & 0 & 0 &-1 & 0 & 0 &-1 & 0 & 0 \\
 0 &-1 & 0 & 0 & 0 & 1 & 0 & 0 & 0 &-1 & 0 & 0 &-1 & 0 & 0 &-1 & 0 \\
-1 & 0 & 0 & 0 & 0 & 0 & 1 & 0 & 0 & 0 & 0 & 0 & 0 & 0 &-1 &-1 &-1 \\
 0 &-1 & 0 & 0 & 0 & 0 & 0 & 1 & 0 & 0 &-1 & 0 & 0 &-1 & 0 & 0 &-1
		\end{array}
	\right]
	$
}
\end{center}
(iv) The blow-up of $\PP_2$ in $[1,0,0]$, $[0,1,0]$, $[0,0,1]$,
$[1,1,0]$, $[1,0,1]$, and $[0,1,\lambda]$
with singularity type 
$3A_1$
has the $\ZZ^7$-graded Cox ring
$\KT{15}/I$ with generators for $I$, where $\lambda\in \KK^*\setminus\{-1\}$,
and the degree matrix given by
\begin{center}
\tiny 
%old
    %     	$
    %     	\begin{array}{ll}
    %     	-T_{3}T_{5}T_{8}+\lambda T_{2}T_{6}T_{11}+T_{12}T_{15},\,&
    % \lambda T_{1}T_{6}T_{9}-T_{3}T_{4}T_{13}+T_{10}T_{14},\,\\
    % (-\lambda -{1})T_{1}T_{5}T_{6}+T_{11}T_{13}-T_{8}T_{14},\,&
    % -T_{1}T_{5}T_{9}-T_{2}T_{4}T_{14}+T_{7}T_{13},\,\\
    % -\lambda T_{1}T_{2}T_{6}^{2}+T_{3}T_{8}T_{13}+T_{10}T_{12},\,&
    % T_{3}T_{5}T_{13}+T_{2}T_{6}T_{14}+T_{9}T_{12},\,\\
    % T_{1}T_{3}T_{5}^{2}+T_{2}T_{11}T_{14}+T_{7}T_{12},\,&
    % \lambda T_{1}T_{5}T_{6}+T_{4}T_{12}+T_{8}T_{14},\,\\
    % -T_{3}T_{4}T_{8}+T_{1}T_{6}T_{15}+T_{10}T_{11},\,&
    % (-\lambda -{1})T_{3}T_{4}T_{5}+\lambda T_{9}T_{11}-T_{14}T_{15},\,\\
    % -T_{2}T_{3}T_{4}^{2}+T_{1}T_{9}T_{15}+T_{7}T_{10},\,&
    % -\lambda T_{2}T_{4}T_{6}+T_{5}T_{10}+T_{13}T_{15},\,\\
    % \lambda T_{8}T_{9}+(-\lambda -{1})T_{5}T_{10}-T_{13}T_{15},\,&
    % -T_{2}T_{4}T_{11}+T_{1}T_{5}T_{15}+T_{7}T_{8},\,\\
    % (\lambda +{1})T_{6}T_{7}-T_{9}T_{11}-T_{14}T_{15}
    %     	\end{array}
    %     	$
%new
		$
		\begin{array}{ll}
    -T_{1}T_{5}T_{7} + \lambda T_{2}T_{6}T_{8} - T_{9}T_{10},\,&
    T_{1}T_{4}T_{15} + \lambda T_{3}T_{6}T_{13} - T_{11}T_{14},\,\\
    (-\lambda  - 1)T_{3}T_{5}T_{6} + T_{7}T_{14} + T_{8}T_{15},\,&
    -T_{2}T_{4}T_{14} - T_{3}T_{5}T_{13} + T_{12}T_{15},\,\\
    T_{1}T_{7}T_{15} - \lambda T_{2}T_{3}T_{6}^2 + T_{10}T_{11},\,&
    -T_{1}T_{5}T_{15} + T_{2}T_{6}T_{14} - T_{10}T_{13},\,\\
    -T_{1}T_{3}T_{5}^2 + T_{2}T_{8}T_{14} - T_{10}T_{12},\,&
    \lambda T_{3}T_{5}T_{6} - T_{4}T_{10} - T_{7}T_{14},\,\\
    -T_{1}T_{4}T_{7} + T_{3}T_{6}T_{9} - T_{8}T_{11},\,&
    (\lambda  + 1)T_{1}T_{4}T_{5} + \lambda T_{8}T_{13} - T_{9}T_{14},\,\\
    T_{1}T_{2}T_{4}^2 + T_{3}T_{9}T_{13} - T_{11}T_{12},\,&
    -\lambda T_{2}T_{4}T_{6} - T_{5}T_{11} + T_{9}T_{15},\,\\
    (\lambda  + 1)T_{5}T_{11} - \lambda T_{7}T_{13} - T_{9}T_{15},\,&
    -T_{2}T_{4}T_{8} + T_{3}T_{5}T_{9} - T_{7}T_{12},\,\\
    (\lambda  + 1)T_{6}T_{12} - T_{8}T_{13} - T_{9}T_{14}
		\end{array}
		$
\\[1ex]
\mbox{\tiny
%old
% 	$
% 	\left[
% 		\begin{array}{rrrrrrrrrrrrrrr}
% -1&-1&0&0&0&1&-1&0&0&0&0&0&0&0&0 \\
% 1&0&1&0&-1&0&0&0&0&1&0&0&0&0&0 \\
% 1&-1&-1&1&0&0&0&1&-1&0&1&0&0&0&0 \\
% 0&1&1&-1&0&0&0&0&0&0&0&1&0&0&0 \\
% 0&0&-2&1&1&0&0&1&0&0&0&0&1&0&0 \\
% -1&-1&1&0&0&1&0&-1&1&0&0&0&0&1&0 \\
% -2&0&0&0&1&1&0&0&1&0&0&0&0&0&1 
% 		\end{array}
% 	\right]
% 	$
	$
	\left[
		\begin{array}{rrrrrrrrrrrrrrr}
1 & 1 & 1 & 0 & 0 & 0 & 0 & 0 & 0 & 1 & 1 & 1 & 1 & 1 & 1 \\
-1&-1 & 0 & 1 & 0 & 0 & 0 & 0 & 0 &-1 & 0 & 0 & 0 & 0 & 0 \\
-1& 0 &-1 & 0 & 1 & 0 & 0 & 0 & 0 & 0 &-1 & 0 & 0 & 0 & 0 \\
0 &-1 &-1 & 0 & 0 & 1 & 0 & 0 & 0 & 0 & 0 &-1 & 0 & 0 & 0 \\
-1& 0 & 0 & 0 & 0 & 0 & 1 & 0 & 0 & 0 & 0 &-1 &-1 &-1 & 0 \\
0 &-1 & 0 & 0 & 0 & 0 & 0 & 1 & 0 & 0 &-1 & 0 &-1 & 0 &-1 \\
0 & 0 &-1 & 0 & 0 & 0 & 0 & 0 & 1 &-1 & 0 & 0 & 0 &-1 &-1
		\end{array}
	\right]
	$
}
\end{center}
(v) The blow-up of $\PP_2$ in $[1,0,0]$, $[0,1,0]$, $[0,0,1]$,
$[1,1,0]$, $[1,0,1]$, and $[1,1,1]$
with singularity type 
$4A_1$
has the $\ZZ^7$-graded Cox ring
$\KT{13}/I$ with generators for $I$
and the degree matrix given by
\begin{center}
\tiny 
% old
% $
% \begin{array}{ll}
% 	-T_{3}T_{4}T_{5}-T_{6}T_{9}T_{10}+T_{12}T_{13},\, &
% 	-T_{2}T_{4}T_{6}-T_{5}T_{8}T_{10}+T_{11}T_{12},\, \\
% 	-T_{3}T_{5}^{2}T_{8}+T_{2}T_{6}^{2}T_{9}+T_{7}T_{12},\, &
% 	-T_{2}T_{3}T_{4}^{2}+T_{8}T_{9}T_{10}^{2}+T_{1}T_{12},\,\\
% 	-T_{3}T_{5}T_{11}+T_{2}T_{6}T_{13}+T_{7}T_{10},\, &
% 	T_{6}T_{9}T_{11}-T_{5}T_{8}T_{13}+T_{4}T_{7},\,\\
% 	T_{3}T_{9}T_{11}^{2}-T_{2}T_{8}T_{13}^{2}+T_{1}T_{7},\, &
% 	-T_{3}T_{4}T_{11}+T_{8}T_{10}T_{13}+T_{1}T_{6},\, \\
% 	T_{9}T_{10}T_{11}-T_{2}T_{4}T_{13}+T_{1}T_{5}
% \end{array}
% $
% new
$
\begin{array}{ll}
	-T_{3}T_{7}T_{8}+T_{4}T_{6}T_{9}+T_{5}T_{12},\,&
	-T_{2}T_{5}T_{9}+T_{3}T_{10}T_{8}+T_{6}T_{11},\,\\
	-T_{1}T_{5}T_{8}+T_{4}T_{10}T_{9}+T_{7}T_{11},\,&
	-T_{2}T_{5}T_{7}-T_{4}T_{6}T_{10}+T_{8}T_{13},\,\\
	-T_{1}T_{5}T_{6}-T_{3}T_{7}T_{10}+T_{9}T_{13},\,&
	T_{1}T_{6}T_{8}-T_{2}T_{7}T_{9}+T_{10}T_{12},\,\\
	T_{1}T_{4}T_{6}^{2}-T_{2}T_{3}T_{7}^{2}+T_{12}T_{13},\, &
	-T_{1}T_{2}T_{5}^{2}+T_{3}T_{4}T_{10}^{2}+T_{11}T_{13},\,\\
	-T_{1}T_{3}T_{8}^{2}+T_{2}T_{4}T_{9}^{2}+T_{11}T_{12}
\end{array}
$
\\[1ex]
\mbox{\tiny
	% $
	% \left[
        %   %old
	% 	\begin{array}{rrrrrrrrrrrrr}
	% 		0 & 1 & 1 & -1 & 0 & 0 & 1 & 0 & 0 & 0 & 0 & 0 & 0\\
	% 		-1 & 0 & -1 & 0 & 1 & 0 & 0 & -1 & 0 & 0 & 0 & 0 & 0\\
	% 		-1 & -1 & 0 & 0 & 0 & 1 & 0 & 0 & -1 & 0 & 0 & 0 & 0\\
	% 		2 & 2 & 2 & -1 & -1 & -1 & 0 & 0 & 0 & 1 & 0 & 0 & 0\\
	% 		0 & 0 & -2 & 1 & 1 & 0 & 0 & 0 & 0 & 0 & 1 & 0 & 0\\
	% 		-1 & -1 & -1 & 1 & 1 & 1 & 0 & 0 & 0 & 0 & 0 & 1 & 0\\
	% 		0 & -2 & 0 & 1 & 0 & 1 & 0 & 0 & 0 & 0 & 0 & 0 & 1
	% 	\end{array}
	% \right]
        % $
        $
	\left[
          %new
		\begin{array}{rrrrrrrrrrrrr}
			 1 & 1 &  1 & 1 & 0 & 0 & 0 & 0 & 0 & 0 & 1 & 1 & 1 \\
			-1 & -1 & 0 & 0 & 1 & 0 & 0 & 0 & 0 & 0 & 0 & -1 & 0\\
			-1 & 0 & 0 & -1 & 0 & 1 & 0 & 0 & 0 & 0 & -1 & 0 & 0\\
			0 & -1 & -1 & 0 & 0 & 0 & 1 & 0 & 0 & 0 & -1 & 0 & 0\\
			-1 & 0 & -1 & 0 & 0 & 0 & 0 & 1 & 0 & 0 & 0 & 0 & -1\\
			0 & -1 & 0 & -1 & 0 & 0 & 0 & 0 & 1 & 0 & 0 & 0 & -1\\
			0 & 0 & -1 & -1 & 0 & 0 & 0 & 0 & 0 & 1 & 0 & -1 & 0
		\end{array}
	\right]
	$
}
\end{center}
(vi) The minimal desingularization of
$V(x_0x_1x_2-x_0x_2x_3+(\lambda-1)x_1x_2x_3-\lambda
  x_1^2x_2+x_1x_2^2)\subseteq \PP_3$
with singularity type 
$A_2A_1$
has the $\ZZ^7$-graded Cox ring
$\KT{14}/I$ with generators for $I$, where $\lambda\in \KK^*\setminus\{1\}$,
and the degree matrix given by
\begin{center}
\tiny 
      $
     \begin{array}{ll}
T_{1}T_{2}T_{4}T_{6} + T_{10}T_{14} - T_{11}T_{13},\, &
    \lambda T_{1}T_{2}T_{4}T_{7} + T_{9}T_{14} - T_{11}T_{12},\,\\
    T_{1}T_{2}T_{4}T_{8} + T_{9}T_{13} - T_{10}T_{12},\, &
    T_{1}T_{3}T_{4}^2 + T_{6}T_{9} - T_{7}T_{10},\,\\
    T_{6}T_{9} -\lambda T_{7}T_{10} + T_{8}T_{11},\, &
    T_{1}T_{4}T_{5} + T_{6}T_{12} - T_{7}T_{13},\,\\
    T_{6}T_{12} - \lambda T_{7}T_{13} + T_{8}T_{14},\, &
    -T_{3}T_{4}T_{12} + T_{5}T_{9} + T_{2}T_{7}T_{8},\,\\
    -T_{3}T_{4}T_{13} + T_{5}T_{10} + T_{2}T_{6}T_{8},\, &
    -T_{3}T_{4}T_{14} + T_{5}T_{11} + (\lambda-1)T_{2}T_{6}T_{7}
    \end{array}
    $
\\[1ex]
\mbox{\tiny
	$
	\left[
          \begin{array}{rrrrrrrrrrrrrr}
            1 &  1 &  0 &  0 &  0 &  0 &  0 &  0 &  1 &  1 &  1 &  1 &  1 &  1\\
           -1 &  0 &  1 &  0 &  0 &  0 &  0 &  0 &  0 &  0 &  0 & -1 & -1 & -1\\
           -1 &  0 & -1 &  1 &  0 &  0 &  0 &  0 &  0 &  0 &  0 &  0 &  0 &  0\\
           -1 &  0 &  0 &  0 &  1 &  0 &  0 &  0 & -1 & -1 & -1 &  0 &  0 &  0\\
            0 & -1 &  0 &  0 &  0 &  1 &  0 &  0 & -1 &  0 &  0 & -1 &  0 &  0\\
            0 & -1 &  0 &  0 &  0 &  0 &  1 &  0 &  0 & -1 &  0 &  0 & -1 &  0\\
            0 & -1 &  0 &  0 &  0 &  0 &  0 &  1 &  0 &  0 & -1 &  0 &  0 & -1
          \end{array}
	\right]
	$
}
\end{center}
(vii) The minimal desingularization of
$V(x_0x_3(x_1-x_2)-x_1(x_2-x_3)(x_1-x_2+\lambda x_3))\subseteq \PP_3$
with singularity type 
$A_3$
has the $\ZZ^7$-graded Cox ring
$\KT{13}/I$ with generators for $I$, where $\lambda \in \KK^*$,
and the degree matrix given by
\begin{center}
\tiny 
  $
  \begin{array}{ll}
T_{1}T_{2}^2T_{3}T_{4}^2 - T_{10}T_{13} + T_{11}T_{12},\,&
T_{1}T_{2}T_{4}T_{5} - T_{7}T_{12} + T_{8}T_{13},\,\\
T_{1}T_{2}T_{4}T_{6} - T_{7}T_{10} + T_{8}T_{11},\,&
T_{2}T_{3}T_{4}T_{7} + T_{5}T_{11} - T_{6}T_{13},\,\\
T_{2}T_{3}T_{4}T_{8} + T_{5}T_{10} - T_{6}T_{12},\,&
\lambda T_{1}T_{5}T_{6} + T_{3}T_{7}T_{8} - T_{4}T_{9}
  \end{array}
  $
  \\[1ex]
\mbox{\tiny
	$
	\left[
          \begin{array}{rrrrrrrrrrrrr}
            1 &  0 &  1 &  0 &  0 &  0 &  0 &  0 &  1 &  1 &  1 &  1 &  1\\
           -1 &  1 & -1 &  0 &  0 &  0 &  0 &  0 & -1 &  0 &  0 &  0 &  0\\
            0 & -1 &  0 &  1 &  0 &  0 &  0 &  0 & -1 &  0 &  0 &  0 &  0\\
           -1 &  0 &  0 &  0 &  1 &  0 &  0 &  0 &  0 & -1 & -1 &  0 &  0\\
           -1 &  0 &  0 &  0 &  0 &  1 &  0 &  0 &  0 &  0 &  0 & -1 & -1\\
            0 &  0 & -1 &  0 &  0 &  0 &  1 &  0 &  0 & -1 &  0 & -1 &  0\\
            0 &  0 & -1 &  0 &  0 &  0 &  0 &  1 &  0 &  0 & -1 &  0 & -1
          \end{array}
	\right]
	$
}
\end{center}
(viii) The minimal desingularization of
$V(x_0x_1x_2+x_3^2(x_1+x_2+x_3))\subseteq \PP_3$
with singularity type 
$A_22A_1$
has the $\ZZ^7$-graded Cox ring
$\KT{12}/I$ with generators for $I$
and the degree matrix given by
\begin{center}
\tiny 
		$
		\begin{array}{ll}
		-T_{5}T_{9} + T_{3}T_{6}T_{8} +
    T_{2}T_{4}T_{7}^2,\,&
    T_{5}T_{11} - T_{6}T_{10} +
    T_{1}T_{2}T_{7}T_{8},\,\\
    - T_{5}T_{12} + T_{4}T_{7}T_{10} +
    T_{1}T_{3}T_{8}^2,\,&
    -T_{6}T_{12} + T_{4}T_{7}T_{11} +
    T_{1}T_{8}T_{9},\,\\
    T_{2}T_{7}T_{12} + T_{3}T_{8}T_{11} -
    T_{9}T_{10}
		\end{array}
		$
\\[1ex]		
\mbox{\tiny
	$
	\left[
          \begin{array}{rrrrrrrrrrrr}
            1 &  1 &  0 &  0 &  1 &  1 &  0 &  0 &  0 &  1 &  1 &  0\\
           -1 &  0 &  1 &  0 &  0 & -1 &  0 &  0 &  0 &  0 & -1 &  0\\
           -1 &  0 & -1 &  0 &  0 &  0 &  0 &  1 &  0 &  0 &  0 &  0\\
            0 & -1 &  0 &  1 &  0 &  0 &  0 &  0 &  0 & -1 & -1 &  0\\
            0 & -1 &  0 & -1 &  0 &  0 &  1 &  0 &  0 &  0 &  0 &  0\\
           -1 &  0 &  0 &  0 & -1 &  0 &  0 &  0 &  1 & -1 &  0 &  0\\
            0 & -1 &  0 &  0 & -1 & -1 &  0 &  0 &  0 &  0 &  0 &  1
          \end{array}
	\right]
	$
}
\end{center}
(ix) The minimal desingularization of
$V(x_0x_1x_3-x_2(x_1+x_2)(x_1+x_3))\subseteq \PP_3$
with singularity type 
$A_3A_1$
has the $\ZZ^7$-graded Cox ring
$\KT{11}/I$ with generators for $I$
and the degree matrix given by
\begin{center}
\tiny 
$
		\begin{array}{ll}
T_1T_4T_8^2+T_3T_6T_{11}-T_7T_{10},\, &
    T_1T_2T_7T_8+T_5T_{11}-T_6T_9
		\end{array}
$
\\[1ex]
    \mbox{\tiny $
     \left[
          \begin{array}{rrrrrrrrrrr}
            1 &  0 &  1 &  0 &  1 &  0 &  0 &  0 &  1 &  1 &  0\\
           -1 &  1 & -1 &  0 &  0 &  0 &  0 &  0 &  0 & -1 &  0\\
            0 & -1 &  0 &  0 &  0 &  0 &  1 &  0 &  0 & -1 &  0\\
           -1 &  0 &  0 &  1 & -1 &  0 &  0 &  0 & -1 &  0 &  0\\
           -1 &  0 &  0 & -1 &  0 &  0 &  0 &  1 &  0 &  0 &  0\\
            0 &  0 & -1 &  0 & -1 &  0 &  0 &  0 &  0 &  0 &  1\\
            0 &  0 & -1 &  0 &  0 &  1 &  0 &  0 & -1 &  0 &  0
          \end{array}
    \right]
    $}
\end{center}
(x) The minimal desingularization of
$V(x_0x_1(x_2-\lambda x_3)-x_2x_3(x_2-x_3)=0)\subseteq \PP_3$
with singularity type 
$2A_2$
has the $\ZZ^7$-graded Cox ring
$\KT{11}/I$ with generators for $I$, where $\lambda\in \KK^*\setminus\{1\}$,
and the degree matrix given by
\begin{center}
\tiny 
		$
		\begin{array}{ll}
    T_{2}T_{4}T_{5}^2 + \lambda T_{6}T_{9} - T_{7}T_{10},\,&
    T_{6}T_{9} - T_{7}T_{10} + T_{8}T_{11}
		\end{array}
		$
\\[1ex]
\mbox{\tiny
	$
	\left[
          \begin{array}{rrrrrrrrrrr}
            0 &  0 &  1 &  1 &  0 &  1 &  1 &  1 &  0 &  0 &  0\\
            1 &  0 &  0 & -1 &  0 & -1 & -1 & -1 &  0 &  0 &  0\\
           -1 &  1 &  0 & -1 &  0 &  0 &  0 &  0 &  0 &  0 &  0\\
            0 & -1 &  0 & -1 &  1 &  0 &  0 &  0 &  0 &  0 &  0\\
            0 &  0 & -1 &  0 &  0 & -1 &  0 &  0 &  1 &  0 &  0\\
            0 &  0 & -1 &  0 &  0 &  0 & -1 &  0 &  0 &  1 &  0\\
            0 &  0 & -1 &  0 &  0 &  0 &  0 & -1 &  0 &  0 &  1
          \end{array}
	\right]
	$
}
\end{center}
(xi) The minimal desingularization of
$V(x_0x_1x_3-x_1^2x_2+x_2^2x_3-x_2x_3^2)\subseteq \PP_3$
with singularity type 
$A_4$
has the $\ZZ^7$-graded Cox ring
$\KT{12}/I$ with generators for $I$
and the degree matrix given by
\begin{center}
\tiny 
	$
	\begin{array}{ll}
	T_{1}T_{2}^2T_{3}T_{5}^2 - T_{7}T_{9} + T_{8}T_{10},\,&
    T_{1}^2T_{2}^2T_{3}T_{5}T_{6}^2
    + T_{7}T_{12} - T_{8}T_{11},\,\\
    T_{5}T_{11} + T_{1}T_{6}^2T_{10} -
    T_{3}T_{4}^2T_{7}^2T_{8},\,&
    T_{5}T_{12} + T_{1}T_{6}^2T_{9} -
    T_{3}T_{4}^2T_{7}T_{8}^2,\\
    T_{1}T_{2}^2T_{3}^2T_{4}^2T_{5}T_{7}T_{8}
    - T_{9}T_{11} + T_{10}T_{12}.
	\end{array}
	$
\\[1ex]
    \mbox{\tiny $
    \left[
          \begin{array}{rrrrrrrrrrrr}
            1 &  0 &  0 &  1 &  0 &  0 &  0 &  0 &  1 &  1 &  2 &  2\\
           -1 &  0 &  1 & -1 &  0 &  0 &  0 &  0 &  0 &  0 & -1 & -1\\
           -1 &  1 & -1 &  0 &  0 &  0 &  0 &  0 &  0 &  0 & -1 & -1\\
            0 & -1 &  0 &  0 &  1 &  0 &  0 &  0 &  0 &  0 & -1 & -1\\
           -1 &  0 &  0 &  0 &  0 &  1 &  0 &  0 & -1 & -1 &  0 &  0\\
            0 &  0 &  0 & -1 &  0 &  0 &  1 &  0 & -1 &  0 &  0 & -1\\
            0 &  0 &  0 & -1 &  0 &  0 &  0 &  1 &  0 & -1 & -1 &  0
          \end{array}
    \right]
    $}
\end{center}
(xvi) The minimal desingularization of
$V((x_0x_1+x_2x_3)(x_2+x_3)+x_1^3)\subseteq \PP_3$ with singularity
type $A_5$ has the $\ZZ^7$-graded Cox ring $\KT{13}/I$ with generators
for $I$ and the degree matrix given by
\begin{center}
\tiny 
    $
    \begin{array}{ll}
    T_{6}T_{10} + T_{8}T_9 -
    T_{1}T_{2}^3T_{3}^2T_{4}T_{7}^3,\,&
    T_{1}T_9^2 - T_{3}T_{4}^2T_{5}^3T_{6}^3T_{8} +
    T_{7}T_{11},\,\\
    T_{1}T_{10}^2 - T_{3}T_{4}^2T_{5}^3T_{6}T_{8}^3 +
    T_{7}T_{12},\,&
    T_{1}^2T_{2}^3T_{3}^2T_{4}T_{7}^2T_9 +
    T_{8}T_{11} - T_{6}T_{13},\,\\
    T_{1}^2T_{2}^3T_{3}^2T_{4}T_{7}^2T_{10} +
    T_{6}T_{12} - T_{8}T_{13},\,&
    T_{1}T_9T_{10} +
    T_{3}T_{4}^2T_{5}^3T_{6}^2T_{8}^2 - T_{7}T_{13},\,\\
    T_1^3T_2^6T_3^5T_4^4T_5^3T_6T_7^4T_8+T_{11}T_{12} -
    T_{13}^2,\,&
    T_9T_{13}+T_{10}T_{11} -
    T_1T_2^3T_3^3T_4^3T_5^3T_6^2T_7^2T_8,\,\\
    T_{10}T_{13} + T_9T_{12} -
    T_1T_2^3T_3^3T_4^3T_5^3T_6T_7^2T_8^2
   \end{array}
   $
\\[1ex]
\mbox{\tiny
	$
\left[
    \mbox{\tiny $
          \begin{array}{rrrrrrrrrrrrr}
            1 &  0 &  0 &  0 &  1 &  0 &  0 &  0 &  1 &  1 &  3 &  3 &  3\\
           -1 &  0 &  0 &  1 & -1 &  0 &  0 &  0 &  0 &  0 & -1 & -1 & -1\\
           -1 &  0 &  1 & -1 &  0 &  0 &  0 &  0 &  0 &  0 & -1 & -1 & -1\\
           -1 &  1 & -1 &  0 &  0 &  0 &  0 &  0 &  0 &  0 & -1 & -1 & -1\\
            0 & -1 &  0 &  0 &  0 &  0 &  1 &  0 &  0 &  0 & -1 & -1 & -1\\
            0 &  0 &  0 &  0 & -1 &  1 &  0 &  0 &  0 & -1 &  0 & -2 & -1\\
            0 &  0 &  0 &  0 & -1 &  0 &  0 &  1 & -1 &  0 & -2 &  0 & -1
          \end{array}
    $}
    \right]
	$
}
\end{center}
\end{theorem}

\begin{proof}
  Let $Y$ be a cubic surface with at most rational double points as
  singularities. Its minimal desingularization $X$ is the blow-up
  $\pi\colon X \to \PP_2$ in six points in \emph{almost general position},
  i.e., the map $\pi$ is a composition
  \begin{equation*}
    X\,=\,X_6
    \ \xrightarrow{\pi_6}\
    X_5
    \ \xrightarrow{\pi_5}\ 
    \ \ldots\ 
    \xrightarrow{\pi_1}\ 
    X_0 \,=\, \PP_2
  \end{equation*}
  of six blow-ups $\pi_i\colon X_i \to X_{i-1}$, where the blown-up point
  $p_i \in X_{i-1}$ does not lie on a $(-2)$-curve on $X_{i-1}$; see
  \cite[Theorem~8.1.15]{Do}.

  Our starting point is the classification of such $X$ via the sub-root
  systems of the root system $R(E_6)$ of type $E_6$ in $\Pic(X)$; see
  \cite[Appendix, Theorem~3.12]{Manin} and the references before
  \cite[Appendix, Definition~3.6]{Manin}, for example. A basis for $\Pic(X)$
  is given by $\ell_0, \dots, \ell_6$, where $\ell_0$ is the class of
  $\pi^*\OO_{\PP_2}(1)$ and $\ell_i$ is the class of the total transform of
  the exceptional divisor of $\pi_i$;
  the anticanonical class is $-K_X = 3\ell_0-\ell_1-\dots-\ell_6$. Then
  \begin{equation*}
    R(E_6) = \{L \in \Pic(X) \mid L.L=-2,\ L.(-K_X)=0\}.
  \end{equation*}

  By \cite[p.~480, (iv)]{coxeter} (see also \cite[Table]{Urabe}), every
  sub-root system or $R(E_6)$ has one of the types listed in the second column
  of Table~\ref{tab:cubic}, and conversely, up to the action of the Weyl group
  associated to $R(E_6)$, there is a unique sub-root system of each type.
  Given a sub-root system $R \subset R(E_6)$ with simple roots $R_i$, there is
  an $X$ such that the classes of its $(-2)$-curves are precisely $R_i$ and
  the classes of its $(-1)$-curves are precisely
  \begin{equation*}
    \{L' \in \Pic(X) \mid L'.L'=-1,\ L.'(-K_X)=1,\ L'.L \ge 0\text{ for all }L \in R_i\}.
  \end{equation*}
  The configuration of negative curves on $X_6=X$ can be encoded in an
  \emph{extended Dynkin diagram}, where each vertex corresponds to a negative
  curve $L$, with $L.L'$ edges between vertices $L,L'$. Now we can produce the
  extended Dynkin diagrams of $X_5, \dots, X_0=\PP_2$ as follows: Given that
  of $X_i$, let $L$ be the class of a $(-1)$-curve to be contracted; delete
  the vertex $L$, increase the self-intersection numbers of the other vertices
  $L'$ by $(L.L')^2$ and add $(L.L')\cdot (L.L'')$ edges between the vertices
  $L'$ and $L''$. Finally, this leads to the extended Dynkin diagram of
  $X_0=\PP_2$, where we can read off the degrees of the images of the negative
  curves of $X$, and we can extract the configuration of the points $p_1,
  \dots, p_6$ that are blown up (as points in $\PP_2$ or as intersection
  points of exceptional divisors of $\pi_i$ with the strict transforms of
  curves in $\PP_2$). See \cite[Section~2.1]{Der} for more details.  The
  resulting classification can also be found in~\cite[Section~9.2.2]{Do}.

  Below, we describe the resulting configuration of blown-up points for each
  type whose Cox ring we want to determine.  For some sub-root systems $R
  \subset R(E_6)$, the corresponding surface $X$ is unique up to automorphism,
  and hence the configuration of blown-up points can be uniquely described up
  to isomorphism of $\PP_2$. For other $R$, there is a family of isomorphism
  classes of surfaces $X$ (see \cite{BW}), and some of the points involve
  parameters. Note that \tt{compcox.lib} can deal with parameters, see
  Remark~\ref{ex:parameters}.  The source code used to compute the listed Cox
  rings can be obtained from the homepage~\cite{DeHaHeKeLa} of our package.
  See also Section~\ref{sec:example} for the code used to obtain the Cox ring
  of the resolution of the singular cubic surface with singularity type
  $A_4$. The degree matrices are given in terms of our basis $\ell_0, \dots,
  \ell_6$ of $\Pic(X)$.

 Cases (i)--(v) with singularity types $A_1$, $2A_1$, $A_2$, $3A_1$,
 and $4A_1$ can be realized as blow-ups of six distinct points in
 $\PP_2$; the configurations of points are given in the statement of
 the theorem. Then the Cox rings are directly obtained by an
 application of the procedure \texttt{blowupCEMDSpoints}
 of our package \texttt{compcox.lib}. 
 %The degree matrices are
 % given in terms of a basis of $\Pic(X)$ chosen by our algorithms.

 For the remaining cases, at least one of the blown-up points lies on
 an exceptional divisor of a previous blow-up. In the following, we
 give a precise description of the blown-up points in each case. Since
 these do not give a concise description of the resulting surfaces, we
 provide instead a description of the anticanonical map $X \to Y
 \subset \PP_3$. This leads to a cubic form defining $Y$, as listed in
 the statement of the theorem. Again, we obtain the Cox rings from the
 sequence of blow-ups using \texttt{blowupCEMDSpoints} from
 \texttt{compcox.lib}.
% For these cases, the degrees of the generators
% of the Cox rings are given in terms of our basis $\ell_0, \dots,
%   \ell_6$ of $\Pic(X)$.
 We denote the coordinates on $\PP_2$ by $y_0,y_1,y_2$.

{\em Case $A_2A_1$:\/} The blown-up points are $p_1=[0,0,1]$, the
intersection $p_2$ of the first exceptional divisor with the strict
transform of $\{y_0-y_1=0\}$, and the preimages $p_3, \dots, p_6$ of
$[1,1,1]$, $[1,0,0]$, $[0,1,0]$, $[1,\lambda,0]$. 
The monomials
\begin{equation*}
(T_{3}T_{11}T_{12}T_{13},
   T_{1}T_{2}T_{3}T_{4}T_{7}T_{13},
   T_{1}T_{2}T_{3}T_{4}T_{6}T_{12},
   T_{1}T_{2}^2T_{6}T_{7}T_{8})
\end{equation*}
have degree $(3,-1,-1,-1,-1,-1,-1)$, which is
the anticanonical degree; in view of the relations in
  the Cox ring, they are linearly independent. Hence they form a basis of $H^0(X,-K_X)$. This
choice leads to the given cubic equation of the anticanonical image under
$X \to Y \subset \PP_3$.

{\em Case $A_3$:\/} The first blown-up point is $p_1=[1,1,0] \in
\PP_2$. The second blown-up point is the intersection of the first
exceptional divisor with the strict transform of $\{y_0-y_1+\lambda
y_2=0\}$. The further blown-up points $p_3, \dots, p_6$ are the
preimages of $[1,0,0]$, $[0,1,0]$, $[0,0,1]$, $[1,1,1]$ in $\PP_2$.
Choosing the basis 
\begin{equation*}
  (T_{9}T_{11}T_{12},T_{1}T_{2}T_{3}T_{6}T_{7}T_{12}
   ,T_{1}T_{2}T_{3}T_{5}T_{7}T_{10}
   ,T_{1}^2T_{2}^2T_{3}T_{4}T_{5}T_{6})
\end{equation*}
of $H^0(X,-K_X)$ leads to the given cubic equation of the
anticanonical image under $X \to Y \subset \PP_3$.

{\em Case $A_22A_1$:\/} The blown-up points are $p_1=[1,0,0]$, the
intersection $p_2$ of the first exceptional divisor with the strict
transform of $\{y_1=0\}$, the preimage $p_3$ of $[0,1,0]$, the
intersection $p_4$ of the third exceptional divisor with the strict
transform of $\{y_0=0\}$, and the preimages $p_5, p_6$ of $[1,0,-1]$,
$[0,1,-1]$. Choosing the basis 
\begin{equation*}
  (T_{3}T_{4}T_{5}T_{11}^2,
    T_{1}T_{2}^2T_{4}T_{7}^2T_{12},
    T_{1}^2T_{2}T_{3}T_{8}^2T_{9},
    T_{1}T_{2}T_{3}T_{4}T_{7}T_{8}T_{11})
\end{equation*}
of $H^0(X,-K_X)$ leads to the given cubic equation of the
anticanonical image under $X \to Y \subset \PP_3$.

{\em Case $A_3A_1$:\/} The blown-up points are $p_1=[0,1,0]$, the
intersection $p_2$ of the first exceptional divisor with the strict
transform of $\{y_0+y_2=0\}$, the preimage $p_3$ of $[0,0,1]$, the
intersection $p_4$ of the third exceptional divisor with the strict
transform of $\{y_0=0\}$, and the preimages $p_5, p_6$ of $[1,0,0]$,
$[1,-1,0]$. Choosing the basis 
\begin{equation*}
  (T_4T_5T_9T_{10}, T_1^2T_2^2T_3T_4T_7T_8^2, T_1T_2T_3T_4T_5T_8T_{11}, T_1T_2^2T_3^2T_6T_7T_{11})
\end{equation*}
of $H^0(X,-K_X)$ leads to the given cubic equation of the
anticanonical image under $X \to Y \subset \PP_3$.

{\em Case $2A_2$:\/} The blown-up points are $p_1=[1,0,0]$, the
intersection $p_2$ of the first exceptional divisor with the strict
transform of $\{y_1-\lambda y_2=0\}$, the intersection $p_3$ of the
second exceptional divisor with the strict transform of $\{y_1-\lambda
y_2=0\}$, and the preimages $p_4, p_5, p_6$ of $[0,1,0]$, $[0,0,1]$,
$[0,1,-1]$. Choosing the basis 
\begin{equation*}
  (T_{3}^2T_{4}T_{9}T_{10}T_{11},T_{1}^2T_{2}T_{6}T_{7}T_{8},
  T_{1}T_{2}T_{3}T_{4}T_{5}T_{7}T_{10},
  T_{1}T_{2}T_{3}T_{4}T_{5}T_{6}T_{9})
\end{equation*}
of $H^0(X,-K_X)$ leads to the given cubic equation of the
anticanonical image under $X \to Y \subset \PP_3$.

{\em Case $A_4$:\/} The blown-up points are $p_1=[0,1,0]$, the
intersection $p_2$ of the first exceptional divisor with the strict
transform of $\{y_2=0\}$, the intersection $p_3$ of the second
exceptional divisor with the strict transform of $\{y_0^2-y_1y_2=0\}$,
and the preimages $p_4, p_5, p_6$ of $[1,0,0]$, $[0,0,1]$,
$[0,1,1]$. Choosing the basis
\begin{equation*}
  (T_{9}T_{11},
    T_{1}T_{2}^2T_{3}^2T_{4}^2T_{5}T_{7}T_{8},
    T_{1}T_{2}T_{3}T_{4}T_{6}T_{7}T_{9},
    T_{1}^2T_{2}^3T_{3}^2T_{4}T_{5}^2T_{6})
\end{equation*}
of $H^0(X,-K_X)$ leads to the given cubic equation of the
anticanonical image under $X \to Y \subset \PP_3$.

{\em Case $A_5$:\/} The blown-up points are $p_1 = [0,1,-1]$,
the
intersection $p_2$ of the first exceptional divisor with the strict
transform of $\{y_1+y_2=0\}$, the intersection $p_3$ of the second
exceptional divisor with the strict transform of $\{y_1+y_2=0\}$, 
the
intersection $p_4$ of the third exceptional divisor with the strict
transform of $\{y_0^3-(y_1+y_2)y_2^2=0\}$, 
and the preimages $p_5,
p_6$ of $[0,0,1]$, $[0,1,0]$. Choosing the basis
\begin{equation*}
  (T_{13},
  T_{1}T_{2}^2T_{3}^2T_{4}^2T_{5}^2T_{6}T_{7}T_{8},
  T_{1}T_{2}T_{3}T_{4}T_{5}T_{6}T_{10},
  T_{1}T_{2}T_{3}T_{4}T_{5}T_{8}T_9)
\end{equation*}
of $H^0(X,-K_X)$ leads to the given cubic equation of the
anticanonical image under $X \to Y \subset \PP_3$.
\end{proof}

\begin{remark}
In Theorem~\ref{thm:cubic}, the Cox rings of the singular cubic
surfaces can be obtained by contracting the $(-2)$-curves
with the procedure \tt{contractCEMDS} of our library \tt{compcox.lib}.
As a concrete example, we have treated the $A_4$ case explicitly 
in Step~\ref{ex:contractCEMDS}.
\end{remark}

Now, we consider other rational surfaces arising from the 
projective plane by blowing up six distinct points. If the points are
in \emph{general position} (i.e., no three points lie on a common
line, and not all six lie on a common conic), then this leads to the
smooth cubic surfaces; a description of their Cox rings is given
in~\cite{LaVe}.  When the six distinct points are in almost general
position (i.e., at most three points lie on a common line), then we
obtain minimal resolutions of cubic surfaces with at most rational
double points; see Theorem~\ref{thm:cubic} for their Cox rings. The
following settles the remaining cases.

\begin{theorem}
  Let $X$ be the blow-up of $\PP_2$ at six distinct points
  $p_1,\ldots, p_6$ in non-general position.  For the possible types
  of configurations, indicated by the figures, the following table
  provides the Cox ring of $X$ by specifying generators $T_1,\ldots,
  T_r$, their degrees $w_1, \ldots, w_r \in \Cl(X) \cong \ZZ^7$ and
  generators of the ideal of relations.
The parameters $a,b,c$ are pairwise distinct elements of $\KK^*\setminus\{1\}$.

\begingroup
\tiny 
\begin{longtable}{ll}\hline 
configuration & degree matrix $[w_1, \ldots, w_r]$ and $\Cox(X)$ 
\\ \hline \\ 
\begin{minipage}{2cm}
\footnotesize
\begin{center}
\tiny (i)\\ %6p6col
\begin{tikzpicture}[scale=.2]
\coordinate (100) at (-2,0);
\coordinate (010) at (2,0);
\coordinate (001) at ($(100)!1.0!60:(010)$);
\coordinate (110) at ($(100)!.2!(010)$);
\coordinate (1a0) at ($(100)!.4!(010)$);
\coordinate (1b0) at ($(100)!.6!(010)$);
\coordinate (1c0) at ($(100)!.8!(010)$);
\filldraw[fill=black!25, draw=black] (100)--(010)--(001)--cycle;
\fill (100) circle (.20cm);
\fill (010) circle (.20cm);
\fill (110) circle (.20cm);
\fill (1a0) circle (.20cm);
\fill (1b0) circle (.20cm);
\fill (1c0) circle (.20cm);
\end{tikzpicture}\\[1ex]
\begingroup
\tiny 
$
\begin{array}{rcl}
p_1 &=& [1,0,0],\\
p_2 &=& [0,1,0],\\
p_3 &=& [1,1,0], \\
p_4 &=& [1,a,0],\\
p_5 &=& [1,b,0],\\
p_6 &=& [1,c,0].\\
\end{array}
$
\endgroup
\end{center}
\end{minipage}
&
\begin{minipage}{6cm}
\hspace*{1em}
\mbox{\tiny
$
\left[
\begin{array}{rrrrrrrrrrrrr}
1 & 0 & 0 & 0 & -1 & -1 & 1 & 0 & 0 & 0 & 0 & 0 & 0\\
1 & -1 & 0 & 1 & -1 & 0 & 0 & 0 & 0 & 0 & 0 & 0 & 0\\
-1 & 0 & 0 & 0 & 1 & 0 & 0 & 1 & -1 & 0 & 0 & 0 & 0\\
-1 & 0 & -1 & 0 & 1 & 0 & 0 & 0 & 0 & 0 & 0 & 0 & 0\\
-1 & 0 & 0 & 0 & 1 & 0 & 0 & 0 & 0 & 1 & -1 & 0 & 0\\
0 & 1 & 0 & 0 & 1 & 1 & 0 & 1 & 0 & 1 & 0 & 1 & 0\\
1 & 1 & 0 & 0 & 0 & 1 & 0 & 1 & 0 & 1 & 0 & 0 & 1
\end{array}
\right]
$
}
\vspace{0.15cm} \\
\hspace*{1cm}
\begin{minipage}{5cm}
\arraycolsep=2.2pt
$
\begin{array}{r}
(ab-ac)T_{8}T_{9}+(-a^2+ac)T_{10}T_{11}+(a-b)T_{12}T_{13},\\
(a-b)T_{6}T_{7}+(b-1)T_{8}T_{9}+(-a+1)T_{10}T_{11},\\
(a-b)T_{1}T_{5}+T_{8}T_{9}-T_{10}T_{11},\\
(a-b)T_{2}T_{4}+(b)T_{8}T_{9}-aT_{10}T_{11}
\end{array}
$
\end{minipage}
\end{minipage}
\\ \\ \hline \\ 
\begin{minipage}{1.6cm}
\footnotesize
\begin{center}
\tiny (ii) \\ 
\begin{tikzpicture}[scale=.2]
\coordinate (100) at (-2,0);
\coordinate (010) at (2,0);
\coordinate (001) at ($(100)!1.0!60:(010)$);
\coordinate (110) at ($(100)!.25!(010)$);
\coordinate (1a0) at ($(100)!.5!(010)$);
\coordinate (1b0) at ($(100)!.75!(010)$);
\filldraw[fill=black!25, draw=black] (100)--(010)--(001)--cycle;
\fill (100) circle (.20cm);
\fill (010) circle (.20cm);
\fill (001) circle (.20cm);
\fill (110) circle (.20cm);
\fill (1a0) circle (.20cm);
\fill (1b0) circle (.20cm);
\end{tikzpicture}\\[1ex]
\begingroup
\tiny 
$
\begin{array}{rcl}
p_1 &=& [1,0,0],\\
p_2 &=& [0,1,0],\\
p_3 &=& [0,0,1], \\
p_4 &=& [1,1,0],\\
p_5 &=& [1,a,0],\\
p_6 &=& [1,b,0].\\
\end{array}
$
\endgroup
\end{center}
\end{minipage}
&
\begin{minipage}{7cm}
\hspace*{1em}
\mbox{\tiny
$
\left[
\begin{array}{rrrrrrrrrrrrr}
0 & -1 & 0 & 1 & 0 & 1 & 0 & 0 & 0 & 0 & 0 & 0\\
0 & 1 & 1 & -1 & 0 & 0 & 0 & 0 & 0 & 0 & 0 & 0\\
0 & -1 & 0 & 1 & 0 & 0 & 1 & -1 & 0 & 0 & 0 & 0\\
1 & -1 & 0 & 1 & -1 & 0 & 0 & 0 & 0 & 0 & 0 & 0\\
0 & -1 & 0 & 1 & 0 & 0 & 0 & 0 & 1 & -1 & 0 & 0\\
1 & 0 & 0 & 1 & 0 & 0 & 1 & 0 & 1 & 0 & 1 & 0\\
1 & 1 & 0 & 0 & 0 & 0 & 1 & 0 & 1 & 0 & 0 & 1
\end{array}
\right]
$
}
\vspace{0.15cm} \\ 
\begin{minipage}{5cm}
\hspace*{1cm}
$
\begin{array}{r}
(a-b)T_{7}T_{8}+(b-1)T_{9}T_{10}+(-a+1)T_{11}T_{12},\\
(a-1)T_{1}T_{5}-T_{7}T_{8}+T_{9}T_{10},\\
(a-1)T_{2}T_{4}-aT_{7}T_{8}+T_{9}T_{10}
\end{array}
$
\end{minipage}
\end{minipage}
\\ \\ \hline \\ 
\begin{minipage}{1.6cm}
\footnotesize
\begin{center}
\tiny (iii) \\ 
\begin{tikzpicture}[scale=.2]
\coordinate (100) at (-2,0);
\coordinate (010) at (2,0);
\coordinate (001) at ($(100)!1.0!60:(010)$);
\coordinate (111) at ($(100)!.5!(010)!1/3!(001)$);
\coordinate (110) at ($(100)!0.5!(010)$);
\coordinate (1a0) at ($(100)!0.75!(010)$);
\filldraw[fill=black!25, draw=black] (100)--(010)--(001)--cycle;
\draw[draw=black, densely dotted] (001)--(110);
\fill (100) circle (.20cm); 
\fill (010) circle (.20cm);  
\fill (001) circle (.20cm);  
\fill (111) circle (.20cm); 
\fill (110) circle (.20cm); 
\fill (1a0) circle (.20cm); 
\end{tikzpicture}
\\[1ex]
\begingroup
\tiny 
$
\begin{array}{rcl}
p_1 &=& [1,0,0],\\
p_2 &=& [0,1,0],\\
p_3 &=& [0,0,1], \\
p_4 &=& [1,1,1],\\
p_5 &=& [1,1,0],\\
p_6 &=& [1,a,0].\\
\end{array}
$
\endgroup
\end{center}
\end{minipage}
&
\begin{minipage}{7cm}
\mbox{\tiny
$
\left[
\begin{array}{rrrrrrrrrrrrrr}
0 & 1 & 1 & -1 & 0 & 0 & 0 & 0 & 0 & 0 & 0 & 0 & 0 & 0\\
0 & 0 & -1 & 0 & 0 & 0 & -1 & 1 & -1 & 0 & 0 & 0 & 0 & 0\\
1 & 0 & 1 & 0 & -1 & 0 & 0 & 0 & 0 & 1 & 0 & 0 & 0 & 0\\
0 & -1 & 1 & 0 & 0 & 1 & 0 & 0 & 0 & 0 & 1 & 0 & 0 & 0\\
-1 & 0 & -1 & 0 & 1 & 1 & -1 & 1 & 0 & 0 & 0 & 1 & 0 & 0\\
1 & 1 & 0 & 0 & 0 & -1 & 1 & 0 & 0 & 0 & 0 & 0 & 1 & 0\\
0 & 1 & -2 & 0 & 1 & 0 & 0 & 1 & 0 & 0 & 0 & 0 & 0 & 1
\end{array}
\right]
$
}
\vspace{0.15cm} \\ 
\begin{minipage}{5cm}
$
\begin{array}{rl}
(a-1)T_{3}T_{4}T_{5}T_{8}+T_{11}T_{12}+T_{6}T_{13},&
aT_{3}T_{5}T_{7}T_{8}^2+T_{2}T_{12}-T_{9}T_{13},\\
T_{3}T_{4}T_{7}T_{8}^2+T_{1}T_{12}-T_{10}T_{13},&
T_{6}T_{7}T_{8}+(a-1)T_{5}T_{10}+T_{12}T_{14},\\
T_{4}T_{9}-aT_{5}T_{10}-T_{12}T_{14},&
T_{3}T_{7}T_{8}^2T_{14}+T_{1}T_{9}-T_{2}T_{10},\\
-T_{3}T_{5}T_{8}T_{14}+T_{2}T_{6}+T_{9}T_{11},&
-T_{3}T_{4}T_{8}T_{14}+T_{1}T_{6}+T_{10}T_{11},\\
-T_{7}T_{8}T_{11}+(a-1)T_{1}T_{5}+T_{13}T_{14},&
T_{2}T_{4}-aT_{1}T_{5}-T_{13}T_{14}
\end{array}
$
\end{minipage}
% }
\end{minipage}
\\ \\ \hline \\ 
\begin{minipage}{1.6cm}
\footnotesize
\begin{center}
\tiny (iv)\\
\begin{tikzpicture}[scale=.2]
\coordinate (100) at (-2,0);
\coordinate (010) at (2,0);
\coordinate (001) at ($(100)!1.0!60:(010)$);
\coordinate (111) at ($(100)!.5!(010)!1/3!(001)$);
\coordinate (110) at ($(100)!1/3!(010)$);
\coordinate (1a0) at ($(100)!2/3!(010)$);
\filldraw[fill=black!25, draw=black] (100)--(010)--(001)--cycle;
\fill (100) circle (.20cm);
\fill (010) circle (.20cm);
\fill (001) circle (.20cm);
\fill (111) circle (.20cm);
\fill (110) circle (.20cm);
\fill (1a0) circle (.20cm);
\end{tikzpicture}
\\[1ex]
\begingroup
\tiny 
$
\begin{array}{rcl}
p_1 &=& [1,0,0],\\
p_2 &=& [0,1,0],\\
p_3 &=& [0,0,1], \\
p_4 &=& [1,1,1],\\
p_5 &=& [1,a,0],\\
p_6 &=& [1,b,0].\\
\end{array}
$
\endgroup
\end{center}
\end{minipage}
&
\begin{minipage}{7cm}
\mbox{\tiny
$
\left[
\begin{array}{rrrrrrrrrrrrrrrr}
0 & -1 & 0 & 1 & 0 & 1 & -1 & 0 & -1 & 1 & 0 & 0 & 0 & 0 & 0 & 0\\
1 & -1 & 0 & 1 & -1 & 0 & -1 & 1 & 0 & 0 & 0 & 0 & 0 & 0 & 0 & 0\\
0 & 1 & 1 & -1 & 0 & 0 & 1 & 0 & 0 & 0 & 0 & 0 & 0 & 0 & 0 & 0\\
0 & -1 & 0 & 1 & 0 & 0 & -1 & 0 & 0 & 0 & 1 & 1 & -1 & 0 & 0 & 0\\
-1 & 0 & 0 & 0 & 1 & 1 & 1 & 0 & 0 & 0 & 1 & 0 & 0 & 1 & 0 & 0\\
1 & 1 & 0 & 0 & 0 & -1 & 0 & 0 & 1 & 0 & 0 & 1 & 0 & 0 & 1 & 0\\
0 & 2 & 0 & -1 & 1 & 0 & 2 & 0 & 1 & 0 & 1 & 1 & 0 & 0 & 0 & 1
\end{array}
\right]
$
}
\vspace{0.15cm} \\ 
\begin{minipage}{5cm}
$
\begin{array}{rl}
(-a^2+ab)T_{3}T_{4}T_{5}T_{9}+T_{12}T_{14}+aT_{11}T_{15},&
(a-1)T_{5}T_{8}+T_{6}T_{9}-T_{11}T_{13},\\
(ab-a)T_{3}T_{4}T_{5}T_{13}+T_{10}T_{14}+aT_{6}T_{15},&
(a-1)T_{4}T_{7}+aT_{6}T_{9}-T_{11}T_{13},\\
(ab)T_{3}T_{5}T_{9}T_{13}+T_{2}T_{14}-aT_{7}T_{15},&
T_{3}T_{9}T_{13}T_{16}+T_{1}T_{7}-T_{2}T_{8},\\
aT_{3}T_{4}T_{9}T_{13}+T_{1}T_{14}-aT_{8}T_{15},&
-T_{3}T_{5}T_{13}T_{16}+T_{2}T_{6}+T_{7}T_{10},\\
(-a+1)T_{3}T_{4}T_{5}T_{16}+T_{10}T_{11}-T_{6}T_{12},&
-T_{3}T_{4}T_{13}T_{16}+T_{1}T_{6}+T_{8}T_{10},\\
(-a)T_{3}T_{5}T_{9}T_{16}+T_{2}T_{11}+T_{7}T_{12},&
(a-1)T_{1}T_{5}-T_{9}T_{10}+T_{12}T_{13},\\
-T_{3}T_{4}T_{9}T_{16}+T_{1}T_{11}+T_{8}T_{12},&
T_{2}T_{4}-T_{1}T_{5}-T_{9}T_{10}
\end{array}
$
\\
\hspace*{2cm}
$
\begin{array}{r}
(a-b)T_{9}T_{10}+(b-1)T_{12}T_{13}+(-a+1)T_{15}T_{16},\\
(a^2-ab)T_{6}T_{9}+(ab-a)T_{11}T_{13}+(a-1)T_{14}T_{16},
\end{array}
$
\end{minipage}
\end{minipage}
\\ \\ \hline \\ 
\multicolumn{2}{l}{
\begin{minipage}{1.6cm}
\footnotesize
\begin{center}
\begin{tikzpicture}[scale=.2]
\coordinate (100) at (-2,0);
\coordinate (010) at (2,0);
\coordinate (001) at ($(100)!1.0!60:(010)$);
\coordinate (111) at ($(100)!.5!(010)!1/3!(001)$);
\coordinate (110) at ($(100)!.5!(010)$);
\coordinate (101) at ($(100)!.5!(001)$);
\filldraw[fill=black!25, draw=black] (100)--(010)--(001)--cycle;
\draw[draw=black, densely dotted] (001) -- (110);
\draw[draw=black, densely dotted] (101) -- (010);
\fill (100) circle (.20cm);
\fill (010) circle (.20cm);
\fill (001) circle (.20cm);
\fill (111) circle (.20cm);
\fill (110) circle (.20cm);
\fill (101) circle (.20cm);
\end{tikzpicture}\\
\tiny (v) 
\end{center}
\end{minipage}
\qquad\qquad
\begin{minipage}{1.6cm}
\footnotesize
\begin{center}
\begin{tikzpicture}[scale=.2]
\coordinate (100) at (-2,0);
\coordinate (010) at (2,0);
\coordinate (001) at ($(100)!1.0!60:(010)$);
\coordinate (110) at ($(100)!.5!(010)$);
\coordinate (101) at ($(100)!.5!(001)$);
\coordinate (011) at ($(010)!.5!(001)$);
\filldraw[fill=black!25, draw=black] (100)--(010)--(001)--cycle;
\fill (100) circle (.20cm);
\fill (010) circle (.20cm);
\fill (001) circle (.20cm);
\fill (110) circle (.20cm);
\fill (101) circle (.20cm);
\fill (011) circle (.20cm);
\end{tikzpicture}\\
\tiny (vi) 
\end{center}
\end{minipage}
\qquad\qquad 
\begin{minipage}{1.6cm}
\footnotesize
\begin{center}
\begin{tikzpicture}[scale=.2]
\coordinate (100) at (-2,0);
\coordinate (010) at (2,0);
\coordinate (001) at ($(100)!1.0!60:(010)$);
\coordinate (111) at ($(100)!.5!(010)!1/3!(001)$);
\coordinate (1a0) at ($(100)!2/3!(010)$);
\coordinate (10b) at ($(100)!1/3!(001)$);
\filldraw[fill=black!25, draw=black] (100)--(010)--(001)--cycle;
\fill (100) circle (.20cm);
\fill (010) circle (.20cm);
\fill (001) circle (.20cm);
\fill (111) circle (.20cm);
\fill (1a0) circle (.20cm);
\fill (10b) circle (.20cm);
\end{tikzpicture}\\
\tiny (vii) 
\end{center}
\end{minipage}
\qquad\qquad
\begin{minipage}{1.6cm}
\footnotesize
\begin{center}
\begin{tikzpicture}[scale=.2]
\coordinate (100) at (-2,0);
\coordinate (010) at (2,0);
\coordinate (001) at ($(100)!1.0!60:(010)$);
\coordinate (110) at ($(100)!.5!(010)$);
\coordinate (111) at ($(100)!.5!(010)!1/3!(001)$);
\coordinate (1a0) at ($(100)!2/3!(010)$);
\coordinate (11b) at ($(100)!.5!(010)!2/3!(001)$);
\filldraw[fill=black!25, draw=black] (100)--(010)--(001)--cycle;
\draw[draw=black, densely dotted] (001) -- (110);
\fill (100) circle (.20cm);
\fill (010) circle (.20cm);
\fill (001) circle (.20cm);
\fill (111) circle (.20cm);
\fill (1a0) circle (.20cm);
\fill (11b) circle (.20cm);
\end{tikzpicture}\\
\tiny (viii) 
\end{center}
\end{minipage}
\qquad\qquad
\begin{minipage}{1.6cm}
\footnotesize
\begin{center}
\begin{tikzpicture}[scale=.2]
\coordinate (100) at (-2,0);
\coordinate (010) at (2,0);
\coordinate (001) at ($(100)!1.0!60:(010)$);
\coordinate (111) at ($(100)!.5!(010)!1/3!(001)$);
\coordinate (1a0) at ($(100)!2/3!(010)$);
\coordinate (1bc) at ($(100)!2/3!(010)!2/3!(001)$);
\filldraw[fill=black!25, draw=black] (100)--(010)--(001)--cycle;
\fill (100) circle (.20cm);
\fill (010) circle (.20cm);
\fill (001) circle (.20cm);
\fill (111) circle (.20cm);
\fill (1a0) circle (.20cm);
\fill (1bc) circle (.20cm);
\end{tikzpicture}\\
\tiny (ix)
\end{center}
\end{minipage}
}
\\ \\
\multicolumn{2}{l}{
These are the cubic surfaces with singularity types
$4A_1$, $3A_1$, $2A_1$, $A_2$ and~$A_1$
listed in Theorem~\ref{thm:cubic}.}
\\ \\ \hline
\end{longtable}
\endgroup
\end{theorem}

\begin{proof}
  This is an application of the procedure \tt{blowupCEMDSpoints}
  of our package \tt{compcox.lib}.
  The source code used to compute the listed Cox rings can 
  be obtained from the homepage~\cite{DeHaHeKeLa} of our package.
\end{proof}

\section{Smooth Fano threefolds}
\label{sec:3fanos}

Smooth Fano threefolds come in a finite number of families. For Picard
number one they have been classified by Iskovskikh~\cite{Is,Is2} and
for Picard number at least two by Mori and Mukai~\cite{MoMu}.  We use
the descriptions of these varieties given in~\cite[Section~12]{AG5} to
compute Cox rings 
in the cases of Picard number at most two.
Recall that the {\em index\/} of a smooth Fano threefold
$X$ is the maximal $r\in \ZZZ$ such that 
$-\frac{1}{r}K_X$ is an integral ample divisor.

\begin{theorem}
\label{thm:fano1}
Let $X$ be a non-toric smooth Fano threefold with Picard number $1$.
Let $r\in \ZZZ$ be the index of $X$.
\begin{enumerate}
\item 
If $-\frac{1}{r}K_X$ is not very ample,
then the $\ZZ$-graded Cox ring $\Cox(X)$ is listed in the following table.\\

\begingroup
\tiny
\begin{longtable}{lll}
\hline
$(r,-K_X^3)$
&
Cox ring $\Cox(X)$
&
degree matrix $Q$
\\ \hline \\
$(1,2)$
&
$
\begin{array}{l}
\KK[T_1,\ldots,T_{5}] / 
\langle
T_5^2-f_1
\rangle    
\end{array}
$
&
$
\left[
\mbox{\tiny $
\begin{array}{rrrrr}
1 & 1 & 1 & 1 & 3
\end{array}
$}
\right]
$
\\ \\ 
\multicolumn{3}{l}{
where $f_1\in \KT{4}$ is homogeneous with $\deg(f_1)=6$
and $V(f_1)\subseteq \PP_3$ is smooth.
}
\\
\\
\hline
\\
$(1,4)$
&
$
\begin{array}{l}
\KK[T_1,\ldots,T_{6}] / 
\langle
f_2,T_6^2-f_3
\rangle    
\end{array}
$
&
$
\left[
\mbox{\tiny $
\begin{array}{rrrrrr}
1 & 1 & 1 & 1 & 1 & 2
\end{array}
$}
\right]
$
\\ \\ 
\multicolumn{3}{l}{
where $f_2,f_3\in \KT{5}$ are homogeneous 
with $\deg(f_2)=2$, $\deg(f_3)=4$}\\[2pt]
\multicolumn{3}{l}{
and both $V(f_2)\subseteq \PP_4$ and $V(f_2,f_3)\subseteq \PP_4$ are smooth.
}
\\
\\
\hline
\\
$(2,8)$
&
$
\begin{array}{l}
\KK[T_1,\ldots,T_{5}] / 
\langle
f_4
\rangle    
\end{array}
$
&
$
\left[
\mbox{\tiny $
\begin{array}{rrrrr}
1 & 1 & 1 & 2 & 3
\end{array}
$}
\right]
$
\\ \\ 
\multicolumn{3}{l}{
where $f_4\in \KT{5}$ is homogeneous with $\deg(f_4)=6$
and $V(f_4)\subseteq \PP(1,1,1,2,3)$}\\[2pt]
\multicolumn{3}{l}{ is smooth.
}
\\
\\
\hline
\\
$(2,16)$
&
$
\begin{array}{l}
\KK[T_1,\ldots,T_{5}] / 
\langle
T_5^2-f_5
\rangle
\end{array}
$
&
$
\left[
\mbox{\tiny $
\begin{array}{rrrrr}
1 & 1 & 1 & 1 &  2
\end{array}
$}
\right]
$
\\ \\
\multicolumn{3}{l}{
where $f_5\in \KT{4}$ is homogeneous 
with $\deg(f_5)=4$ and $V(f_5)\subseteq \PP_3$ is smooth.
}
\\
\\
\hline
\end{longtable}
\endgroup
\item 
If the divisor $-\frac{1}{r}K_X$ is very ample,
  it gives rise to an embedding $\nu\colon X\to\mathbb{P}_n$ and the
  Cox ring of $X$ is isomorphic to the homogeneous coordinate ring
  of~$\nu(X)$. 
 Moreover, for $(r,-K_X^3)=(4,64)$ we have $X=\PP_3$.
 We list the 
 Cox rings for the cases $r>1$ explicitly.
 \begingroup
 \tiny
\begin{longtable}{lll}
\hline
$(r,-K_X^3)$
&
Cox ring $\Cox(X)$
&
degree matrix $Q$
\\ \hline \\
$(2,24)$
&
$
\begin{array}{l}
\KK[T_1,\ldots,T_{5}] / 
\langle
f
\rangle    
\end{array}
$
&
$
\left[
\mbox{\tiny $
\begin{array}{rrrrr}
1 & 1 & 1 & 1 & 1
\end{array}
$}
\right]
$
\\ \\ 
\multicolumn{3}{l}{
where $f$ is homogeneous of degree $3$
and $V(f)=\nu(X)$.
}
\\
\\
\hline
\\
$(2,32)$
&
$
\begin{array}{l}
\KK[T_1,\ldots,T_{6}] / 
\langle
f_1,\,f_2
\rangle    
\end{array}
$
&
$
\left[
\mbox{\tiny $
\begin{array}{rrrrrr}
1 & 1 & 1 & 1 & 1 & 1
\end{array}
$}
\right]
$
\\ \\ 
\multicolumn{3}{l}{
where $f_1,f_2$ are homogeneous of degree $2$ 
and $V(f_1,f_2)=\nu(X)$.
}
\\
\\
\hline
\\
$(2,40)$
&
$
\begin{array}{l}
\KK[T_1,\ldots,T_{10}] / I
\\[1ex]
\text{with $I$ generated by}
\\[1ex]
T_7T_8 - T_6T_9 + T_5T_{10},\\
T_4T_6 - T_3T_7 - T_1T_{10},\\
T_4T_8 - T_3T_9 + T_2T_{10},\\
T_4T_5 - T_2T_7 - T_1T_9,\\
T_3T_5 - T_2T_6 - T_1T_8,\\
f_1,\ f_2,\ f_3
\end{array}
$
&
$
\left[
\mbox{\tiny $
\begin{array}{rrrrrrrrrr}
1 & 1 & 1 & 1 & 1 & 1 & 1 & 1 & 1 & 1
\end{array}
$}
\right]
$
\\ \\ 
\multicolumn{3}{l}{
where $f_1,f_2$ and $f_3$ are homogeneous of degree one.
}
\\
\\
\hline
\\
$(4,54)$
&
$
\begin{array}{l}
\KK[T_1,\ldots,T_{5}] / 
\langle
f
\rangle    
\end{array}
$
&
$
\left[
\mbox{\tiny $
\begin{array}{rrrrr}
1 & 1 & 1 & 1 & 1
\end{array}
$}
\right]
$
\\ \\ 
\multicolumn{3}{l}{
where $f$ is homogeneous of degree $2$
and $V(f)=\nu(X)$.
}
\\
\\
\hline
\\
 \end{longtable}
\endgroup
\end{enumerate}
\end{theorem}

\begin{remark}
 In Theorem~\ref{thm:fano1}, the 
 Cox rings with $r=2$
  can easily be adjusted to
 the Cox rings of higher dimensional 
 del Pezzo varieties as classified in~\cite[Table~12.1]{AG5}.
\end{remark}

For the proof of Theorem~\ref{thm:fano1}, we first collect some observations on
$n$-cyclic covers of smooth varieties,
see~\cite[Sec.~I.17]{BHPV} for the definition and 
basic background.

\begin{lemma}
\label{pic}
Let $\pi \colon X \to Y$ be a $n$-cyclic cover of smooth 
varieties with a smooth, prime branch divisor 
$B \in \WDiv(Y)$.
Assume that $\Cl(X)$ is free and finitely generated
and that $\pi^*\Cl(Y)$ is of finite index in 
$\Cl(X)$.
Then we have $\pi^*\Cl(Y) = \Cl(X)$.
\end{lemma}

\begin{proof}
Consider the covering automorphism $\sigma \colon X \to X$ of order $n$.
We show that each $D\in \WDiv(X)$
is linearly equivalent to a pullback divisor.

Assume $\sigma(D)=D$ holds. 
Observe that any such divisor is an integral
sum $a_1 E_1+\ldots +a_s E_s$
of reduced invariant divisors $E_i$
where each $E_i$ is the sum of all the prime divisors of a $\<\sigma\>$-orbit.
Thus, it suffices to show the claim
for a divisor $D:=E_i$.
If $D$ is the ramification divisor, i.e., if 
$\pi(D) = B$ holds, then this is guaranteed 
by~\cite[Lemma~I.17.1(i)]{BHPV}. 
So, we are left with the case $\pi(D) \ne B$. 
Then we have
$D = \pi^{-1}(\pi(D))$.
Since $\pi$ is unramified outside $\pi^{-1}(B)$,
we obtain $D = \pi^* \pi(D)$.

Assume now $\sigma(D) \ne D$ holds. 
Since 
$\pi^*\Cl(Y)$ is of finite index in the 
free finitely generated $\Cl(X)$ 
and $\sigma^{-1}$  
acts trivially on $\pi^*\Cl(Y)$,
we see that $\sigma^{-1}$ acts trivially on the
whole $\Cl(X)$.
We conclude $D - \sigma^{-1}(D) = \div(g)$
with some $g \in \KK^*(X)$.
We claim that
$$ 
\tau \colon \Gamma(X,\mathcal{O}(D)) \ \to \ \Gamma(X,\mathcal{O}(D)),
\qquad
f \ \mapsto \ \frac{\sigma^*f}{g}
$$
is a linear isomorphism
and, after suitably rescaling $g$,
the order of $\tau$ divides $n$.
Indeed,
$\sigma^*f/g$ is an element of $\Gamma(X,\mathcal{O}(D))$
and thus $\tau$ is a well defined linear map.
Iteratively applying $\sigma^{-i}$ to 
$D - \sigma^{-1}(D) = \div(g)$
leads to $n$ equations; 
adding them up and using 
$\sigma^{-1}(\div(g)) = \div(\sigma^*g)$,
we achieve $g \sigma^*g\cdots (\sigma^*)^{n-1}g = 1$
after possibly rescaling $g$.
Therefore, the order of the linear map $\tau$ divides $n$. 
Now, take an eigenvector $h \in \Gamma(X,\mathcal{O}(D))$
of $\tau$. 
Then $D' := \div(h) + D$ is invariant under $\sigma^{-1}$ and
thus, as seen before, a pullback divisor.
Consequently, $D$ is linearly equivalent to a pullback 
divisor.
\end{proof}

\begin{lemma}
\label{lem:doublecover}
Let $\pi\colon X\to Y$ be an $n$-cyclic cover of smooth Fano 
varieties with ample branch divisor $B=\div_{[B]}(f)$. 
Assume that the Picard
rank of $X$ equals that of $Y$.
Then there is an isomorphism
\[
 \mathcal R(X)\ \ \cong\ \ \mathcal R(Y)[S]/\langle S^n-f\rangle.
\]
of $\Cl(X)$-graded algebras.
The grading of $\mathcal R(X)$
is the same as for $\mathcal R(Y)$
with the additional degree $\deg(S)=\deg(f)/n$.
\end{lemma}

\begin{proof}
  Since $B$ is ample, it is connected and thus it must be 
  irreducible,  as $X$ is smooth.
  Hence the ramification divisor $R$ is irreducible
  and the hypothesis of Lemma~\ref{pic} is satisfied.
  Then the result
  follows from \cite[Proposition 5.3.1.3]{ArDeHaLa}.
\end{proof}

\begin{proof}[Proof of Theorem~\ref{thm:fano1}]
First, observe that according to~\cite[Table 12.2]{AG5}
we have $r\in\{1,2,3,4\}$. Moreover, $X=\mathbb{P}_3$
when $r=4$ and $X$ is a smooth quadric of $\mathbb{P}_4$
for the case of $r=3$. 

We now show statement (ii) for $r\in \{1,2\}$.
In what follows $S$ will be a smooth surface in $X$
linearly equivalent to $-\frac{1}{r}K_X$ and $C$ will 
be a smooth curve of $S$ linearly equivalent to 
$\OO_X(S)|_S$.
Since $S$ is ample and $rS$ is linearly equivalent
to $-K_X$, we have $(n-1)S = K_X +A$ with $A$ ample
for any $n\in \ZZ_{\geq 1}$.
By the Kawamata--Viehweg vanishing theorem,
for any such $n$,
 we have an exact sequence
\[
 \xymatrix@R=5pt@C=15pt{
  0\ar[r] & H^0(X,(n-1)S)\ar[r] & H^0(X,nS)\ar[r] & H^0(S,nC)\ar[r] & 0.
 }
\]

Case $r=2$. By the adjunction formula $S$ is a del Pezzo surface and
$C\sim -K_S\sim -\tfrac{1}{2}K_X|_S$.  
According to~\cite[Table 12.2]{AG5}, the surface $S$
has degree $\geq 5$; if the degree is $\leq 2$ then
$-\tfrac{1}{2}K_X$ is not very ample
since its restriction to $S$ is not very ample.
For the case of degree $\geq 3$, by~\cite[Theorem~8.3.4]{Do},
making use of the above exact sequence
one directly shows that the anticanonical ring
$R(S,C)=R(S,-K_S)$ is generated in degree one. Thus, by the above
exact sequence, the same holds for $R(X,S)\cong \Cox(X)$.

Case $r=1$. By the adjunction formula $S$ is a K3 surface 
and $C$ is a canonically embedded curve and thus $C$ 
is non-hyperelliptic.
By the Kawamata--Viehweg vanishing theorem and the
ampleness of $C$, we have 
for any $n\in \ZZ_{\geq 1}$
an exact sequence
\[
 \xymatrix@R=5pt@C=15pt{
  0\ar[r] & H^0(S,(n-1)C)\ar[r] & H^0(S,nC)\ar[r] & H^0(C,nK_C)\ar[r] & 0.
 }
\]
Since $C$ is non-hyperelliptic, its canonical ring
$R(C,K_C)$
is generated in degree one by Max Noether's Theorem
~\cite[page~117]{ACGH}. Thus, due to the above
exact sequences, the same holds for the algebras
$R(S,C)$ and $R(X,-K_X)$.

For both $r=1$ and $r=2$ this shows in particular
that the Cox ring $\Cox(X)$ equals the homogeneous 
coordinate ring of $\nu(X)$; they are directly obtained
from~\cite[Table~12.2]{AG5}.

We now compute the Cox rings listed in the table of~(i).
If $(r, -K_X^3)=(1,2)$, by~\cite[Table~12.2]{AG5}, $X\to \PP_3$
is a double cover with the branch divisor of degree six.
If $(r, -K_X^3)=(2,16)$, by~\cite[Table~12.2]{AG5}, 
$X$ is a finite cover of $\PP_3$ of degree two 
branched along a smooth surface in $\PP_3$ 
of degree four.
In both cases, Lemma~\ref{lem:doublecover} gives the Cox ring.
For $(r, -K_X^3)=(2,8)$ the variety $X$ is del Pezzo
and therefore appears in~\cite[Table~12.1]{AG5}, i.e.,
we have $X=V(f_5)$ and~\cite[Corollary~4.1.1.3]{ArDeHaLa}
yields the listed Cox ring.
Similarly, for $(r, -K_X^3)=(1,4)$
we use again~\cite[Corollary~4.1.1.3]{ArDeHaLa}
and obtain
$$
\Cox(Q)\ =\ \KT{5}/\<f_2\>\qquad
\text{where}\qquad
Q\ :=\ V(f_2)\,\subseteq\, \PP_4.
$$
By~\cite[Table~12.2]{AG5}, $X$ is a double cover of $Q$
with the branch divisor of degree eight. 
Lemma~\ref{lem:doublecover} applies.
\end{proof}

The classification of Mori and Mukai leads to $36$ deformation
families of smooth Fano threefolds with Picard number $2$, as listed
in \cite[Table~12.3]{AG5}. Families 33--36 are toric. In the following
result, we compute Cox rings for the families 1--32: in some
cases, we can treat the whole family using computations by hand; in
other cases, we treat only a typical representative
using our software \texttt{compcox.lib} from Section~\ref{sec:example},
 see Example~\ref{ex:X19}. 
The code used to obtain these Cox rings is available at~\cite{DeHaHeKeLa}.

\begin{theorem}
\label{thm:fano2}
Let $X_i$ be as in the classification of non-toric smooth Fano
threefolds in~\cite[Table~12.3]{AG5}. Then the Cox ring is of the
following form (for the cases marked with $\dagger$ we only computed
a typical representative).
\\[1ex]
(1) The smooth Fano threefold $X_{1}$ has the $\ZZ^2$-graded Cox ring
$\KT{6}/I$ with a generator for $I$ and the degree matrix given by
\begin{center}
\tiny
$
\begin{array}{l}
T_1^2
+ 
T_2^3 
\\
+
T_2 f_4(T_4 -T_7T_8 + T_6T_8, T_4, T_4 - T_7T_8)\\
+
f_6(T_4 -T_7T_8 + T_6T_8, T_4, T_4 - T_7T_8)
\end{array}
\quad
\mbox{\tiny $
\left[
    \begin{array}{rrrrrr}
3 & 2 & 1 & 1 & 1 & 0 \\
0 & 0 & 0 & -1 & -1 & 1
\end{array}
\right]
$}
$
\end{center}
where $f_d\in \KK[T_4,T_5,T_6]$ are homogeneous polynomials
of degree $\deg(f_d) = d$.\\[1ex]
(2) The smooth Fano threefold $X_{2}$ 
has the $\ZZ^2$-graded Cox ring
$\KT{6}/I$ 
with a generator for $I$ and the degree matrix given by 
\begin{center}
$T_6^2 - f$
\qquad
\mbox{\tiny $
\left[
\begin{array}{rrrrrr}
    1 & 1 & 0 & 0 & 0 & 1 \\
    0 & 0 & 1 & 1 & 1 & 2
    \end{array}
\right]
$}
\end{center}
where  $f\in \KT{5}$ is homogeneous
of degree  $(2,4)\in \ZZ^2$
and $V(f)\subseteq \PP_1\times\PP_2$
is smooth.
 \\[1ex]
(3) The smooth Fano threefold $X_{3}$ 
has the $\ZZ^2$-graded Cox ring
$\KT{6}/I$ with a generator for $I$
and the degree matrix given by
\begin{center}
\begin{minipage}{4cm}
$
T_3^2 - f(T_4T_6, T_5T_6, T_1,T_2)
$
\end{minipage}
\quad 
\begin{minipage}{5cm}
\mbox{\tiny $
\left[
\begin{array}{rrrrrr}
    1 & 1 & 2 & 0 & 0 & 1 \\
    0 & 0 & 0 & -1 & -1 & 1
    \end{array}
\right]    
$}
\end{minipage}
\end{center}
where $f\in \KK[S_1,\ldots,S_4]$ 
is classically homogeneous of degree four and 
$V(f)\subseteq \PP_3$ is smooth.
\\[1ex]
(4) The smooth Fano threefold $X_{4}$ 
has the $\ZZ^2$-graded Cox ring
$\KT{7}/I$ with generators for $I$
and the degree matrix given by
\begin{center}
\begin{minipage}{4cm}
$
g_1 - T_5T_7,\ 
g_2 - T_6T_7
$
\end{minipage}
\ 
\begin{minipage}{5cm}
    \mbox{\tiny $
    \left[
    \begin{array}{rrrrrrr}
        1 & 1 & 1 & 1 & 3 & 3 & 0 \\
	0 & 0 & 0 & 0 & -1 & -1 & 1 
    \end{array}
    \right]
    $}
\end{minipage}
\end{center}
where 
$g_i\in \KT{4}$ are prime, homogeneous of degree $\deg(g_i) = 3$
and $V(g_1,g_2)\subseteq \PP_3$ is smooth.
\\[1ex]
(5) The smooth Fano threefold $X_{5}$ 
has the $\ZZ^2$-graded Cox ring
$\KT{6}/I$ with a generator for $I$
and the degree matrix given by
\begin{center}
\begin{minipage}{4cm}
$
f(T_1,T_2,T_3,T_4T_6,T_5T_6)
$
\end{minipage}
\ 
\begin{minipage}{5cm}
    \mbox{\tiny $
    \left[
        \begin{array}{rrrrrr}
            1 & 1 & 1 & 1 & 1 & 0 \\
            0 & 0 & 0 & -1 & -1 & 1 
    \end{array}
    \right]
    $}
\end{minipage}
\end{center}
where $f\in \KK[S_1,\ldots,S_5]$ is homogeneous
of degree three such that both $V(f)\subseteq \PP_4$
 and $V(f,T_4,T_5)\subseteq \PP_4$ are smooth.
\\[1ex]
(6a) The smooth Fano threefold $X_{6,a}$ 
has the $\ZZ^2$-graded Cox ring
$\KT{6}/I$ with a generator for $I$
and the degree matrix given by
\begin{center}
$f$\qquad
\mbox{\tiny $
\left[
    \begin{array}{rrrrrr}
    1 & 1 & 1 & 0 & 0 & 0\\
    0 & 0 & 0 & 1 & 1 & 1
    \end{array}
\right]
$}
\end{center}
where $f\in \KT{6}$ is homogeneous
of degree  $(2,2)\in \ZZ^2$
such that $V(f)\subseteq \PP_2\times \PP_2$ is smooth.
\\[1ex]
(6b) The smooth Fano threefold $X_{6b}$ 
has the $\ZZ^2$-graded Cox ring
$\KT{7}/I$ with generators for $I$
and the degree matrix given by
\begin{center}
$
f,\ T_7^2 -g
$
\qquad
$
\left[
\mbox{\tiny $
\begin{array}{rrrrrrr}
1 & 1 & 1 & 0 & 0 & 0 & 1\\
0 & 0 & 0 & 1 & 1 & 1 & 1
\end{array}
$}
\right]
$
\end{center}
with a homogeneous $g\in \KT{6}$ 
of degree $(2,2)\in \ZZ^2$
and a homogeneous polynomial $f\in \KT{6}$
of degree $(1,1)\in \ZZ^2$
such that both
$V(f)\subseteq \PP_2\times \PP_2$ and
$V(f,g)\subseteq \PP_2\times \PP_2$ are smooth.
\\[1ex]
(7) The smooth Fano threefold $X_{7}$ 
has the $\ZZ^2$-graded Cox ring
$\KT{8}/I$ with generators for $I$
and the degree matrix given by
\begin{center}
\begin{minipage}{4cm}
$
\begin{array}{l}
f,\ 
g_1 - T_6T_8,\ 
g_2 - T_7T_8
\end{array}
$
\end{minipage}
\quad
\begin{minipage}{5cm}
\mbox{\tiny $
\left[
\begin{array}{rrrrrrrr}
    1 & 1 & 1 & 1 & 1 & 2 & 2 & 0\\
    0 & 0 & 0 & 0 & 0 & -1 & -1 & 1     
\end{array}
\right]
$}
\end{minipage}
\end{center}
where $f,g_1,g_2\in \KT{5}$ are classically homogeneous
with $\deg(f)=\deg(g_i)=2$ such that 
$V(f,g_1,g_2)\subseteq \PP_4$ is smooth.
\\[1ex]
(8) The smooth Fano threefold $X_{8}$
has the $\ZZ^2$-graded Cox ring
$\KT{6}/I$
with a generator for $I$
and the degree matrix given by
\begin{center}
$T_6^2 -f$
\qquad
\mbox{\tiny $
\left[
\begin{array}{rrrrrr}
1 & 1 & 1 & 1 & 0 & 2\\
-1 & -1 & -1 & 0 & 1 & -1 
\end{array}
\right]
$}
\end{center}
with a homogeneous polynomial $f\in \KT{5}$
of degree $(4,-2)\in \ZZ^2$ such that
$V(f)\subseteq \PP_3$ is smooth and
 $V(f,T_5)\subseteq Y$ is reduced
 where $Y$ is the blow-up of $\PP_3$ at a point. 
\\[1ex]
(9) $\dagger$ The smooth Fano threefold $X_{9}$ 
has the $\ZZ^2$-graded Cox ring
$\KT{8}/I$ with generators for $I$
and the degree matrix given by
\begin{center}
\tiny 
\begin{minipage}{5cm}
$
\begin{array}{l}
T_{2}T_{5} - T_{3}T_{6} - T_{3}T_{7} + 3T_{4}T_{5} - T_{4}T_{6} + T_{4}T_{7},
\\
T_{1}^2T_{3} - T_{1}^2T_{4} + T_{2}^2T_{3} + T_{2}^2T_{4} + T_{3}^3 -\\ 2T_{3}^2T_{4} +   T_{3}T_{4}^2 - 3T_{4}^3 - T_{5}T_{8},\\
    T_{2}^3 + 5/2T_{2}^2T_{3} + 1/2T_{2}^2T_{4} - 1/2T_{2}T_{3}^2 -\\ T_{2}T_{4}^2 + T_{3}^3 - 5/2T_{3}^2T_{4}
    + 1/2T_{3}T_{4}^2 -\\
    7/2T_{4}^3 - T_{6}T_{8},
\end{array}
$
\end{minipage}
\qquad
\begin{minipage}{5cm}
$
\begin{array}{l}
    T_{1}^2T_{2} + 3T_{1}^2T_{4} - 5/2T_{2}^2T_{3} - 5/2T_{2}^2T_{4}
    \\
    + 3/2T_{2}T_{3}^2 
    +  2T_{2}T_{4}^2 - T_{3}^3 + 7/2T_{3}^2T_{4} - 1/2T_{3}T_{4}^2 \\ +11/2T_{4}^3 - T_{7}T_{8}, \\
        -T_{1}^2T_{6} + T_{2}^2T_{7} + 5/2T_{2}T_{3}T_{6} \\
        + 5/2T_{2}T_{3}T_{7} + 5/2T_{2}T_{4}T_{6} - 5/2T_{2}T_{4}T_{7}\\
        + T_{3}^2T_{5} - 3/2T_{3}^2T_{6} - 1/2T_{3}^2T_{7} -  15/2T_{3}T_{4}T_{6} -\\
         15/2T_{3}T_{4}T_{7} + 23T_{4}^2T_{5}
        - 19/2T_{4}^2T_{6} +  13/2T_{4}^2T_{7}
\end{array}
$
\end{minipage}
\\[1ex]
    \mbox{\tiny $
    \left[
    \begin{array}{rrrrrrrr} 
    1 & 1 & 1 & 1 & 3 & 3 & 3 & 0\\
    0 & 0 & 0 & 0 & -1 & -1 & -1 & 1
    \end{array}
    \right]
    $}
\end{center}
(10) The smooth Fano threefold $X_{10}$ 
has the $\ZZ^2$-graded Cox ring
$\KT{7}/I$ with generators for $I$
and the degree matrix given by
\begin{center}
\begin{minipage}{5cm}
$
\begin{array}{l}
f_1(T_1,\ldots,T_4,T_5T_7,T_6T_7),\\ 
f_2(T_1,\ldots,T_4,T_5T_7,T_6T_7)
\end{array}
$
\end{minipage}
\quad 
\begin{minipage}{5cm}
\mbox{\tiny $
\left[
\begin{array}{rrrrrrr}
 1 & 1 & 1 & 1 & 1 & 1 & 0 \\
 0 & 0 & 0 & 0 & -1 & -1 & 1
\end{array}
\right]
$}
\end{minipage}
\end{center}
where $f_{i}\in \KT{6}$ are homogeneous of degree $\deg(f_{i})=2$
and $V(f_{1},f_{2})\subseteq \PP_{5}$ is smooth.
\\[1ex]
(11) The smooth Fano threefold $X_{11}$ 
has the $\ZZ^2$-graded Cox ring
$\KT{6}/I$ with generators for $I$
and the degree matrix given by
\begin{center}
\begin{minipage}{5cm}
\tiny
$
\begin{array}{l}
T_3
g_1(T_3T_6, T_4T_6, T_5T_6, T_1, T_2)\\
+
T_4 
g_2(T_3T_6, T_4T_6, T_5T_6, T_1, T_2)\\
+
T_5
g_3(T_3T_6, T_4T_6, T_5T_6, T_1, T_2)
\end{array}
$
\end{minipage}
\quad 
\begin{minipage}{5cm}
$
\left[
\mbox{\tiny $
 \begin{array}{rrrrrr}
    1 & 1 & 1 & 1 & 1 & 0\\
    0 & 0 & -1 & -1 & -1 & 1 
    \end{array}
$}
\right]
$
\end{minipage}
\end{center}
where $g_1,g_2,g_3\in \KK[T_1,\ldots,T_5]$
are homogeneous of degree 
 $\deg(g_i)=2$ such that $V(T_1g_1+T_2g_2+T_3g_3)\subseteq \PP_4$ is smooth.
\\[1ex]
(12) $\dagger$ The smooth Fano threefold $X_{12}$ 
has the $\ZZ^2$-graded Cox ring
$\KT{10}/I$ with generators for $I$
and the degree matrix given by
\begin{center}
\begingroup
\allowdisplaybreaks
\tiny 
\begin{minipage}{5cm}
$
\begin{array}{l}
T_{1}T_{7} - T_{2}T_{8} + T_{4}T_{6}, \\
    -T_{1}T_{6} + T_{2}T_{7} + T_{3}T_{5} - T_{3}T_{8},\\
    T_{1}T_{5} - T_{2}T_{6} - T_{3}T_{7} + T_{4}T_{8}, \\
    T_{4}T_{9} - T_{5}^2T_{8} + T_{5}T_{6}T_{7} + T_{5}T_{8}^2 - T_{7}^3,\\
    T_{1}^3 - T_{1}T_{2}T_{3} + T_{1}T_{2}T_{4} + T_{3}^2T_{4} - T_{8}T_{10},    
    \\
    T_{1}T_{9} - T_{5}T_{6}^2 + T_{5}T_{8}^2 + T_{6}^2T_{8} + T_{6}T_{7}^2 - T_{8}^3,\\
    -T_{1}^2T_{4} + T_{2}^3 - T_{2}T_{4}^2 + T_{3}^2T_{4} - T_{5}T_{10},\\
    -T_{1}^2T_{3} + T_{1}T_{2}^2 - T_{1}T_{3}T_{4} + T_{2}T_{3}^2 - T_{6}T_{10}, \\
    T_{1}^2T_{2} + T_{1}T_{3}T_{4} - T_{2}^2T_{3} + T_{3}T_{4}^2 - T_{7}T_{10},
\end{array}
$
\end{minipage}
\qquad
\begin{minipage}{6cm}
$
\begin{array}{l}
    -T_{2}T_{9} + T_{5}^2T_{6} - T_{5}T_{6}T_{8} - T_{5}T_{7}T_{8} -\\ T_{6}^2T_{7} + T_{7}T_{8}^2, \\
  T_{1}^2T_{8}^2 - T_{1}T_{2}T_{6}T_{7} - T_{1}T_{2}T_{6}T_{8} + T_{1}T_{3}T_{6}T_{8} +\\
  T_{1}T_{3}T_{7}^2 - 2T_{1}T_{3}T_{7}T_{8} + T_{1}T_{4}T_{8}^2 + T_{2}^2T_{5}T_{6} \\
  -    T_{2}^2T_{6}T_{8}
        + T_{2}T_{3}T_{6}T_{7} - T_{2}T_{3}T_{7}T_{8} - T_{2}T_{4}T_{5}T_{8}\\
        + T_{2}T_{4}T_{8}^2 - 
        T_{3}^2T_{6}T_{7} + T_{3}^2T_{7}^2 + T_{3}T_{4}T_{7}^2 - 
       \\
       T_{3}T_{4}T_{7}T_{8} - T_{9}T_{10},\\
        -T_{3}T_{9} + T_{4}T_{9} - T_{5}^2T_{8} + T_{5}T_{8}^2 + T_{6}^3 - T_{6}T_{8}^2 \\
        - T_{7}^3 +  T_{7}^2T_{8}
\end{array}
$
\end{minipage}
\endgroup
\\[1ex]
    \mbox{\tiny $
    \left[
        \begin{array}{rrrrrrrrrr}
          1 & 1 & 1 & 1 & 3 & 3 & 3 & 3 & 8 & 0 \\
          0 & 0 & 0 & 0 & -1 & -1 & -1 & -1 & 3 & 1 \\
        \end{array}
    \right]
    $}
\end{center}
(13) $\dagger$ The smooth Fano threefold $X_{13}$ 
has the $\ZZ^2$-graded Cox ring
$\KT{9}/I$ with generators for $I$
and the degree matrix given by
\begin{center}
\tiny 
$
\begin{array}{ll}
T_{1}T_{6}+T_{3}T_{6}-T_{2}T_{7}-T_{4}T_{7}-T_{4}T_{8}, &
T_{2}T_{3}+T_{1}T_{4}+T_{2}T_{4}-T_{6}T_{9}, \\
T_{1}T_{3}+T_{3}^2+T_{1}T_{4}+T_{3}T_{4}-T_{4}T_{5}-T_{7}T_{9}, &
T_{2}^2+T_{3}T_{4}+T_{1}T_{5},\\
T_{1}^2-T_{3}^2-T_{1}T_{4}-T_{3}T_{4}+T_{2}T_{5}+T_{4}T_{5}-T_{8}T_{9}, &
T_{5}T_{6}+T_{1}T_{7}-T_{3}T_{7}-T_{4}T_{7}-T_{3}T_{8}-T_{4}T_{8}
\end{array}
$
\\[1ex]
    \mbox{\tiny $
    \left[
    \begin{array}{rrrrrrrrr}
    1 & 1 & 1 & 1 & 1 & 2 & 2 & 2 & 0 \\
    0 & 0 & 0 & 0 & 0 & -1 & -1 & -1 & 1
    \end{array}
    \right]
    $}
\end{center}
(14) $\dagger$ 
The smooth Fano threefold $X_{14}$ 
has the $\ZZ^2$-graded Cox ring
$\KT{8}/I$ with generators for $I$
and the degree matrix given by
\begin{center}
\tiny 
$
\begin{array}{ll}
T_{1}T_{4} + T_{2}T_{3} - T_{3}T_{6}T_{8} + T_{5}^2, & 
    T_{1}T_{3} + T_{2}T_{5} - T_{3}T_{7}T_{8} - T_{4}^2,\\
    T_{1}T_{5} - T_{2}^2 + T_{2}T_{6}T_{8} + T_{3}T_{4} - T_{5}T_{7}T_{8}, & 
    T_{1}T_{2} + T_{3}^2 + T_{4}T_{5},\\
    T_{1}^2 - T_{1}T_{7}T_{8} + T_{2}T_{4} - T_{3}T_{5} - T_{4}T_{6}T_{8}
\end{array}
$
\\[1ex]
    \mbox{\tiny $
    \left[
    \begin{array}{rrrrrrrr}
    1 & 1 & 1 & 1 & 1 & 1 & 1 & 0\\
    0 & 0 & 0 & 0 & 0 & -1 & -1 & 1
    \end{array}
    \right]
    $}
\end{center}
(15) The smooth Fano threefold $X_{15}$ 
has the $\ZZ^2$-graded Cox ring
$\KT{7}/I$ with generators for $I$
and the degree matrix given by
\begin{center}
\begin{minipage}{4cm}
$
g_1 - T_5T_7,\ 
g_2 - T_6T_7
$
\end{minipage}
\ 
\begin{minipage}{5cm}
\mbox{\tiny $
\left[
\begin{array}{rrrrrrr}
    0 & 0 & 0 & 0 & -1 & -1 & 1 \\
    1 & 1 & 1 & 1 & 0 & 1 & 2
    \end{array}
    \right]
$}
\end{minipage}
\end{center}
where $g_i\in \KT{4}$ are classically homogeneous,
$\deg(g_1)=2$, $\deg(g_2)=3$
and
$V(g_1,g_2)\subseteq \PP_3$
is a smooth curve.
Moreover, $V(g_1)\subseteq \PP_3$ 
is reduced and irreducible.
\\[1ex]
(16) The smooth Fano threefold $X_{16}$ 
has the $\ZZ^2$-graded Cox ring
$\KK[T_1,\ldots,T_{7}]/I$
with generators for $I$
and the degree matrix given by
\begin{center}
\begin{minipage}{5cm}
$
\begin{array}{l}
f_1(T_4T_{7},T_5T_{7},T_6T_{7},T_1,T_2,T_3),\\
f_2(T_4T_{7},T_5T_{7},T_6T_{7},T_1,T_2,T_3),
\end{array}
$
\end{minipage}
\qquad\quad 
\begin{minipage}{5cm}
\mbox{\tiny $
\left[
\begin{array}{rrrrrrr}
1 & 1 & 1 & 1 & 1 & 1 & 0\\
0 & 0 & 0 & -1 & -1 & -1 & 1
\end{array}
\right]
$}
\end{minipage}
\end{center}
where the $f_i\in \KK[S_1,\ldots,S_{6}]$
are homogeneous of degree two
and $V(f_1,f_2)\subseteq \PP_5$ is 
smooth, three-dimensional and
the conic $V(T_1,T_2,T_3,T_5^2-T_4T_6)\subseteq \PP_5$
is contained in $V(f_1,f_2)$.
\\[1ex]
(17) $\dagger$ The smooth Fano threefold $X_{17}$ 
has the $\ZZ^2$-graded Cox ring
$\KK[T_1,\ldots,T_{11}]/I$ with generators for $I$
and the degree matrix given by
\begin{center}
\begin{minipage}{5cm}
\tiny
$
\begin{array}{l}
    T_{2}T_{8} + T_{3}T_{7} - T_{4}T_{8} - T_{4}T_{9} - T_{5}T_{8}, \\
    T_{1}T_{7} + T_{2}T_{9} + T_{3}T_{8},\\
    T_{2}T_{7} + T_{3}T_{6} - T_{5}T_{9}, \\
    T_{2}T_{6} - T_{4}T_{7} + T_{5}T_{8},\\
    T_{1}T_{6} - T_{2}T_{8} + T_{4}T_{9},\\
    -T_{3}T_{5} + T_{4}^2 + T_{4}T_{5} - T_{6}T_{11},\\
    T_{1}T_{5} + T_{2}T_{4} + T_{2}T_{5} - T_{7}T_{11},\\
    T_{1}T_{4} + T_{2}T_{3} - T_{8}T_{11},\\
    T_{1}T_{5} + T_{2}^2 + T_{3}T_{4},\\
    T_{1}T_{2} - T_{1}T_{4} - T_{1}T_{5} - T_{3}^2 - T_{9}T_{11},\\
    -2T_{1}T_{10} - T_{6}T_{8}T_{9} + T_{7}T_{8}^2 + T_{7}T_{9}^2 + T_{8}^3,
\end{array}
$
\end{minipage}
\quad
\begin{minipage}{5cm}
\tiny
$
\begin{array}{l}
    -2T_{2}T_{10} + T_{6}T_{7}T_{8} + T_{6}T_{8}^2 + T_{6}T_{8}T_{9} - T_{7}^2T_{9},\\
    2T_{1}T_{10} + 2T_{3}T_{10} - 2T_{4}T_{10} + T_{6}^2T_{8} - T_{6}T_{7}T_{9}\\
    +T_{6}T_{8}^2  + 2T_{6}T_{8}T_{9} + T_{6}T_{9}^2 - T_{7}T_{9}^2 - T_{8}^3, \\
    -T_{2}T_{7}T_{8} + T_{4}T_{6}T_{8} - T_{4}T_{7}T_{9} + T_{4}T_{8}^2 + 
    T_{5}T_{6}T_{8}\\
    - T_{5}T_{7}T_{9} + T_{5}T_{8}^2 + T_{5}T_{8}T_{9} + T_{5}T_{9}^2 - 2T_{10}T_{11},\\
    -2T_{1}T_{10} - 2T_{2}T_{10} - 2T_{3}T_{10} + 2T_{4}T_{10} + 2T_{5}T_{10} + \\T_{6}^2T_{9} 
       + T_{6}T_{7}T_{9} - T_{6}T_{8}T_{9} - T_{6}T_{9}^2 
        + T_{7}^3\\
        - T_{7}^2T_{9} + 
        T_{7}T_{9}^2 + T_{8}^3,\\
           2T_{1}T_{10} + 2T_{3}T_{10} + 2T_{6}T_{8}T_{9} + T_{6}T_{9}^2 + T_{7}^2T_{8}\\
           - T_{7}T_{9}^2 - T_{8}^3
\end{array}
$
\end{minipage}
\\[1ex]
    \mbox{\tiny $
    \left[
        \begin{array}{rrrrrrrrrrrrr}
    1 & 1 & 1 & 1 & 1 & 2 & 2 & 2 & 2 &  5 & 0\\
    0 & 0 & 0 & 0 & 0 & -1 & -1 & -1 & -1 & -3 & 1
    \end{array}
    \right]
    $}
\end{center}
(18) The smooth Fano threefold $X_{18}$ 
has the $\ZZ^2$-graded Cox ring
$\KT{6}/I$ with generators for $I$
and the degree matrix given by
\begin{center}
$T_6^2 -f$
\qquad
\mbox{\tiny $
\left[
\begin{array}{rrrrrr}
    1 & 1 & 0 & 0 & 0 & 1 \\
    0 & 0 & 1 & 1 & 1 & 1
    \end{array}
\right]    
$}
\end{center}
 where
$f\in \KT{5}$ is homogeneous  
of degree $(2,2)\in \ZZ^2$
such that $V(f)\subseteq \PP_1\times\PP_2$ is smooth.
\\[1ex]
(19) $\dagger$ The smooth Fano threefold $X_{19}$ 
has the $\ZZ^2$-graded Cox ring
$\KK[T_1,\ldots,T_{8}]/I$
with generators for $I$
and the degree matrix given by
\begin{center}
\begin{minipage}{5cm}
\tiny
$
\begin{array}{l}
 T_{3}^2 + T_{3}T_{4} - T_{6}^2 - T_{7}T_{8},\\
    T_{1}T_{7} + T_{3}T_{4}T_{5} + T_{3}T_{4}T_{6} + T_{3}T_{5}T_{6}\\
    + T_{4}^2T_{5} - T_{4}^2T_{6} + T_{4}T_{5}T_{6} - T_{4}T_{6}^2 \\
    - T_{4}T_{7}T_{8} - T_{5}^2T_{6} -  2T_{6}^3 - T_{6}T_{7}T_{8},\\
    T_{1}T_{3} + T_{2}T_{6} - T_{3}T_{4}T_{8} + T_{4}T_{5}T_{8} + T_{5}T_{6}T_{8},
\end{array}
$
\end{minipage}
\qquad
\begin{minipage}{5cm}
\tiny
$
\begin{array}{l}
    T_{2}T_{7} + 2T_{3}T_{4}^2 + T_{3}T_{4}T_{6} + T_{3}T_{5}^2 + 2T_{3}T_{6}^2 + \\
        T_{3}T_{7}T_{8} - T_{4}T_{5}T_{6} - T_{4}T_{6}^2 - T_{4}T_{7}T_{8} - T_{5}T_{6}^2,\\
    T_{1}T_{6} + T_{2}T_{3} + T_{2}T_{4} - T_{3}T_{4}T_{8} + T_{4}^2T_{8} + T_{5}^2T_{8}\\
    + 2T_{6}^2T_{8} + T_{7}T_{8}^2
\end{array}
$
\end{minipage}
\\[1ex]
    \mbox{\tiny $
    \left[
    \begin{array}{rrrrrrrr}
   1 & 1 & 1 & 1 & 1 & 1 & 2 & 0 \\
   0 & 0 & -1 & -1 & -1 & -1 & -3 & 1
    \end{array}
    \right]
    $}
\end{center}
(20) $\dagger$ The smooth Fano threefold $X_{20}$ 
has the $\ZZ^2$-graded Cox ring
$\KK[T_1,\ldots,T_{8}]/I$ 
with generators for $I$
and the degree matrix given by
\begin{center}
\begin{minipage}{5cm}
\tiny
$
\begin{array}{l}
 T_{1}T_{6} - T_{2}T_{5} - T_{2}T_{6} + T_{2}T_{7} - T_{3}T_{7} \\
 + T_{5}T_{6}T_{8} -  T_{6}T_{7}T_{8},\\
    T_{1}T_{5} - T_{1}T_{7} + T_{2}T_{5} + T_{2}T_{6} - T_{2}T_{7} \\
    - T_{3}T_{6} + 2T_{3}T_{7} -   T_{4}T_{7} \\
    - T_{5}T_{7}T_{8} + T_{7}^2T_{8},\\
    T_{1}^2 - 2T_{1}T_{7}T_{8} - T_{2}T_{3} - T_{2}T_{6}T_{8} \\ + T_{3}T_{6}T_{8} -  T_{3}T_{7}T_{8} + T_{7}^2T_{8}^2,
\end{array}
$
\end{minipage}
\quad
\begin{minipage}{5cm}
\tiny
$
\begin{array}{l}
        T_{1}T_{3} - T_{2}T_{3} + T_{2}T_{4} - T_{2}T_{5}T_{8} \\
        - 2T_{2}T_{6}T_{8} + T_{2}T_{7}T_{8} + T_{3}T_{6}T_{8} \\
        - 2T_{3}T_{7}T_{8} - T_{4}T_{6}T_{8} + T_{5}T_{6}T_{8}^2,\\
        T_{1}T_{4} - T_{2}T_{3} + T_{2}T_{4} - 2T_{2}T_{5}T_{8} -\\ 3T_{2}T_{6}T_{8} + 2T_{2}T_{7}T_{8} + T_{3}^2-\\
        T_{3}T_{5}T_{8} + T_{3}T_{6}T_{8} -T_{3}T_{7}T_{8} -\\ 
        2T_{4}T_{6}T_{8} - T_{4}T_{7}T_{8} + 2T_{5}T_{6}T_{8}^2
\end{array}
$
\end{minipage}
\\[1ex]
    \mbox{\tiny $
    \left[
    \begin{array}{rrrrrrrrrrrrrr}
    1 & 1 & 1 & 1 & 1 & 1 & 1 & 0\\
    0 & 0 & 0 & 0 & -1 & -1 & -1 & 1
    \end{array}
    \right]
    $}
\end{center}
(21) $\dagger$ The smooth Fano threefold $X_{21}$ 
has the $\ZZ^2$-graded Cox ring
$\KK[T_1,\ldots,T_{12}]/I$ 
with generators for $I$
and the degree matrix given by
\begin{center}
\begin{minipage}{5cm}
\tiny
$
\begin{array}{l}
 T_{6}^2 + T_{6}T_{7} + T_{6}T_{8} + T_{6}T_{9} - T_{6}T_{10}\\
 - T_{7}T_{10} + T_{8}T_{9} - T_{9}^2, \\
 T_{3}T_{10} - T_{4}T_{8} + T_{5}T_{7},\\ 
 T_{5}T_{11} + T_{6}T_{8} + T_{7}T_{8} + T_{8}^2 +\\
 T_{8}T_{9} - T_{8}T_{10} - T_{10}^2, \\
 T_{3}T_{11} + T_{7}T_{10} - T_{8}^2 + T_{8}T_{9},\\ 
 T_{4}T_{11} + T_{6}T_{7} + T_{7}^2 + T_{7}T_{8} + T_{7}T_{9}- \\
  T_{7}T_{10} -  T_{8}T_{10} +  T_{9}T_{10}, \\
T_{5}T_{8} - T_{5}T_{9} + T_{11}T_{12},\\ 
T_{3}T_{5} - T_{4}^2 + T_{6}T_{12} + T_{7}T_{12} + T_{8}T_{12}\\
+ T_{9}T_{12} - T_{10}T_{12}, \\
    T_{2}T_{11} + T_{6}T_{10} - T_{7}T_{8} + T_{7}T_{9},\\ 
    T_{3}T_{8} - T_{3}T_{9} - T_{4}T_{7} + T_{5}T_{6} - T_{11}T_{12},\\
    T_{2}T_{8} - T_{2}T_{9} - T_{3}T_{7} + T_{4}T_{6},\\ 
    T_{2}T_{5} - T_{3}T_{4} + T_{10}T_{12},\\
    T_{2}T_{4} - T_{3}^2 + T_{8}T_{12} - T_{9}T_{12},\\ 
    T_{1}T_{11} + T_{6}T_{8} - T_{7}^2,
\end{array}
$
\end{minipage}
\quad
\begin{minipage}{5cm}
\tiny
$
\begin{array}{l}
    T_{1}T_{10} - T_{2}T_{9} + T_{4}T_{6},\\
    T_{1}T_{5} - T_{3}^2 + T_{8}T_{12},\\
    T_{1}T_{4} - T_{2}T_{3} + T_{7}T_{12},\\ 
    T_{1}T_{3} - T_{2}^2 + T_{6}T_{12},\\
    T_{3}T_{6} + T_{3}T_{7} + 2T_{3}T_{9} + T_{4}T_{7}\\
    - T_{4}T_{8} - T_{4}T_{10} - T_{5}T_{6} + T_{5}T_{7} \\
    +T_{2}T_{10} - T_{3}T_{9} + T_{5}T_{6},\\ 
    T_{2}T_{6} + T_{2}T_{7} + 2T_{2}T_{9} + T_{3}T_{7}\\
    - T_{3}T_{9} - T_{4}T_{6} - T_{4}T_{9}\\
    +  T_{5}T_{6} + T_{5}T_{7},\\ 
    T_{1}T_{8} - T_{1}T_{9} - T_{2}T_{7} - \\
    T_{3}T_{7} - 2T_{3}T_{9} - T_{4}T_{7} + T_{4}T_{8} \\
    + T_{4}T_{10} + T_{5}T_{6} - T_{5}T_{7} - \\
    T_{5}T_{8} + T_{5}T_{9} - T_{11}T_{12},\\ 
   T_{1}T_{6} + T_{1}T_{7} + 2T_{1}T_{9} + T_{2}T_{7}\\
   - T_{2}T_{9} + T_{3}T_{7} + T_{3}T_{9} \\
   + T_{4}T_{6} + 2T_{4}T_{7} - T_{4}T_{8} - T_{4}T_{10}- \\
   T_{5}T_{6} + T_{5}T_{7} + T_{5}T_{8}  - T_{5}T_{9} + T_{11}T_{12}
\end{array}
$
\end{minipage}
\\[1ex]
    \mbox{\tiny $
    \left[
    \begin{array}{rrrrrrrrrrrrr}
    1 & 1 & 1 & 1 & 1 & 2& 2 &  2& 2 & 2 &  3 & 0\\
    0 & 0 & 0 & 0 & 0 & -1 & -1 & -1 & -1 & -1 & -2 & 1
    \end{array}
    \right]
    $}
\end{center}
(22) $\dagger$ The smooth Fano threefold $X_{22}$ 
has the $\ZZ^2$-graded Cox ring
$\KK[T_1,\ldots,T_{9}]/I$ with generators for $I$
and the degree matrix given by
\begin{center}
\tiny
$
\begin{array}{ll}
-T_{1}T_{4} + T_{2}T_{3} - T_{8}T_{9}, &
    -T_{2}T_{7} + T_{3}^2T_{8} + T_{4}T_{6},\\
    -T_{1}T_{7} + T_{2}T_{4}T_{8} + T_{3}T_{6}, &
    -T_{1}T_{3}T_{8} + T_{2}T_{6} + T_{4}T_{5},\\
    T_{1}T_{6} - T_{2}^2T_{8} + T_{3}T_{5}, &
    -T_{2}T_{4}^2 + T_{3}^3 - T_{7}T_{9},\\
    T_{1}T_{3}^2 - T_{2}^2T_{4} - T_{6}T_{9}, &
    -T_{1}^2T_{3} + T_{2}^3 - T_{5}T_{9},\\
    -T_{1}T_{4}T_{8}^2 + T_{5}T_{7} + T_{6}^2 - T_{8}^3T_{9}
\end{array}
$
\\
\mbox{\tiny $
\left[
\begin{array}{rrrrrrrrr}
1 & 1 & 1 & 1 & 3 & 3 & 3 & 2 & 0\\
0 & 0 & 0 & 0 & -1 & -1 & -1 & -1 & 1
\end{array}
\right]
$}
\end{center}
(23) The smooth Fano threefold $X_{23}$ 
has the $\ZZ^2$-graded Cox ring
$\KK[T_2,\ldots,T_{8}]/I$ with generators for $I$
and the degree matrix given by
\begin{center}
\begin{minipage}{5cm}
$
\begin{array}{l}
g_2(g_1',T_2,\ldots,T_5),\\ 
g(g_1'-T_6T_8,T_2,\ldots,T_5)
\end{array}
$
\end{minipage}
\ 
\begin{minipage}{5cm}
\mbox{\tiny $
\left[
\begin{array}{rrrrrrr}
    1 & 1 & 1 & 1 & 1 & 2 & 0 \\
    0 & 0 & 0 & 0 & -1 & -1 & 1
    \end{array}
\right]
$}
\end{minipage}
\end{center}
where $g,g_2\in \KT{5}$
are classically homogeneous with
$\deg(g)=\deg(g_2)=2$ and $g_1'\in \KK[T_2,\ldots,T_5]$
is homogeneous of degree $\deg(g_1')=1$
such that both $V(g)\subseteq \PP_4$ 
and $V(T_1 + g_1',g_2,g)\subseteq \PP_4$ are smooth.
\\[1ex]
(24) The smooth Fano threefold $X_{24}$ 
has the $\ZZ^2$-graded Cox ring
$\KT{6}/I$ with a generator for $I$
and the degree matrix given by
\begin{center}
$f$
\qquad
\mbox{\tiny $
\left[
    \begin{array}{rrrrrr}
    1 & 1 & 1 & 0 & 0 & 0\\
    0 & 0 & 0 & 1 & 1 & 1
    \end{array}
\right]
$}
\end{center}
where $f\in \KT{6}$ is homogeneous
of degree  $(1,2)\in \ZZ^2$
and $V(f)\subseteq \PP_2\times \PP_2$ 
is smooth.
\\[1ex]
(25) The smooth Fano threefold $X_{25}$ 
has the $\ZZ^2$-graded Cox ring
$\KT{7}/I$ with generators for $I$
and the degree matrix given by
\begin{center}
\begin{minipage}{4cm}
$
g_1 - T_5T_7,\ 
g_2 - T_6T_7
$
\end{minipage}
\ 
\begin{minipage}{5cm}
\mbox{\tiny $
\left[
\begin{array}{rrrrrrr}
    0 & 0 & 0 & 0 & -1 & -1 & 1 \\
    1 & 1 & 1 & 1 & 1 & 1 & 1
    \end{array}
\right]    
$}
\end{minipage}
\end{center}
where the $g_i\in \KT{4}$ are homogeneous of degree $\deg(g_i)=2$
and $V(g_1,g_2)\subseteq \PP_3$
is an elliptic curve.
\\[1ex]
(26) $\dagger$ The smooth Fano threefold $X_{26}$ has the 
$\ZZ^2$-graded Cox ring
$\KT{8}/I$
with generators for $I$ and the degree matrix given by
\begin{center}
\begin{minipage}{3.5cm}
\tiny
$
\begin{array}{l}
  T_{2}T_{7} - T_{4}T_{6} + T_{5}^2,\\
    T_{1}T_{7} - T_{3}T_{6} + T_{4}T_{5} +\\
    T_{4}T_{7}T_{8} - T_{5}T_{7}T_{8} +  T_{7}^2T_{8}^2,\\
    T_{3}T_{5} + T_{3}T_{6} - T_{4}^2 - T_{4}T_{5} - \\
    T_{4}T_{7}T_{8} + T_{5}T_{7}T_{8} -
        T_{7}^2T_{8}^2,
\end{array}
$
\end{minipage}
\ 
\begin{minipage}{3.5cm}
\tiny
$
\begin{array}{l}
  T_{1}T_{5} + T_{1}T_{6} - T_{2}T_{4} + T_{4}T_{5}T_{8} \\
    - T_{5}^2T_{8} +  T_{5}T_{7}T_{8}^2,\\
    T_{1}T_{4} - T_{1}T_{6} - T_{2}T_{3} + T_{2}T_{4} +\\
    T_{4}^2T_{8} - 2T_{4}T_{5}T_{8} + \\
        T_{4}T_{7}T_{8}^2 + T_{5}^2T_{8} - T_{5}T_{7}T_{8}^2
\end{array}
$
\end{minipage}
\\[1ex]
    \mbox{\tiny $
    \left[
    \begin{array}{rrrrrrrr}
     1 & 1 & 1 & 1 & 1 & 1 & 1 & 0\\
      0 & 0 & -1 & -1 & -1 & -1 & -2 & 1
    \end{array}
    \right]
    $}
\end{center}
(27) The smooth Fano threefold $X_{27}$ 
has the $\ZZ^2$-graded Cox ring
$\KT{8}/I$ with generators for $I$
and the degree matrix given by
\begin{center}
\tiny 
$
\begin{array}{ll}
    T_{2}T_{5} - T_{3}T_{6} + T_{4}T_{7}, &
    -T_{2}T_{4} + T_{3}^2 - T_{5}T_{8}, \\
    T_{1}T_{5} - T_{2}T_{6} + T_{3}T_{7}, &
    -T_{1}T_{4} + T_{2}T_{3} - T_{6}T_{8}, \\
    -T_{1}T_{3} + T_{2}^2 - T_{7}T_{8}
\end{array}
$
\\[1ex]
    \mbox{\tiny $
    \left[
    \begin{array}{rrrrrrrr}
    1 & 1 & 1 & 1 & 2 & 2 & 2 & 0\\
    0 & 0 & 0 & 0 & -1 & -1 & -1 & 1
    \end{array}
    \right]
    $}
\end{center}
(28) The smooth Fano threefold $X_{28}$ 
has the $\ZZ^2$-graded Cox ring
$\KT{6}/I$ with a generator for $I$
and the degree matrix given by
\begin{center}
$f - T_4T_6$
\qquad
\mbox{\tiny $
\left[
\begin{array}{rrrrrr}
    0 & 0 & 0 & -1 & -1 & 1\\
    1 & 1 & 1 & 0 & 2 & 1
    \end{array}
\right]
$}
\end{center}
where $f\in \KT{3}$ is homogeneous
of degree three and $V(f,T_4)\subseteq \PP_3$ 
is smooth.
\\[1ex]
(29) The smooth Fano threefold $X_{29}$ 
has the $\ZZ^2$-graded Cox ring
$\KK[T_1,\ldots,T_{6}]/I$ with generators for $I$
and the degree matrix given by
\begin{center}
$h + g'(T_4T_6, T_5T_6)$ 
\qquad
\mbox{\tiny $
\left[
\begin{array}{rrrrrrrrr}
    1 & 1 & 1 &  1 & 1  & 0  \\
    0 & 0 & 0 & -1 & -1 & 1
\end{array}
\right]
$}
\end{center}
where $g := h + g'\in \KT{5}$ is classically homogeneous of degree $2$
with $g'\in \<T_4,T_5\>$,
$h\in \<T_1,T_2,T_3\>$
and $g$ defines a smooth conic.
\\[1ex]
(30) The smooth Fano threefold $X_{30}$ 
has the $\ZZ^2$-graded Cox ring
$\KT{6}/I$ with a generator for $I$
and the degree matrix given by
\begin{center}
$f - T_4T_6$ 
\qquad
\mbox{\tiny $
\left[
    \begin{array}{rrrrrr}
    0 & 0 & 0 & -1 & -1 & 1\\
    1 & 1 & 1 & 0 & 1 & 1
    \end{array}
\right]
$}
\end{center}
where $f\in \KT{3}$ is homogeneous
of degree two and $V(f)\subseteq \PP_2$ is smooth.
\\[1ex]
(31) The smooth Fano threefold $X_{31}$ 
has the $\ZZ^2$-graded Cox ring
$\KK[T_1,\ldots,T_{6}]/I$
 with a generator for $I$
and the degree matrix given by
\begin{center}
\begin{minipage}{5cm}
\tiny
$
\begin{array}{l}
T_1
g_1(T_1T_6, T_2T_6, T_3T_6, T_4, T_5)\\
+
T_2 
g_2(T_1T_6, T_2T_6, T_3T_6, T_4, T_5)\\
+
T_3
g_3(T_1T_6, T_2T_6, T_3T_6, T_4, T_5)
\end{array}
$
\end{minipage}
\qquad 
\begin{minipage}{5cm}
\mbox{\tiny $
\left[
\begin{array}{rrrrrrrrr}
    1 & 1 & 1 & 1 & 1 &  0\\
    -1 & -1 & -1 & 0 & 0 &  1 
    \end{array}
\right]
$}
\end{minipage}
\end{center}
where $g_1,g_2,g_3\in \KK[T_1,\ldots,T_5]$
are homogeneous of degree 
 $\deg(g_i)=1$ such that $V(T_1g_1+T_2g_2+T_3g_3)\subseteq \PP_4$ is smooth.
\\[1ex]
(32) The smooth Fano threefold $X_{32}$ 
has the $\ZZ^2$-graded Cox ring
$\KT{6}/I$ with a generator for $I$
and the degree matrix given by
\begin{center}
$f$
\qquad
\mbox{\tiny $
\left[
    \begin{array}{rrrrrr}
    1 & 1 & 1 & 0 & 0 & 0\\
    0 & 0 & 0 & 1 & 1 & 1
    \end{array}
\right]
$}
\end{center}
where $V(f)\subseteq \PP_2\times \PP_2$
is smooth 
and $f$ is of degree  $(1,1)\in \ZZ^2$.
\end{theorem}

For the proof of Theorem~\ref{thm:fano2}, note that each Fano threefold
$X_i$ has been described in~\cite{AG5} as a blow-up $X_i\to X$ along
a subvariety $C\subseteq X$ where we know the Cox
ring of $X$, e.g., from Theorem~\ref{thm:fano1}.
For the $\dagger$-cases, we present a concrete choice for $C$
and directly apply the procedure \tt{blowupCEMDS} 
from Section~\ref{sec:example}; see Example~\ref{ex:X19} for 
an example computation and~\cite{DeHaHeKeLa} for all used code.
For the remaining $X_i$,  
we apply~\cite[Algorithm~5.4]{HaKeLa} by hand as 
explained in Section~\ref{sec:example}; the following remark summarizes the steps.

 \begin{remark}
 \label{rem:steps}
Consider a Mori dream space $X_1$
and an irreducible subvariety $C \subseteq X_1$ 
that is contained in the
smooth locus $X_1^{\rm reg}$.
Assume we know $\Cl(X_1)$-prime generators $f_1,\ldots,f_l\in \Cox(X_1)$
for the vanishing ideal $I(p^{-1}(C))\subseteq \Cox(X_1)$
where $p\colon\widehat X_1 \to X_1$ is the quotient 
by the characteristic quasitorus. 
Let $d_1,\ldots,d_l\in \ZZ_{\geq 1}$ be coprime;
we have discussed them on page~\pageref{eq:multiplic}.
In geometric terms, $d_i$ is the multiplicity
of the prime divisor defined by $f_i$
at the generic point of $C$.
We now summarize the steps of~\cite[Algorithm~5.4]{HaKeLa}. 
\begin{itemize}\label{list:blowupcemdssteps}
\item
Assume $\Cox(X_1) = \KT{r_1}/I_1$ with
degree matrix $Q_1$.
Set $w_i :=\deg(f_i)$ for $1\leq i\leq l$
and $r_2 := r_1+l+1$.
Consider the ideals and matrix
\begin{gather*}
 \qquad
 I_2'
 \ :=\ 
 I_1 +
 \<
T_{r_1+j}T_{r_2}^{d_j} - f_j 
 \>,\quad
 I_2
 \ :=\ 
 I_2':T_{r_2}^\infty
 \quad \subseteq\ 
 \KT{r_2},\\
 Q_2
 \ :=\ 
 \left[
 \mbox{\tiny
 $
 \begin{array}{c|rrrrr}
  Q_1 & w_1 & \dots & w_l & 0 \\
  \hline
  0 & -d_1 & \dots & -d_l & 1
 \end{array}
 $}
 \right].
\end{gather*}
\item
Let $T^\nu$ be the product over all $T_i$ with $C\not\subseteq V(X_1;\,T_i)$.
Test whether we have 
$$
\dim(I_2 +\<T_{r_2}\>)\ >\ \dim(I_2 + \<T_{r_2},T^\nu\>).
$$
\end{itemize}
Then the blow-up $X_2\to X_1$ of $X_1$ along $C$ has 
the $\Cl(X_1)\oplus \ZZ$-graded Cox ring
$\Cox(X_2) = \KT{r_2}/I_2$
with degree matrix~$Q_2$.
\end{remark}
 
In that notation, the computation of $I_2$ becomes simple
for the case of a prime ideal $I_2'\subseteq\KT{r_2}$
 since we then have $I_2 = I_2'$.
For the case of a complete intersection ring $\KT{r_2}/I_2'$,
we will use the following lemma.
Here, we assume the grading to be {\em pointed\/}, i.e.,
the cone over $\deg(T_1),\ldots,\deg(T_{r_2})$ is pointed
and all $\deg(T_i)$ are non-zero.

\begin{lemma}
\label{lem:serre}
Set $R:=\KT{r}$.
Consider an ideal $I\subseteq R[S_1,\ldots,S_n]$
generated by polynomials $h_1,\ldots,h_s$ of shape
\[
h_i\ =\ 
g_i - g_i'
,\qquad
g_i\ \in\ R
,\qquad
g_i'\ \in\ \KK[S_1,\ldots,S_n]
\]
such that the 
$g_i$ are classically homogeneous,
$V(g_1,\ldots,g_s)\subseteq \PP_{r-1}$ is smooth
and of dimension~$r-s-1$.
Assume $I$ is homogeneous with respect to a pointed grading.
Then both $\<g_1,\ldots,g_s\>\subseteq R$ and 
$I\subseteq R[S_1,\ldots,S_n]$ are prime.
\end{lemma}

\begin{proof}
We use Serre's criterion, 
see~\cite{Kra}.
Write $\b X := V(I)\subseteq \KK^{r+n}$
and let $J := (\partial h_i/\partial T_j)_{i,j}$
be the Jacobian matrix.
We show that the closed set
$$
A
\ :=\ 
\b X
\ \cap\ 
\{
z\in \KK^{r+n};\ \rank(J(z) )\ <\ s
\}
\ \subseteq \ \b X_2
$$
is of codimension at least two in~$\b X$.
Since the $g_i\in \KT{r}$ are homogeneous 
and $V(g_1,\ldots,g_s)\subseteq \PP_{r-1}$ is smooth
the singular locus of the affine cone satisfies
$V(\KK^{r};\,g_1,\ldots,g_s)^{\rm sing}\subseteq \{0\}$.
Therefore, 
the first $s\times r$ submatrix of 
\[
J\ =\ 
\left[
\mbox{\tiny $
\begin{array}{rrrrrr}
\frac{\partial g_1}{\partial T_1}
& 
\ldots
&
\frac{\partial g_1}{\partial T_{r}}
&
\ldots
\\
\vdots
&
&
\vdots
&
\ldots
\\
\frac{\partial g_2}{\partial T_1}
& 
\ldots
&
\frac{\partial g_2}{\partial T_r}
&
\ldots
\end{array}
$}
\right]
\]
and hence the matrix $J$
is of rank $s$ 
on $\KK^r\setminus\{0\}\times \KK^n$.
This shows that $A$ is small in $\b X$.
As the grading is pointed, $\b X$ is connected
and Serre's criterion delivers that $I$ is prime.
The same argument holds for $\<g_1,\ldots,g_s\>$.
\end{proof}

\begin{proof}[Proof of Theorem~\ref{thm:fano2}]
{\em Case $X_1$:\/}
Consider the variety $X$ 
listed in Theorem~\ref{thm:fano1}
for the case $(r,-K_X^3)=(2,8)$, i.e., we have
$X=V(f)\subseteq \PP(3,2,1,1,1)$
where $f\in \KT{5}$ is homogeneous of degree six. 
Then $X_1$ is the blow-up of $X$ 
along an elliptic curve $C\subseteq X$
that is the intersection of
two divisors of degree one. 
Without loss of generality,
applying a linear change of coordinates
yields
\begin{gather*}
f\ =\ 
T_1^2 + T_2^3 + T_2f_4 + f_6\qquad
\text{where}\quad 
f_d\,\in\, \KK[T_3,T_4,T_5],\qquad
\deg(f_d)\ =\ d,\\
C\ =\ V(T_3-T_5,\,T_4-T_5,f)\ \subseteq\ \PP(3,2,1,1,1).
\end{gather*}
We now apply the steps of Remark~\ref{rem:steps}
with $f_1 = T_3-T_5$, $f_2 = T_4-T_5$ and $d_1 = d_2 = 1$.
The ideal $I_2'$ then is generated by
\[
f,\qquad
T_6T_8 - T_3 + T_5,\qquad
T_7T_8 - T_4 + T_5
\quad 
\in
\quad
\KK[T_1,\ldots,T_{8}].
\]
We show that $I_2' = I_2': T_8^\infty$
by showing that $I_2'$ is prime.
Replacing $T_3$ with $T_6T_8 + T_5$ and $T_5$ with $-T_7T_8 + T_4$,
we may show equivalently that 
$$
g\ :=\ 
T_1^2
+ 
T_2^3 
+
T_2 g_4
+
g_6
,\qquad
g_i\ :=\ f_i(T_4 -T_7T_8 + T_6T_8, T_4, T_4 - T_7T_8)
$$
is a prime polynomial in the ring $R := \KK[T_1,T_2,T_4,T_6,T_7,T_8]$.
Suppose $g=ab$ with $a,b\in R$.
Consider the $\ZZ_{\geq 0}$-grading on $R$ given by  $\deg(T_1):=\deg(T_2):=0$ and $\deg(T_j):=1$
for $j\ne 1,2$.
Then $a_0b_0 = T_1^2+T_2^3$ allows us to assume
$a_0 = T_1^2 + T_2^3$ and $b_0=1$.
As $T_2g_4$ and $g_6$ are independent of $T_1$ and $T_2$,
we obtain $b=1$, i.e., $I_2'$ is prime.
Moreover, for $T^\nu := T_1\cdots T_5$ we have
$$
\dim(I_2' +\<T_{8}\>)
\ =\ 
4
\ >\ 
\dim(I_2' + \<T_{8},T^\nu\>)
\ =\ 
3
.
$$
By the steps of Remark~\ref{rem:steps},
$\KT{8}/I_2$ is the Cox ring of~$X_1$.
Performing the previous replacements of $T_3$ and $T_5$
we arrive at $\Cox(X_1) = R/\<g\>$.

{\em Cases $X_i$ with $i\in\{2,8,18\}$:\/}
All these cases are double covers with branch divisor $V(f)$
where $f$ is as shown in the list of the theorem.
The Cox rings are obtained using Lemma~\ref{lem:doublecover}:
$X_2$ and $X_{18}$ are double covers of $\PP_1\times \PP_2$
whereas case $X_{8}$ is a double cover of $Y$ where 
$Y\to \PP_3$ is the blow-up of a point, i.e.,
\[
\Cox(Y)\ =\ \KT{5},\qquad
\left[
\mbox{
\tiny $
\begin{array}{rrrrr}
1 & 1 & 1 & 1 & 0\\
-1 & -1 & -1 & 0 & 1
\end{array}
$
}
\right].
\]

{\em Case $X_3$:\/}
This is the blow-up of the variety 
$Y$ of Theorem~\ref{thm:fano1} with $(r,-K_X^3)=(2,16)$
along an elliptic curve 
that is the intersection of two divisors
$D_1,D_2\subseteq V(T_5^2 -f)$ 
of degree 
$-1/2\cdot w_{Y}^{\rm can} = 1$
where 
$-w_{Y}^{\rm can}\in \Cl(Y)=\ZZ$ 
is the anticanonical divisor class.
By a linear change of coordinates we 
achieve $D_1 = V(T_1)$ and $D_2 = V(T_2)$.
We now apply the steps listed in
Remark~\ref{rem:steps}
with $f_1 = T_1$, $f_2 = T_2$ and $d_1 = d_2 = 1$.
The ideal $I_2'$ then is generated by
\[
T_5^2 - f,\qquad
T_1 - T_6T_8,\qquad
T_2 - T_7T_8
\quad 
\in
\quad
\KK[T_1,\ldots,T_{8}].
\]
We show that $I_2' = I_2': T_8^\infty$
by showing that $I_2'$ is prime.
Equivalently, we may show that 
$$
T_5 ^2 - f'
\ \in\ 
\KK[T_3,\ldots,T_8],
\qquad 
f'\ :=\
f(T_6T_8, T_7T_8, T_3,T_4)
$$ 
is prime.
Otherwise, $f'$ must be a square.
In particular, $f(T_6,T_7,T_3,T_4)$ is a square.
This contradicts the choice of~$f$.
Moreover, for $T^\nu := T_3T_4T_5$ we have
$$
\dim(I_2' +\<T_{8}\>)
\ =\ 
4
\ >\ 
\dim(I_2' + \<T_{8},T^\nu\>)
\ =\ 
3
.
$$
By the steps of Remark~\ref{rem:steps},
$\KT{8}/I_2$ is the Cox ring of~$X_3$.
Substitution of
$T_1 = T_6T_8$ and $T_2 = T_7T_8$ into $T_5^2 - f_2$
delivers 
\[
\Cox(X_3)
\ =\ 
\KK[T_3,\ldots,T_8]/
\<T_5^2 - f_2(T_6T_8, T_7T_8, T_3,T_4)\>.
\]

{\em Cases $X_i$ with $i\in \{4,15,25\}$:\/}
We will exemplarily compute $\Cox(X_4)$;
the other two Cox rings can be computed analogously
using the polynomials $g_i$ listed in the table.
The case $X_4$ is the blow-up of $\PP_3$
along the smooth intersection of two cubics  $V(g_i)\subseteq \PP_3$.
Applying the steps listed in Remark~\ref{rem:steps}
delivers the ideal $I_2'$ generated by 
\[
g_1 - T_5T_7,\qquad
g_2 - T_6T_7
\ \in\ \KK[T_1,\ldots,T_7]
\]
where the $g_i$ are as shown in the theorem.
We have $I_2' = I_2':T_7^\infty$
 for $I_2'$ is prime by Lemma~\ref{lem:serre}.
By the steps of Remark~\ref{rem:steps},
$\KT{7}/I_2'$ is the Cox ring of~$X_4$:
setting $T^\nu := T_1\cdots T_4$, we have
$$
\dim(I_2' +\<T_{7}\>)
\ =\ 
4
\ >\ 
\dim(I_2' + \<T_{7},T^\nu\>)
\ =\ 
3
.
$$

{\em Case $X_5$:\/}
This is the blow-up of the variety $X$
listed in Theorem~\ref{thm:fano1}
for $(r,-K_X^2)=(2,24)$ along a plane cubic.
After a linear change of coordinates
this means we want to blow-up $X=V(f)\subseteq \PP_4$
in $C := V(T_4,T_5,f)\subseteq V(f)$
with a homogeneous polynomial $f\in\KT{5}$
of degree three such that $V(f)\subseteq \PP_4$ is smooth.
By the steps of Remark~\ref{rem:steps},
we obtain an ideal $I_2'$ generated by
\[
f,\qquad
T_4 - T_6T_8,\qquad
T_5 - T_7T_8
\quad 
\in
\
\KK[T_1,\ldots,T_{8}].
\]
As $V(f,T_4,T_5)\subseteq \PP_4$ is smooth
the ideal $I_2'$ is prime.
In particular $I_2' = I_2' : T_8^\infty$.
Note that the dimension test succeeds:
setting $T^\nu := T_1\cdots T_3$, we have
$$
\dim(I_2 +\<T_{8}\>)
\ =\ 
4
\ >\ 
\dim(I_2 + \<T_{8},T^\nu\>)
\ =\ 
3
.
$$
By the steps listed in Remark~\ref{rem:steps}, 
the Cox ring of $X_{5}$ is $\KT{8}/I_2'$.
Removal of redundant generators yields
$$
\Cox(X_5)\ =\ \KK[T_1,\ldots,T_3,T_6,\ldots,T_8]/\<g\>,\qquad
 g
 \ :=\ 
 f(T_1,\ldots,T_3,T_6T_8, T_7T_8).
$$

{\em Cases $X_i$ with $i\in \{6a, 24, 32\}$:\/}
This is a prime divisor on $\PP_2\times \PP_2$ of 
degree $\nu\in \ZZ^2$
where $\nu = (2,2)$, $(1,2)$ or $(1,1)$ in the respective cases, see~\cite[Table~12.3]{AG5}.
Let $f$ be a homogeneous polynomial 
of degree $\nu\in \ZZ^2$
such that $V(f)$ is smooth.
By~\cite[Corollary~4.1.1.3]{ArDeHaLa}
the Cox ring is
\[
\Cox(X_{i})\ =\ \KT{6}/\<f\>,
\qquad
\left[
\mbox{\tiny $
\begin{array}{rrrrrr}
1 & 1 & 1 & 0 & 0 & 0\\
0 & 0 & 0 & 1 & 1 & 1
\end{array}
$}
\right].
\]

{\em Case $X_{6b}$:\/}
This is a double cover of a smooth 
divisor $W$
on $\PP_2\times \PP_2$ of degree $(1,1)\in \ZZ^2$
with branch locus a divisor in $|-K_W|$.
By~\cite[Corollary~4.1.1.3]{ArDeHaLa},
the Cox ring and degree matrix of $W$ are
$$
\Cox(W)\ =\ \KT{6}/\<f\>,
\qquad
\left[
\mbox{\tiny $
\begin{array}{rrrrrr}
1 & 1 & 1 & 0 & 0 & 0\\
0 & 0 & 0 & 1 & 1 & 1
\end{array}
$}
\right]
$$
where $f$ is as in the table of the theorem.
Picking a homogeneous polynomial $g\in \KT{6}$ 
of degree $(2,2)\in \ZZ^2$,
by Lemma~\ref{lem:doublecover}, the Cox ring
and its degree matrix are 
\[
\Cox(X_{6b})
\ =\ 
\KT{7}/\<f, T_7^2 -g\>,
\qquad
\left[
\mbox{\tiny $
\begin{array}{rrrrrrr}
1 & 1 & 1 & 0 & 0 & 0 & 1\\
0 & 0 & 0 & 1 & 1 & 1 & 1
\end{array}
$}
\right].
\]

{\em Case $X_7$:\/}
Let $f,g_1,g_2\in \KT{5}$ be classically homogeneous
of degree $\deg(f)=\deg(g_i)=2$ 
such that $V(f,g_1,g_2)\subseteq \PP_4$ is smooth.
Then $X_7$ is  the blow-up of $V(f)$ along $V(g_1,g_2)\subseteq V(f)$.
By the steps listed in Remark~\ref{rem:steps},
we obtain an ideal $I_2'$ generated by
\[
f,\qquad
g_1 - T_6T_8,\qquad
g_2 - T_7T_8
\quad 
\in
\quad
\KK[T_1,\ldots,T_8].
\]
By Lemma~\ref{lem:serre}, $I_2'$ is prime.
In particular $I_2' = I_2':T_8^\infty$.
By the steps listed in Remark~\ref{rem:steps},
$\KT{8}/I_2'$ is the Cox ring of~$X_7$:
setting $T^\nu := T_1\cdots T_5$, we have
$$
\dim(I_2' +\<T_{8}\>)
\ =\ 
4
\ >\ 
\dim(I_2' + \<T_{8},T^\nu\>)
\ =\ 
3
.
$$

{\em Case $X_{9}$:\/}
This is the blow-up of $\PP_3$ along a curve $C\subseteq \PP_3$ 
of degree seven and of genus five such that 
$C$ is an intersection of cubics.
The Cox ring listed in the table has been computed
for $C\subseteq\PP_3$ with
$I(C)\subseteq \KT{4}$ generated by

\begingroup
\tiny
\begin{center}
$
\begin{array}{l}
    T_{1}^2T_{3} + T_{2}^2T_{3} + T_{3}^3 - T_{1}^2T_{4} + T_{2}^2T_{4} - 2T_{3}^2T_{4} + T_{3}T_{4}^2 - 3T_{4}^3,\\
    T_{2}^3 + 5/2T_{2}^2T_{3} - 1/2T_{2}T_{3}^2 + T_{3}^3 + 1/2T_{2}^2T_{4} - 5/2T_{3}^2T_{4} - T_{2}T_{4}^2  + 1/2T_{3}T_{4}^2 - 7/2T_{4}^3,\\
    T_{1}^2T_{2} - 5/2T_{2}^2T_{3} + 3/2T_{2}T_{3}^2 - T_{3}^3 + 3T_{1}^2T_{4} - 5/2T_{2}^2T_{4} + 7/2T_{3}^2T_{4}  + 2T_{2}T_{4}^2 - 1/2T_{3}T_{4}^2 + 11/2T_{4}^3.
\end{array}    
$
\end{center}
\endgroup

{\em Case $X_{10}$:\/}
This is the blow-up of $X\subseteq\PP_5$ 
from Theorem~\ref{thm:fano1}, case $(r,-K_X^3)=(2,32)$,
along an elliptic curve $C$ that 
is an intersection of two hyperplane sections;
after a linear change of coordinates,
we may choose $C:=V(T_5,T_6)\cap X$.
Applying the steps of 
Remark~\ref{rem:steps},
provides us with the ideal $I_2'$ generated by
\[
T_5-T_7T_{9},\quad 
T_6-T_8T_{9},\quad
f_1,\quad
f_2\quad
\in
\quad
\KK[T_1,\ldots,T_9]
\]
where the $f_i$ are as in Theorem~\ref{thm:fano1}.
Observe that $I_2'$ is prime
since 
$V(f_1,f_2,T_5,T_6)\subseteq \PP_5$ is smooth,
see Lemma~\ref{lem:serre}.
In particular, $I_2' = I_2':T_{9}^\infty$.
Note that the dimension test succeeds:
setting $T^\nu := T_1\cdots T_4$, we have
$$
\dim(I_2 +\<T_{9}\>)
\ =\ 
4
\ >\ 
\dim(I_2 + \<T_{9},T^\nu\>)
\ =\ 
3
.
$$
By the steps of Remark~\ref{rem:steps}, the Cox ring of $X_{10}$
is $\KT{9}/I_2$.
Removal of redundant generators yields
\begin{gather*}
\Cox(X_{10}) 
\ =\ 
\KK[T_1,\ldots,T_4,T_7,\ldots,T_{9}]/I_2,
\\
I_2\ :=\ 
\<
f_1(T_1,\ldots,T_4,T_7T_{9},T_8T_{9}),\ 
f_2(T_1,\ldots,T_4,T_7T_{9},T_8T_{9})
\>
.
\end{gather*}

{\em Case $X_{11}$:\/} 
This is the blow-up of the variety $X$ from  Theorem~\ref{thm:fano1}, case $(r,-K_X^3)=(2,24)$, 
along a line.
We may assume the following.
Let $g_1,g_2,g_3\in \KT{5}$ be classically
homogeneous with $\deg(g_{i})=2$
such that $V(g)\subseteq \PP_4$ is smooth
where $g := T_1g_1 + T_2g_2 + T_3g_3$.
Then $X_{11}$ is the blow-up of
$X=V(g)\subseteq \PP_4$ 
along the line $V(T_1,T_2,T_3)\subseteq X$.
The steps of Remark~\ref{rem:steps}
provide us with the ideal $I_2'$ generated by
\[
g,\qquad
T_1 - T_6T_9,\qquad
T_2 - T_7T_9,\qquad
T_3 - T_8T_9
\quad 
\in
\quad
\KK[T_1,\ldots,T_{9}].
\]
We now show that $I_2 = I_2':T_9^\infty$
where
\begin{gather*}
I_2 
\ :=\ 
\<
T_1 - T_6T_9,\ 
T_2 - T_7T_9,\ 
T_3 - T_8T_9,\ 
h
\>
\ \subseteq\ 
\KT{9}.
\\
h \ :=\ 
T_6g_1(T_6T_9, T_7T_9, T_8T_9, T_4, T_5)
+
T_7g_2(T_6T_9, T_7T_9, T_8T_9, T_4, T_5)
\\
\hphantom{h\ :=\ }
+
T_8g_3(T_6T_9, T_7T_9, T_8T_9, T_4, T_5).
\end{gather*}
Note that it suffices to show that $I_2$ is prime.
Equivalently, we may show that the last generator $h$
is a prime element in $\KK[T_4,\ldots,T_9]$.
This is the case since $V(h)$ is the strict transform of $V(g)$.
The dimension test is satisfied:
setting $T^\nu := T_4T_5$, we have
$$
\dim(I_2 +\<T_{9}\>)
\ =\ 
4
\ >\ 
\dim(I_2 + \<T_{9},T^\nu\>)
\ =\ 
3
.
$$
By the steps of Remark~\ref{rem:steps}, the Cox ring of $X_{11}$
is $\KT{9}/I_2$.
Removal of redundant generators yields
$$
\Cox(X_{11})
\ =\ 
\KK[T_4,\ldots,T_9]/ \<h\>.
$$

{\em Case $X_{12}$:\/}
This is the blow-up of $\PP_3$ along a curve $C$
of degree six and of genus three such that $C$
is an intersection of cubics.
We have chosen the generators $f_1,\ldots,f_4$ for the 
ideal $I(C)\subseteq \KT{4}$ as
\begingroup
\footnotesize
\begin{gather*}
T_{1}^3 - T_{1}T_{2}T_{3} + T_{1}T_{2}T_{4} + T_{3}^2T_{4},\qquad
T_{1}^2T_{2} - T_{2}^2T_{3} + T_{1}T_{3}T_{4} + T_{3}T_{4}^2,\\
T_{1}T_{2}^2 - T_{1}^2T_{3} + T_{2}T_{3}^2 - T_{1}T_{3}T_{4}, \qquad
T_{2}^3 - T_{1}^2T_{4} + T_{3}^2T_{4} - T_{2}T_{4}^2.
\end{gather*}
\endgroup

Note that in order to compute the listed Cox ring
one has to add an additional generator of $I^3:J^\infty$.
Here, the input for Remark~\ref{rem:steps}
is $f_1,\ldots,f_4$ with multiplicities $d_i = 1$ and the polynomial $f_5$ with $d_5 = 3$ given by
\begin{center}
$
\begin{array}{l}
T_1f_5\ =\ 
 f_3^2f_4 - f_1f_3^2 - f_3f_2^2 - f_4f_1^2 +f_1^3
 \ \in\ \KT{4}.
\end{array}
$
\end{center}

{\em Case $X_{13}$:\/}
This is the blow-up of a smooth quadric
$Q\subseteq \PP_4$ along 
 a curve $C\subseteq Q$ 
of degree six and of genus two.
The Cox ring listed in the table has been computed
for 
$Q=V(T_{2}^2 + T_{3}T_{4} + T_{1}T_{5})\subseteq \PP_4$ 
and we have chosen the following generators for~$I(C)$:
\begingroup
\footnotesize
\begin{gather*}
  T_{2}T_{3} + T_{1}T_{4} + T_{2}T_{4},\qquad
  T_{1}T_{3} + T_{3}^2 + T_{1}T_{4} + T_{3}T_{4} - T_{4}T_{5},\\
  T_{1}^2 - T_{3}^2 - T_{1}T_{4} - T_{3}T_{4} + T_{2}T_{5} + T_{4}T_{5}.
\end{gather*}
\endgroup

{\em Case $X_{14}$:\/}
This is the blow-up of 
the variety $X$, case $(r,-K_X^3)=(2,40)$,
from Theorem~\ref{thm:fano1} along
 an elliptic curve $C\subseteq X$
 that is an intersection of two hyperplane sections.
 We have computed the Cox ring listed in the table for the choice of  
 $C\subseteq X$ with vanishing ideal $I(C)\subseteq \KT{7}$ given by
 \begingroup
 \footnotesize
 \begin{gather*}
T_{2} + T_{7},\qquad 
T_{1} + T_{3},\qquad
-T_{4}T_{5} + T_{3}T_{6} + T_{7}^2,\qquad
T_{3}T_{5} - T_{6}^2 + T_{4}T_{7},\\
T_{4}^2 + T_{5}T_{6} + T_{3}T_{7},\qquad
T_{3}T_{4} + T_{5}^2 + T_{6}T_{7},\qquad
T_{3}^2 - T_{4}T_{6} - T_{5}T_{7}.
\end{gather*}
\endgroup

{\em Case $X_{16}$:\/}
This is the blow-up of 
the variety $X\subseteq \PP_5$ 
listed in Theorem~\ref{thm:fano1}, case $(r,-K_X^3)=(2,32)$, 
along a conic $C\subseteq X$;
we may assume that $C =V(T_1,T_2,T_3,T_5^2-T_4T_6)\subseteq \PP_5$.
The steps of Remark~\ref{rem:steps}
deliver the ideal $I_2'\subseteq \KT{11}$ generated by 
\[
f_1,\quad
f_2,\quad
T_1 - T_7T_{11},\quad
T_2 - T_8T_{11},\quad
T_3 - T_9T_{11},\quad
T_5^2-T_4T_6 - T_{10}T_{11}
\]
where the $f_i\in \KT{6}$ homogeneous polynomials of 
degree two, both
$V(f_i)\subseteq \PP_5$ are smooth
and $Y:=V(f_1,f_2)\subseteq \PP_5$ is smooth and of dimension three,
see Theorem~\ref{thm:fano1}.
Note that there is a relation $T_{10} + g\in I_2':T_{11}^\infty$
with $g\in \KK[T_1,\ldots,T_9,T_{11}]$:
without loss of generality, 
we may assume $f_1 = T_1^2 + T_2^2 + T_3^2 + T_5^2 - T_4T_6$
i.e.,
all monomials different from $T_5^2$ and $T_4T_6$ 
depend on $T_1,T_2$ or $T_3$.
Substituting the other equations, we may cancel 
the factor $T_{11}$; this yields an equation 
that is linear in  $T_{10}$.
We now show that 
$I_2''\subseteq \KT{11}$ is generated by
\[
f_1,\quad
f_2,\quad
T_1 - T_7T_{11},\quad
T_2 - T_8T_{11},\quad
T_3 - T_9T_{11},\quad
T_{10} +g
\]
equals $I_2':(T_1\cdots T_{11})^\infty$ since it is prime.
Note that this is equivalent to the ideal 
$$
I_2'''\ :=\ 
\<
f_1,\
f_2,\
T_1 - T_7T_{11},\
T_2 - T_8T_{11},\
T_3 - T_9T_{11}
\>\ \subseteq\ \KK[T_1,\ldots,T_{9},T_{11}]
$$
being prime;
this follows from Lemma~\ref{lem:serre}.
Eliminating the linear equations, we obtain
\begin{gather*}
 \Cox(X_{16})
\ =\ 
\KK[T_4,\ldots,T_9,T_{11}]/I_2,\\
I_2
\ :=\ 
\<
f_1(T_7T_{11},T_8T_{11},T_9T_{11},T_4,T_5,T_6),\
f_2(T_7T_{11},T_8T_{11},T_9T_{11},T_4,T_5,T_6)
\>.
\end{gather*}

{\em Case $X_{17}$:\/}
This is the blow-up of a smooth quadric $Q\subseteq \PP_4$
along an elliptic curve $C\subseteq V(g)$ of degree five.
To compute the Cox ring,
we have chosen $Q$ as $V(T_{2}^2 + T_{3}T_{4} + T_{1}T_{5})\subseteq \PP_4$
and the subvariety $C\subseteq Q$ with the 
vanishing ideal $I(C)\subseteq \KT{5}$ generated by
\begingroup
\footnotesize
\begin{gather*}
    T_{4}^2 - T_{3}T_{5} + T_{4}T_{5}, \qquad
    T_{2}T_{4} + T_{1}T_{5} + T_{2}T_{5},\qquad
    T_{1}T_{2} - T_{3}^2 - T_{1}T_{4} - T_{1}T_{5}, \\
    T_{2}^2 + T_{3}T_{4} + T_{1}T_{5},\qquad
    T_{2}T_{3} + T_{1}T_{4}.
\end{gather*}
\endgroup
Note that in order to compute the listed Cox ring
one has to add the following additional generator of $I^3:J^\infty$
to the generating set 
 used in the steps listed in Remark~\ref{rem:steps}:
\begingroup
\tiny
\begin{center}
$
\begin{array}{l}
T_{1}T_{2}T_{3}T_{4}^2 + T_{2}T_{3}^2T_{4}^2 - 1/2T_{3}^3T_{4}^2 + 1/2T_{1}^2T_{4}^3 + 
        1/2T_{1}T_{2}T_{4}^3 + T_{1}T_{3}T_{4}^3 + 1/2T_{2}T_{3}T_{4}^3 + 1/2T_{1}T_{4}^4 \\
        - 
        T_{1}T_{2}T_{3}^2T_{5} - 1/2T_{2}T_{3}^3T_{5} + 1/2T_{3}^4T_{5} - 
        3/2T_{1}^2T_{2}T_{4}T_{5} - 1/2T_{1}^2T_{3}T_{4}T_{5} + T_{1}T_{2}T_{3}T_{4}T_{5}
        + 
        1/2T_{1}T_{3}^2T_{4}T_{5}\\
        + T_{2}T_{3}^2T_{4}T_{5} - 1/2T_{3}^3T_{4}T_{5}  + 
        2T_{1}^2T_{4}^2T_{5} + 3/2T_{1}T_{2}T_{4}^2T_{5} + T_{1}T_{3}T_{4}^2T_{5} + 
        T_{2}T_{3}T_{4}^2T_{5} + T_{1}T_{4}^3T_{5} - 1/2T_{1}^3T_{5}^2
        \\
        - 3/2T_{1}^2T_{2}T_{5}^2
        + T_{1}T_{3}^2T_{5}^2 + 3T_{1}^2T_{4}T_{5}^2 + 3/2T_{1}T_{2}T_{4}T_{5}^2 + 
        1/2T_{2}T_{3}T_{4}T_{5}^2 + 1/2T_{1}T_{4}^2T_{5}^2 + 3/2T_{1}^2T_{5}^3 + 
        1/2T_{1}T_{2}T_{5}^3.
\end{array}
$
\end{center}
\endgroup

{\em Case $X_{19}$:\/}
This is the blow-up of 
the variety $X=V(f_1,f_2)\subseteq \PP_5$ 
from Theorem~\ref{thm:fano1},
case $(r,-K_X^3)=(2,32)$,
along a line $C\subseteq X$.
The shown Cox ring has been computed for the
choices
\begingroup
\footnotesize
\begin{eqnarray*}
f_1
&:=& 
T_{1}^2 + T_{2}^2 + T_{3}^2 + T_{4}^2 + T_{1}T_{5} + T_{6}T_{4} + T_{3}T_{6}\quad \in\ \KT{6}, 
\\
f_2
&:=&
T_{1}T_{2} + T_{2}T_{3} - T_{3}T_{4} +T_{4}T_{5} + T_{1}T_{6}
\quad \in\ \KT{6}
\end{eqnarray*}
\endgroup
and we took as generators for the vanishing ideal $I(C)\subseteq \KT{6}$
 the following polynomials
\begingroup
\footnotesize
\begin{gather*}
  T_4,\quad
  T_3,\quad
  T_2,\quad
  T_1,\quad
  T_2^2 + T_3^2 + T_3T_4 + 2T_4^2 + T_1T_5 + T_3T_6 + T_4T_6.
\end{gather*}
\endgroup
Moreover, all multiplicities $d_i$ are $1$ except for the last one, 
which is~$3$.
See Example~\ref{ex:X19} for an explanation.

{\em Case $X_{20}$:\/}
This is the blow-up of 
the variety $X\subseteq \PP_6$ from Theorem~\ref{thm:fano1},
case $(r,-K_X^3)=(2,40)$, 
along a twisted cubic $C\subseteq X$.
The shown Cox ring has been computed for the following generators of the
vanishing ideal $I(C)\subseteq \KT{7}$:
\begingroup
\footnotesize
\begin{gather*}
    T_{5} - T_{6} + T_{7}, \qquad
    T_{2} - T_{4} - T_{6} + T_{7},\qquad
    T_{1} - T_{4} - T_{6} + T_{7}, \qquad
    T_{6}^2 - T_{3}T_{7} + T_{4}T_{7} - T_{6}T_{7},\\
    T_{4}T_{6} - 2T_{4}T_{7} - T_{6}T_{7} + T_{7}^2,\qquad
    T_{3}T_{4} - T_{4}^2 - T_{3}T_{7} - T_{4}T_{7} - T_{6}T_{7} + T_{7}^2.
\end{gather*}
\endgroup

{\em Case $X_{21}$:\/}
This is the blow-up of of a smooth quadric
$V(g)\subseteq \PP_4$
along a twisted quartic $C\subseteq V(g)$.
The shown Cox ring has been computed for the
choices
\begingroup
\footnotesize
\begin{eqnarray*}
g
&:=& 
    T_{2}^2 - T_{1}T_{3} + T_{2}T_{3} + T_{3}^2 - T_{1}T_{4} + T_{2}T_{4} - T_{3}T_{4} - T_{4}^2 - 2T_{1}T_{5} + T_{2}T_{5} + T_{3}T_{5},
\\
C
&:=& 
V(
   T_{2}^2 - T_{1}T_{3},\ 
    T_{2}T_{3} - T_{1}T_{4},
T_{3}^2 - T_{1}T_{5},\ 
    T_{2}T_{4} - T_{1}T_{5},
    T_{3}T_{4} - T_{2}T_{5},\ 
    T_{4}^2 - T_{3}T_{5}
).
\end{eqnarray*}
\endgroup
Note that in order to compute the listed Cox ring
one has to use the following additional generator of $I^2:J^\infty$
with multiplicity two in Remark~\ref{rem:steps}
\begingroup
\tiny
\begin{center}
$
\begin{array}{l}
T_{3}^3 - 2T_{2}T_{3}T_{4} + T_{1}T_{4}^2 - T_{2}T_{3}T_{5} - T_{3}^2T_{5} + T_{1}T_{4}T_{5} - T_{2}T_{4}T_{5} + T_{3}T_{4}T_{5} + T_{4}^2T_{5} + 2T_{1}T_{5}^2 - T_{2}T_{5}^2 - T_{3}T_{5}^2.
\end{array}
$
\end{center}
\endgroup

{\em Case $X_{22}$:\/}
According to~\cite[p.~117]{MoMu},
$X_{22}$ can be obtained as the blow-up 
of $\PP_3$ along a rational quartic curve $C\subseteq \PP_3$.
The shown Cox ring has been computed for the following generators of the
vanishing ideal $I(C)\subseteq \KT{4}$:
\begingroup
\footnotesize
\begin{gather*}
T_{2}^3 - T_{1}^2T_{3}, \qquad
    T_{1}T_{3}^2 - T_{2}^2T_{4},\qquad
    T_{3}^3 - T_{2}T_{4}^2,\qquad
    T_{2}T_{3} - T_{1}T_{4}.
\end{gather*}
\endgroup

{\em Case $X_{23}$:\/}
Consider a smooth quadric $Q:=V(g)\subseteq\PP_4$.
Then $X_{23}$ is the blow-up of $Q$
along the intersection $C = V(g_1,g_2,g)$
where $g_i\in \KT{5}$ are  homogeneous
of degrees $\deg(g_1)=1$ and $\deg(g_2)=2$ such that
$C$ is smooth.
By the steps of Remark~\ref{rem:steps}
we obtain an ideal $I_2'$ generated by 
\[
g,\qquad
g_1 - T_6T_8,\qquad
g_2 - T_7T_8\quad
\in\ \KK[T_1,\ldots,T_{8}].
\]
Using Lemma~\ref{lem:serre},
the ideal $I_2'$ is prime, i.e.,
$I_2' = I_2' : T_8^\infty$.
According to the steps listed in Remark~\ref{rem:steps}, 
$\KT{8}/I_2'$ is the Cox ring of~$X_i$:
setting $T^\nu := T_1\cdots T_5$, we have
$$
\dim(I_2' +\<T_{8}\>)
\ =\ 
4
\ >\ 
\dim(I_2' + \<T_{8},T^\nu\>)
\ =\ 
3
.
$$
By assumption, there is a linear relation $T_1 + g_1' - T_6T_8$;
hence, we may remove the redundant generator~$T_1$.

{\em Case $X_{26}$:\/}
Consider the variety $X$
found in Theorem~\ref{thm:fano1} for 
$(r,-K_X^3)=(2,40)$
with the linear relations 
\begingroup
\footnotesize
\[
 f_1\ =\ T_5 - T_7 + T_{10},\qquad
 f_2\ =\ T_3 - T_4 + T_9,\qquad
 f_3\ =\ T_1 - T_4 + T_8.
\]
\endgroup
After elimination of the variables
$T_8$, $T_9$, $T_{10}$, we may assume
$X\subseteq \PP_6$.
The variety $X_{26}$ then is the blow-up of 
$X$ along a line $C\subseteq X$.
We have computed the listed Cox ring for the following 
choice of generators of~$I(C)\subseteq \KT{7}$:
\begingroup
\footnotesize
\begin{gather*}
T_{6} - T_{7},\quad  
T_{5} - T_{7},\quad 
T_{3} - T_{4},\quad 
T_{2} - T_{4},\quad  
T_{1} - T_{4}.
\end{gather*}
\endgroup
\noindent
Note that in order to compute the listed Cox ring
one has to use the additional element 
$T_{1} - T_{3} + T_{5} - T_{7}$
of $I^2:J^\infty$
with multiplicity two in 
Remark~\ref{rem:steps}.

{\em Case $X_{27}$:\/}
This is the blow-up of $\PP_3$ along a twisted cubic 
$C\subseteq \PP_3$.
Since all rational normal curves of $\PP_3$ are
projectively equivalent, it suffices to compute 
the Cox ring for $C\subseteq\PP_3$ with
\begin{eqnarray*}
    I(C)
    &=&\<
  -T_{3}^2 + T_{2}T_{4},
  T_{2}T_{3} - T_{1}T_{4},
  T_{2}^2 - T_{1}T_{3}
    \>\ \subseteq\ \KT{4}.
\end{eqnarray*}

{\em Cases $X_i$ with $i\in \{28,30\}$:\/}
We exemplarily treat the case $i=28$; $i=30$ is analogous.
Then $X_{28}$ is the blow-up of $\PP_3$ with center
a plane cubic $C\subseteq \PP_3$.
By a linear coordinate transformation
we achieve
$C = V(T_4,f)\subseteq \PP_3$
where $f\in \KT{3}$ is classically homogeneous
of degree $\deg(f)=3$.
Applying the steps of Remark~\ref{rem:steps}
yields the ideal $I_2'$ generated by 
\[
T_4 - T_5T_7,\qquad
f - T_6T_7
\ \ \in\ \ \KK[T_1,\ldots,T_7].
\]
Lemma~\ref{lem:serre} shows that
$I_2'$ is prime.
By the steps of Remark~\ref{rem:steps},
$\KT{7}/I_2'$ is the Cox ring of~$X_i$:
setting $T^\nu := T_1T_2T_3$, we have
$$
\dim(I_2' +\<T_{7}\>)
\ =\ 
4
\ >\ 
\dim(I_2' + \<T_{7},T^\nu\>)
\ =\ 
3
.
$$
The listed ring is obtained by removing the redundant
generator~$T_4$, i.e., we substitute $T_4 = T_5T_7$
and relabel the variables.

{\em Case $X_{29}$:\/}
This is the blow-up of a smooth quadric
$V(g)\subseteq \PP_4$ with center
a conic $C\subseteq V(g)$.
We may assume
that $g = h + g'\in \KT{5}$ 
where $g'\in \<T_4,T_5\>$ and
$h\in \<T_1,T_2,T_3\>$ such that
$V(h,T_4,T_5)\subseteq \PP_4$ is a smooth conic.
By the steps listed in Remark~\ref{rem:steps},
we obtain an ideal 
\begin{eqnarray}
I_2'
&=&
\<
h+g',\ 
T_4 - T_6T_9,\ 
T_5 - T_7T_9,\
h - T_8T_9
\>\notag
\\
&=&
\<
T_9(T_8 + T_9^{-1}g'(T_6T_9,T_7T_9)),\ 
T_4 - T_6T_9,\ 
T_5 - T_7T_9,\
h - T_8T_9
\>
\label{eq:X29}.
\\
&\subseteq& \KK[T_1,\ldots,T_{9}].\notag
\end{eqnarray}
Let $I_2\subseteq \KT{9}$ be the ideal
obtained from $I_2'$ by deleting $T_9$-factors
of the generators shown in~\eqref{eq:X29}, i.e.,
\begin{eqnarray*}
I_2
&=&
\<
T_8 + T_9^{-1}g'(T_6T_9,T_7T_9),\ 
T_4 - T_6T_9,\ 
T_5 - T_7T_9,\
h + g'(T_6T_9,T_7T_9)
\>
\\
&\subseteq&
\KT{9}.
\end{eqnarray*}
Then $I_2 = I_2':T_9^\infty$ since,
by Lemma~\ref{lem:serre},
 $I_2$ is a prime ideal.
By the steps listed in Remark~\ref{rem:steps},
$\KT{9}/I_2$ is the Cox ring of~$X_{29}$:
setting $T^\nu := T_1T_2T_3$, we have
$$
\dim(I_2 +\<T_{9}\>)
\ =\ 
4
\ >\ 
\dim(I_2 + \<T_{9},T^\nu\>)
\ =\ 
3
.
$$
The listed Cox ring $\Cox(X_{29})$
is obtained from  $\KT{9}/I_2$
by removing the generators
$T_4$, $T_5$ and~$T_8$.

{\em Case $X_{31}$:\/}
This is the blow-up of 
a smooth quadric $V(g)\subseteq\PP_4$ 
along a line $C\subseteq V(g)$; we may
assume $g\in \<T_1,T_2,T_3\>$ and 
choose $C=V(T_1,T_2,T_3)\subseteq V(g)$.
By the steps listed in Remark~\ref{rem:steps},
we obtain an ideal $I_2'$ generated by 
\[
g,\qquad
T_1 - T_6T_9,\qquad
T_2 - T_7T_9,\qquad
T_3 - T_8T_9\quad
 \in\ \KK[T_1,\ldots,T_{9}].
\]
The ideal $I_2'$ is prime, i.e.,
$I_2' = I_2':T_9^\infty$, see Lemma~\ref{lem:serre}.
By the steps listed in Remark~\ref{rem:steps},
$\KT{9}/I_2$ is the Cox ring of~$X_{31}$:
setting $T^\nu := T_4T_5$, we have
$$
\dim(I_2 +\<T_{9}\>)
\ =\ 
4
\ >\ 
\dim(I_2 + \<T_{9},T^\nu\>)
\ =\ 
3
.
$$
The listed ring is obtained by removing the 
redundant generators, i.e.,
$T_1 = T_6T_9$,
$T_2 = T_7T_9$
and $T_3 = T_8T_9$ are substituted into $g$
and the remaining variables are being relabeled.
\end{proof}

As an example, we show how to compute one of the $\dagger$-cases
of Theorem~\ref{thm:fano2} using our library \tt{compcox.lib}
as presented in Section~\ref{sec:example}.

\begin{example}[$X_{19}$ in Theorem~\ref{thm:fano2} and \tt{compcox.lib}]
\label{ex:X19}
Let $X$ be a smooth complete intersection 
of two quadrics $Q_i=V(g_i)\subseteq\mathbb P^5$.
Then, counted with multiplicity,
through any point $p\in X$
there are exactly four lines of $X$. 
Indeed, the union of these
lines is $Q_1\cap Q_2\cap T_pQ_1\cap T_pQ_2$.
The Fano threefold $X_{19}$ of Theorem~\ref{thm:fano2} 
is the blow-up $\pi\colon X_{19}\to X$
along one of these lines~$C$.
As seen in Theorem~\ref{thm:fano1}, we have
\[
R_1\ :=\ 
\mathcal R(X)\ =\ 
 \mathbb K[T_1,\dots,T_6]/\langle g_1,\,g_2\rangle,
\]
where all the generators have degree one.
To compute $\Cox(X_{19})$ with our library \tt{compcox.lib}, 
we choose 
\begingroup
\footnotesize
\begin{gather*}
g_1
\ =\  
T_{1}^2 + T_{2}^2 + T_{3}^2 + T_{4}^2 + T_{1}T_{5} + T_{6}T_{4} + T_{3}T_{6}\quad \in\ \KT{6}, 
\\
g_2
\ =\ 
T_{1}T_{2} + T_{2}T_{3} - T_{3}T_{4} +T_{4}T_{5} + T_{1}T_{6}
\quad \in\ \KT{6}.
\end{gather*}
\endgroup
Define the CEMDS $X$ encoded by $(P,\Sigma,G)$
similar to Example~\ref{ex:createCEMDS0} (we use 
an empty fan since we are only interested in the Cox ring, 
compare Remark~\ref{rem:nofan}):
\\[1ex]
\begingroup
\footnotesize
\tt{> intmat P[5][6] = }\\
\tt{> -1,  1,  0,  0,  0,  0,}\\
\tt{> -1,  0,  1,  0,  0,  0,}\\
\tt{> -1,  0,  0,  1,  0,  0,}\\
\tt{> -1,  0,  0,  0,  1,  0,}\\
\tt{> -1,  0,  0,  0,  0,  1;}\\
\tt{> fan Sigma = emptyFan(5);}\\
\tt{> ring S = 0,T(1..6),dp;}\\
\tt{> ideal G = }\\
\tt{> T(1)\textasciicircum 2 + T(2)\textasciicircum 2 + T(3)\textasciicircum 2 + T(4)\textasciicircum 2 + T(1)*T(5) + T(6)*T(4) + T(3)*T(6),}\\
\tt{> T(1)*T(2) + T(2)*T(3) - T(3)*T(4) +T(4)*T(5) + T(1)*T(6);}\\
\tt{> CEMDS X = createCEMDS(P, Sigma, G);}\\[1ex]
\endgroup
We choose the line $C\subseteq X$
by specifying the following generators
 $f_1,\ldots,f_4$ for the vanishing ideal 
 $I\subseteq R_1$ of~$p^{-1}(C)$:
\begingroup
\footnotesize
\begin{gather*}
  f_1\,:=\,T_4,\qquad
  f_2\,:=\,T_3,\qquad
  f_3\,:=\,T_2,\qquad
  f_4\,:=\,T_1.
%   -T_2^2 - T_3^2 - T_3T_4 - 2T_4^2 - T_1T_5 - T_3T_6 - T_4T_6.
\end{gather*}
\endgroup
We systematically produce generators 
for the Cox ring of the blow-up along $C$
by taking generators from the positive $\ZZ$-degrees 
$d_i$ of the saturated Rees algebra $R_1[I]^{\rm sat}$, 
as proposed in~\cite[Algorithm~5.6]{HaKeLa}.
Denote by $J:=\<T_1,\ldots,T_6\>\subseteq \KT{6}$
the irrelevant ideal.
Clearly,
the Rees component $(I:J^\infty)t^{-1}=It^{-1}$ of $\ZZ$-degree one
is generated by $f_1,\ldots,f_4$.
Set $d_1=\ldots=d_4=1$.
Observe that there are no new generators in
$(I^2:J^\infty)t^{-2}$, i.e., in $\ZZ$-degree two:\\[1ex]
\begingroup
\footnotesize
\tt{> ideal J = T(1),T(2),T(3),T(4),T(5),T(6); // the irrelevant ideal}\\
\tt{> ideal I1 = T(4), T(3), T(2), T(1); // the Rees component of degree one}\\
\tt{> ideal I2 = sat(I1\textasciicircum 2 + G, J)[1]; // the Rees component of degree two}\\
\tt{> std(reduce(I2, std(I1\textasciicircum 2 + G)));}\\
\tt{0}\\[1ex]
\endgroup
However, there is a generator in $(I^3:J^\infty)t^{-3}$
which is not generated by elements of smaller $\ZZ$-degree:\\[1ex]
\begingroup
\footnotesize
 \tt{> ideal I3 = sat(I1\textasciicircum 3 + G, J)[1];}\\
\tt{> std(reduce(I2, std(I1\textasciicircum 2 + G)));}\\
\tt{T(2)\textasciicircum 2 + T(3)\textasciicircum 2 + T(3)*T(4) + 2*T(4)\textasciicircum 2 + T(1)*T(5) + T(3)*T(6) + T(4)*T(6)}\\[1ex]
\endgroup
Define $f_5$ % :=  T_2^2 + T_3^2 + T_3T_4 + 2T_4^2 + T_1T_5 + T_3T_6 + T_4T_6$
as this polynomial and set $d_5 := 3$.
Starting our algorithm with\\[1ex]
\begingroup 
\footnotesize
\tt{> list L =}\\
\tt{T(4), T(3), T(2), T(1),}\\
\tt{T(2)\textasciicircum 2 + T(3)\textasciicircum 2 + T(3)*T(4) + 2*T(4)\textasciicircum 2 + T(1)*T(5) + T(3)*T(6) + T(4)*T(6);}\\
\tt{> intvec d = 1,1,1,1,3;}\\
\tt{> CEMDS X2 = blowupCEMDS(X1, L, d, 0);}
\begin{center}
\tiny
CEMDS verification successful.
\end{center}
\endgroup
\noindent
shows that $R_1[I]^{\rm sat}$ is generated in
$\ZZ$-degrees one and three, i.e., this yields 
the Cox ring of $X_{19}$. 
It is as shown in Theorem~\ref{thm:fano2}.
Note that the quadric 
of $\mathbb P^5$ defined by $f_5$
is singular with vertex
$C$ and it cuts out on $X$ the union of all the lines 
of $X$ which have non-empty intersection with $C$. 
\end{example}

\begin{remark}
The methods used for the proof of Theorem~\ref{thm:fano2} 
should expand directly to higher Picard rank. 
\end{remark}

\end{document}